\documentclass[11pt, reqno]{amsart}

\usepackage{amsmath}

\usepackage{tabu}
\usepackage{tikz}
\usepackage[margin=1.2in]{geometry}
\usepackage[matrix,arrow,curve,frame]{xy}
 \usepackage{hyperref}
\usepackage{amsthm}
\usepackage{amsfonts}
\usepackage{amssymb}
\usepackage{wasysym}
\usepackage{mathrsfs}
\usepackage{mathtools}
\usepackage{graphicx}

\definecolor{my-linkcolor}{rgb}{0.75,0,0}
\definecolor{my-citecolor}{rgb}{0.1,0.57,0}
\definecolor{my-urlcolor}{rgb}{0,0,0.75}
\hypersetup{
	colorlinks,
	linkcolor={my-linkcolor},
	citecolor={my-citecolor},
	urlcolor={my-urlcolor}
}

\title[Rationality for vertex operator algebras]{On rationality for $C_2$-cofinite vertex operator algebras}
 \author{Robert McRae}
\date{}

 \address{Yau Mathematical Sciences Center, Tsinghua University, Beijing 100084, China}
  \email{rhmcrae@tsinghua.edu.cn}


\newtheorem{mthm}{Main Theorem}
\newtheorem{thm}{Theorem}[section]
\newtheorem{cor}[thm]{Corollary}
\newtheorem{lem}[thm]{Lemma}
\newtheorem{prop}[thm]{Proposition}
\theoremstyle{definition}\newtheorem{defi}[thm]{Definition}
\theoremstyle{definition}\newtheorem{rem}[thm]{Remark}
\theoremstyle{definition}\newtheorem{exam}[thm]{Example}
\theoremstyle{definition}\newtheorem{assum}[thm]{Assumption}

\newcommand{\cE}{\mathcal{E}}
\newcommand{\cY}{\mathcal{Y}}
\newcommand{\cU}{\mathcal{U}}

\newcommand{\cA}{\mathcal{A}}
\newcommand{\cR}{\mathcal{R}}

\newcommand{\cC}{\mathcal{C}}

\newcommand{\cW}{\mathcal{W}}

\newcommand{\til}{\widetilde}

\newcommand{\CC}{\mathbb{C}}
\newcommand{\ZZ}{\mathbb{Z}}
\newcommand{\NN}{\mathbb{N}}
\newcommand{\RR}{\mathbb{R}}
\newcommand{\QQ}{\mathbb{Q}}

\newcommand{\Id}{\mathrm{Id}}

\newcommand{\tens}{\boxtimes}
\newcommand{\vac}{\mathbf{1}}

 \DeclareMathOperator{\im}{Im}

 \DeclareMathOperator{\tr}{Tr}
 
 \let\ker\relax
 \let\hom\relax
 \DeclareMathOperator{\ker}{Ker}
 \DeclareMathOperator{\hom}{Hom}
 \DeclareMathOperator{\Endo}{End}

\begin{document}
\bibliographystyle{alpha}

\numberwithin{equation}{section}

 \begin{abstract}
  Let $V$ be an $\NN$-graded, simple, self-contragredient, $C_2$-cofinite vertex operator algebra. We show that if the $S$-transformation of the character of $V$ is a linear combination of characters of $V$-modules, then the category $\cC$ of grading-restricted generalized $V$-modules is a rigid tensor category. We further show, without any assumption on the character of $V$ but assuming that $\cC$ is rigid, that $\cC$ is a factorizable finite ribbon category, that is, a not-necessarily-semisimple modular tensor category. As a consequence, we show that if the Zhu algebra of  $V$ is semisimple, then $\cC$ is semisimple and thus $V$ is rational. The proofs of these theorems use techniques and results from tensor categories together with the method of Moore-Seiberg and Huang for deriving identities of two-point genus-one correlation functions associated to $V$.
 
  We give two main applications. First, we prove the conjecture of Kac-Wakimoto and Arakawa that $C_2$-cofinite affine $W$-algebras obtained via quantum Drinfeld-Sokolov reduction of admissible-level affine vertex algebras are strongly rational. The proof uses the recent result of Arakawa and van Ekeren that such $W$-algebras have semisimple (Ramond twisted) Zhu algebras. Second, we use our rigidity results to reduce the ``coset rationality problem'' to the problem of $C_2$-cofiniteness for the coset. That is, given a vertex operator algebra inclusion $U\otimes V\hookrightarrow A$ with $A$, $U$ strongly rational and $U$, $V$ a pair of mutual commutant subalgebras in $A$, we show that $V$ is also strongly rational provided it is $C_2$-cofinite.
 \end{abstract}

\maketitle

\tableofcontents

\allowdisplaybreaks

\section{Introduction}

One of the highlights in the theory of vertex operator algebras is Huang's proof, culminating in \cite{Hu-rig-mod}, that the module category of a ``strongly rational'' vertex operator algebra is a semisimple modular tensor category. While vertex operator algebras are expected to provide a mathematically rigorous construction of two-dimensional chiral conformal quantum field theories, modular tensor categories have well-known deep connections to low-dimensional topology and physics, including topological quantum field theories, knot invariants, and topological quantum computing. Thus one can view the modular tensor category of modules for a strongly rational vertex operator algebra $V$ as containing the topological information of a conformal field theory associated to $V$.

To apply Huang's modularity theorem, one needs examples of strongly rational vertex operator algebras. Although many examples are known, it is difficult in general to show that a given vertex operator algebra is strongly rational. A vertex operator algebra $V$ is strongly rational if it is simple, positive-energy (also known as CFT-type), self-contragredient as a module for itself, $C_2$-cofinite, and rational. The first three conditions are relatively minor, while the last two are rather deep: The technical $C_2$-cofiniteness condition, introduced by Zhu \cite{Zh} to prove convergence of conformal-field-theoretic genus-one correlation functions associated to $V$, guarantees finiteness of the category of $V$-modules. Rationality, in the presence of the other conditions, is equivalent to semisimplicity of the category of $V$-modules.

The definition of strongly rational vertex operator algebra, combined with Huang's modularity theorem, suggests two questions. Let $V$ be a simple positive-energy (or more generally $\NN$-graded) self-contragredient $C_2$-cofinite vertex operator algebra; following \cite{CG}, we will call such $V$ \textit{strongly finite}. First, are there checkable criteria that guarantee the category of $V$-modules is semisimple, and thus $V$ is rational? And if on the other hand the $V$-module category is not semisimple, is there still something we can say about its structure? 

This paper will address both questions: we will show that a strongly finite vertex operator algebra is rational if its Zhu algebra \cite{Zh} is semisimple, and we will show that if the module category for a strongly finite vertex operator algebra is a rigid tensor category, then it is a factorizable finite ribbon category, also known as a non-semisimple modular tensor category. As a first application, we will obtain rationality for all $C_2$-cofinite affine $W$-algebras obtained via quantum Drinfeld-Sokolov reduction of affine vertex operator algebras at admissible levels; the Zhu algebras of these $W$-algebras were recently proved semisimple by Arakawa and van Ekeren \cite{AvE}. As a second application, we will show that the coset, or commutant, of a strongly rational subalgebra of a strongly rational vertex operator algebra is also strongly rational, provided the coset is $C_2$-cofinite. We  will discuss these results in more detail after first reviewing previous work on strongly rational and strongly finite vertex operator algebras.

\subsection{Background and previous work}

The earliest examples of strongly rational vertex operator algebras are those associated to affine Lie algebras at positive integer levels \cite{FZ}, Virasoro minimal model vertex operator algebras \cite{Wan, DMZ}, lattice vertex operator algebras \cite{Do-lattice, DLM-reg}, and the Frenkel-Lepowsky-Meurman moonshine module \cite{Do-moonshine}. The rationality proofs for these vertex operator algebras relied heavily on the specific structure of these algebras and their representations, for example of representations for affine and Virasoro Lie algebras at the relevant levels and central charges. 

More recent rationality results often use semisimplicity of the Zhu algebra \cite{Zh} to rule out non-split extensions of simple modules with the same lowest conformal weight, but then also require the detailed classification and structure of simple modules to rule out non-split extensions in general and thus prove semisimplicity of the module category. 
Vertex operator algebras which have been proved to be strongly rational in this way include the fixed-point subalgebra $V_L^+$ of a lattice vertex operator algebra $V_L$ for the involution lifted from the $-1$ isometry of the lattice $L$ \cite{DJL, Ab-VL+}, and many affine $W$-algebras at admissible levels \cite{Ar-BP, Ar-rat, AvE, Fa}. As discussed below, our results in this paper show that semisimplicity of the Zhu algebra by itself is enough to prove that such vertex operator algebras are rational, and a detailed understanding of simple modules is not necessary. As a consequence, we will show that all $C_2$-cofinite affine $W$-algebras at admissible levels are rational.

Beyond specific examples, there are general rationality results for vertex operator algebras that contain strongly rational subalgebras. Early work in this direction includes \cite{DLM-simple-current}, which showed (under some conditions) that a vertex operator algebra is rational if it contains a rational fixed-point subalgebra for a finite abelian automorphism group, and \cite{DGH}, which showed that a vertex operator algebra is rational if it contains a tensor product of simple central-charge-$\frac{1}{2}$ Virasoro vertex operator algebras. More recent work in \cite{KO, HKL, McR2}, using the theory of commutative algebras in braided tensor categories, shows that a simple positive-energy vertex operator algebra $A$ containing a strongly rational subalgebra $V$ is strongly rational if it has non-zero categorical dimension in the modular tensor category of $V$-modules. This holds in particular if $V$ is the fixed-point subalgebra for any finite automorphism group of $A$ \cite{McR} or if $V$ is the tensor product of two strongly rational subalgebras \cite{CKM2}. Recently, these results have been used to prove rationality for several classes of affine $W$-algebras \cite{CL-coset, CL-sln-trialities, CL-osym-trialities}.

A more difficult long-standing question is, when do subalgebras of strongly rational vertex operator algebras inherit strong rationality? It is widely expected that a subalgebra $V$ of a strongly rational vertex operator algebra $A$ is strongly rational if $A$ is a finite direct sum of simple $V$-modules (see for example the comment following \cite[Conjecture 1.2]{AvEM}), but this conjecture seems wide open. 
The best general result so far is Carnahan and Miyamoto's solution \cite{CM} of the ``orbifold rationality problem'' for finite solvable automorphism groups: $V$ is strongly rational if it is the fixed-point (or ``orbifold'') subalgebra of $A$ for a finite solvable automorphism group. A key step in their proof uses modular invariance properties of characters of $V$-modules to show that the tensor category of $V$-modules is rigid. We will extend these methods to prove the main theorems of this paper, and one application will be a contribution to the ``coset rationality problem,'' in which the subalgebra $V$ of the strongly rational vertex operator algebra $A$ is a tensor product of two vertex subalgebras.

%

Now we discuss results and conjectures for non-rational strongly finite vertex operator algebras. For $V$ a positive-energy $C_2$-cofinite vertex operator algebra, Huang showed \cite{Hu-C2} that the category $\cC$ of (grading-restricted generalized) $V$-modules admits the vertex algebraic braided tensor category structure of Huang-Lepowsky-Zhang \cite{HLZ1}-\cite{HLZ8}. Huang then conjectured \cite{Hu-conj} that if $V$ is in addition self-contragredient, then $\cC$ should be a rigid tensor category, and that the braiding on $\cC$ should be nondegenerate in some sense, similar to the definition of (semisimple) modular tensor category. This conjecture, that $\cC$ should be a not-necessarily-semisimple version of a modular tensor category, has also appeared in several more recent works, for example \cite{CG, GR}. 

There are actually several notions of non-semisimple modular tensor category (the first formulated by Lyubashenko \cite{Ly} for the purpose of constructing non-semisimple topological quantum field theories), but Shimizu \cite{Sh} has recently shown that all these definitions are equivalent. We shall here use the terminology of \cite{ENO} and call non-semisimple modular tensor categories \textit{factorizable finite ribbon categories}, since the categories we consider may indeed be semisimple. However, the definition we will use derives from M\"{u}ger \cite{Mu}: a finite ribbon category $\cC$ is factorizable if and only if the only transparent objects in $\cC$ are finite direct sums of the unit object. (An object $W$ in $\cC$ is transparent if the double braiding $\cR_{W,X}^2\in\Endo_\cC(W\tens X)$ is the identity for all objects $X$.)


Thus the conjecture of \cite{Hu-conj, CG, GR} may be stated as follows: If $V$ is a strongly finite vertex operator algebra, then the category $\cC$ of grading-restricted generalized $V$-modules should be a factorizable finite ribbon category. In \cite{TW, GN}, this conjecture was proved for the triplet $W$-algebras $\cW(p)$ for integers $p\geq 2$. These are subalgebras of rank-one lattice vertex operator algebras with modified conformal vectors, and they were shown to be strongly finite but non-rational in \cite{AM}. Since the first version of this paper was completed, Main Theorem \ref{mthm:factorizability} below has been used to prove the factorizability conjecture for all finite cyclic orbifolds of $\cW(p)$ \cite{CMY-typical-singlet} and for the even subalgebras of the symplectic fermion vertex operator superalgebras \cite{McR-Deligne}.

\subsection{Main results and applications}

 The factorizability conjecture of \cite{Hu-conj, CG, GR} for the tensor category $\cC$ of modules for a strongly finite vertex operator algebra includes the assertion that the tensor category $\cC$ should be rigid, with duals given by the contragredient modules of \cite{FHL}. Our first main theorem is that if $\cC$ is already known to be rigid, then $\cC$ is indeed factorizable (see Theorem \ref{thm:factorizability} below):
\begin{mthm}\label{mthm:factorizability}
 Let $V$ be an $\NN$-graded simple self-contragredient $C_2$-cofinite vertex operator algebra and let $\cC$ be the braided tensor category of grading-restricted generalized $V$-modules. If $\cC$ is rigid, then $\cC$ is a factorizable finite ribbon category.
\end{mthm}


Since the first version of this paper was completed, this theorem has been applied to the finite cyclic orbifolds of the triplet algebras $\cW(p)$ for integers $p\geq 2$ \cite{CMY-typical-singlet}. The automorphism group of $\cW(p)$ contains $\ZZ/m\ZZ$ for all $m\in\ZZ_+$, and the $\ZZ/m\ZZ$-fixed-point subalgebra $\cW(p)^{\ZZ/m\ZZ}$ is strongly finite \cite{ALM}. It is then shown in \cite{CMY-typical-singlet} that the tensor category of $\cW(p)^{\ZZ/m\ZZ}$-modules is rigid, so Main Theorem \ref{mthm:factorizability} implies that this tensor category is factorizable. Another well-known class of non-rational $C_2$-cofinite vertex operator algebras is the even subalgebras of symplectic fermion vertex operator superalgebras \cite{Ab}, which are certain extensions of tensor powers of $\cW(2)$. It has now been proved in \cite{McR-Deligne}, using the vertex operator algebra extension theory of \cite{CKM1}, that the symplectic fermion algebras have rigid module categories, and thus Main Theorem \ref{mthm:factorizability} applies to these examples as well.

Our second main theorem gives a criterion for the module category $\cC$ of a strongly finite vertex operator algebra $V$ to be rigid, in terms of the characters of $V$-modules. The \textit{character} of a $V$-module $W$ is the graded trace of the vertex operator $Y_W$, and it is a genus-one correlation function in the conformal field theory associated to $V$:
\begin{equation*}
\mathrm{ch}_W(v;\tau) = \tr_W Y_W(\cU(1)v, 1)q_\tau^{L(0)-c/24},
\end{equation*}
where $v\in V$, $q_\tau=e^{2\pi i\tau}$ for $\tau$ in the upper half plane $\mathbb{H}$, $\cU(1)$ is a certain linear operator on $V$ associated to local coordinate changes at punctures on a torus, $L(0)$ is the Virasoro algebra zero-mode, and $c$ is the Virasoro central charge of $V$. Zhu showed \cite{Zh} that if $V$ is $C_2$-cofinite, then $\mathrm{ch}_W$ converges absolutely to a function $V\times\mathbb{H}\rightarrow\CC$ linear on $V$ and holomorphic on $\mathbb{H}$, and that if $V$ is in addition rational, then the characters of irreducible $V$-modules form a basis for a representation of $SL(2,\ZZ)$. In particular, the $S$-transformation of $\mathrm{ch}_W$, given by
\begin{equation*}
 S(\mathrm{ch}_W)(v;\tau)=\mathrm{ch}_W\bigg(\left(-\frac{1}{\tau}\right)^{L(0)}v;-\frac{1}{\tau}\bigg),
\end{equation*}
is a linear combination of characters. More recently, Miyamoto showed \cite{Mi2} that if $V$ is $C_2$-cofinite but not necessarily rational, then $S(\mathrm{ch}_W)$ is in general a linear combination of both traces and pseudo-trace functions on indecomposable $V$-modules. 

It is natural to conjecture (see for instance \cite[Section 3.1.1]{CG}) that a strongly finite vertex operator algebra $V$ is rational precisely when no pseudo-traces appear in the $S$-transformations of characters of $V$-modules. In any case, if we assume no pseudo-traces in the $S$-transformation of the character of $V$ itself, then we can use the two-point genus-one correlation function methods of \cite{MS, Hu-Verlinde, CM} to prove rigidity for the category of $V$-modules (see Corollary \ref{cor:rigidity} below):
\begin{mthm}\label{mthm:rigidity}
 Let $V$ be an $\NN$-graded simple self-contragredient $C_2$-cofinite vertex operator algebra. If $S(\mathrm{ch}_V)$ is a linear combination of characters of $V$-modules, then the tensor category of grading-restricted generalized $V$-modules is rigid.
\end{mthm}

Zhu's modular invariance proof for characters in \cite{Zh} uses an associative algebra now called the Zhu algebra. For any $\NN$-graded vertex operator algebra $V$, the irreducible representations of this associative algebra $A(V)$ are in one-to-one correspondence with those of $V$. Zhu also showed that $A(V)$ is a finite-dimensional semisimple algebra if $V$ is rational; if $V$ is $C_2$-cofinite but not necessarily rational, then $A(V)$ is at least finite dimensional \cite{DLM}. Modular invariance of characters for rational $C_2$-cofinite vertex operator algebras follows directly from the semisimplicity of $A(V)$ and does not require the full strength of the rationality assumption, so we can combine Main Theorems \ref{mthm:factorizability} and \ref{mthm:rigidity} to conclude that if $V$ is a strongly finite vertex operator algebra and $A(V)$ is semisimple, then $V$-modules form a factorizable finite ribbon category. We can then address whether the category of $V$-modules is semisimple: in Subsection \ref{subsec:rat}, we show that semisimplicity of $A(V)$ and properties of projective covers in factorizable ribbon categories imply that irreducible $V$-modules are projective. Thus we get our third main theorem (see Corollary \ref{cor:rationality} below):
\begin{mthm}\label{mthm:rationality}
 Let $V$ be an $\NN$-graded simple self-contragredient $C_2$-cofinite vertex operator algebra. If the Zhu algebra $A(V)$ is semisimple, then $V$ is rational. In particular, the category of grading-restricted generalized $V$-modules is a semisimple modular tensor category, and if $V$ has positive energy, then $V$ is strongly rational.
\end{mthm}

A nice feature of Main Theorem \ref{mthm:rationality} is that it gives purely internal criteria for $V$ to be rational, even though the definition of rational vertex operator algebra involves $V$-modules. That is, one can in principal show that $V$ is $C_2$-cofinite and compute $A(V)$ without any knowledge about $V$-modules other than $V$ itself. In practice, it may be difficult to compute $A(V)$ without knowing about $V$-modules, but there is a rich class of $C_2$-cofinite vertex operator algebras whose Zhu algebras were recently shown to be semisimple:

Given a simple Lie algebra $\mathfrak{g}$, a level $k\in\CC$, and a nilpotent element $f\in\mathfrak{g}$, the simple affine $W$-algebra $\cW_k(\mathfrak{g},f)$ is the simple quotient of the quantum Drinfeld-Sokolov reduction, relative to $f$, of the universal affine vertex operator algebra $V^k(\mathfrak{g})$ at level $k$ \cite{FF, KRW}. The $W$-algebra $\cW_k(\mathfrak{g},f)$ is called exceptional \cite{KW, Ar-C2, AvE} if $k\in\QQ$ is an admissible level for $\mathfrak{g}$ and $f$ is contained in a certain nilpotent orbit associated to the denominator of $k$. Kac-Wakimoto \cite{KW} and Arakawa \cite{Ar-C2} conjectured that all exceptional $W$-algebras are rational. Indeed, Arakawa showed that all exceptional $W$-algebras are $C_2$-cofinite \cite{Ar-C2} and that exceptional $W$-algebras associated to principal nilpotent elements are rational \cite{Ar-rat}. Recently, Arakawa-van Ekeren showed \cite{AvE} that all exceptional $W$-algebras have semisimple Ramond twisted Zhu algebras (for $\NN$-graded vertex operator algebras, these are the ordinary Zhu algebras). So now Main Theorem \ref{mthm:rationality} immediately proves the Kac-Wakimoto-Arakawa rationality conjecture for all $\NN$-graded self-contragredient exceptional $W$-algebras.

Whether $\cW_k(\mathfrak{g},f)$ is $\NN$-graded and self-contragredient depends on its conformal vector, which is not unique. The standard conformal vector in $\cW_k(\mathfrak{g},f)$ is obtained from the Dynkin grading $\mathfrak{g}=\bigoplus_{j\in\frac{1}{2}\ZZ} \mathfrak{g}_j$, where $\lbrace e, h, f\rbrace$ is an $\mathfrak{sl}_2$-triple containing $f$ and $\mathfrak{g}_j$ is the $h$-eigenspace of $\mathfrak{g}$ with eigenvalue $2j$. It turns out that $\cW_k(\mathfrak{g},f)$ is always self-contragredient with respect to the Dynkin conformal vector, but it is $\frac{1}{2}\NN$-graded when the Dynkin grading of $\mathfrak{g}$ is not integral. Nevertheless, we show that $\cW_k(\mathfrak{g},f)$ is still rational as long as it is $\NN$-graded with respect to a possibly different conformal vector, and existence of a suitable conformal vector amounts to a purely Lie algebraic condition on the nilpotent element centralizer $\mathfrak{g}^f$ (see Theorem \ref{thm:W_alg_rat} below). We check this condition for all nilpotent orbits which give rise to exceptional $W$-algebras, so we conclude (see Theorem \ref{thm:ex_W_alg_rat} below):
\begin{mthm}\label{mthm:W_alg_rat}
 All exceptional affine $W$-algebras $\cW_k(\mathfrak{g},f)$ are strongly rational with respect to the Dynkin conformal vector, possibly as $\frac{1}{2}\NN$-graded vertex operator algebras.
\end{mthm}

The most-studied nilpotent elements in a simple Lie algebra are probably those in the principal, subregular, and minimal nilpotent orbits. Thus Main Theorem \ref{mthm:W_alg_rat} finishes the proof that all exceptional $W$-algebras associated to these three nilpotent orbits are strongly rational; many cases (including all principal $W$-algebras) were already known thanks to \cite{Wan, Ar-BP, Ar-rat, AvE, CL-coset, CL-sln-trialities, CL-osym-trialities, Fa}.

The proof of Main Theorem \ref{mthm:W_alg_rat} for $\frac{1}{2}\NN$-graded $W$-algebras requires a generalization of Main Theorem \ref{mthm:rationality} where, instead of assuming that $A(V)$ is semisimple, we assume that $V$ is the fixed-point subalgebra of a finite-order automorphism of a larger vertex operator algebra that has a semisimple Zhu algebra. This more general result recovers, in slightly stronger form, Carnahan and Miyamoto's theorem \cite{CM} that the fixed-point subalgebra of a strongly rational vertex operator algebra under a finite-order automorphism is strongly rational. This indicates that our results in this paper should be useful for addressing the long-standing problem of when subalgebras of rational vertex operator algebras inherit rationality. 

For example, one well-known subalgebra rationality problem is the ``coset rationality problem,'' which may be stated as follows: If $U\otimes V\hookrightarrow A$ is a vertex operator algebra inclusion with $A$, $U$ strongly rational and $U$, $V$ a pair of mutual commutant subalgebras in $A$, is the coset subalgebra $V$ also strongly rational? It is still not known in general whether $V$ is $C_2$-cofinite in this setting, but in Theorem \ref{thm:coset_rat} below, we  use Main Theorem \ref{mthm:rigidity} to show that if $V$ is in fact $C_2$-cofinite, then $V$ is also strongly rational:
\begin{mthm}\label{mthm:coset}
 Let $A$ be a strongly rational vertex operator algebra and let $U\subseteq A$ be a vertex subalgebra which is a strongly rational vertex operator algebra with respect to a conformal vector $\omega_U\in U\cap A_{(2)}$ such that $L(1)\omega_U=0$. Assume also that:
 \begin{itemize}
  \item The subalgebra $U$ is its own double commutant: $U=C_A(C_A(U))$.
  \item The commutant, or coset, $V=C_A(U)$ is $C_2$-cofinite.
 \end{itemize}
Then $V$ is also strongly rational.
\end{mthm}

We do not use Main Theorem \ref{mthm:rationality} to prove this result because precise relations between the Zhu algebras of $A$, $U$, and $V$ are not clear. However, there is an alternate method that was used to prove orbifold rationality in \cite{CM} and was generalized to prove rationality for more general subalgebras in \cite{McR2}. One first induces from the category of $U\otimes V$-modules to the category of $A$-modules to prove that $U\otimes V$ is projective as a module for itself. Then since a rigid tensor category with projective unit object is semisimple, $U\otimes V$, and then $V$ also, will be rational if the category of $U\otimes V$-modules is rigid. In fact, it is sufficient to show that the category of $V$-modules is rigid, and this follows from Main Theorem \ref{mthm:rigidity} once we show in Theorem \ref{thm:KM_app}, using the multivariable trace functions of \cite{KM} and the rationality of $A$ and $U$, that $S(\mathrm{ch}_V)$ is a linear combination of characters.

An interesting class of $C_2$-cofinite coset vertex operator algebras is certain principal affine $W$-algebras associated to simple Lie algebras $\mathfrak{g}$ in simply-laced types \cite{ACL}. Thus Main Theorem \ref{mthm:coset} provides a third rationality proof for these $W$-algebras, besides the original proof in \cite{Ar-rat} and Main Theorem \ref{mthm:W_alg_rat}.

\subsection{Methods}

Now we discuss the proofs of Main Theorems \ref{mthm:factorizability} through \ref{mthm:rationality} in reverse order. For Main Theorem \ref{mthm:rationality}, the first question is how semisimplicity of the Zhu algebra affects the representation theory of a strongly finite vertex operator algebra $V$. If $A(V)$ is semisimple, then non-split extensions between irreducible $V$-modules with the same lowest conformal weight are impossible, but a semisimple Zhu algebra does not directly rule out non-split extensions when the difference between lowest conformal weights is a non-zero integer. Thus we need a different approach to show that all $V$-modules are semisimple.

Since every irreducible $V$-module $W$ has a projective cover $p_W: P_W\twoheadrightarrow W$ by \cite{Hu-C2}, we can show that the category of $V$-modules is semisimple by proving $p_W$ is an isomorphism for all irreducible $W$. In fact, in Subsection \ref{subsec:rat}, we use semisimplicity of $A(V)$ to show that $p_W$ restricts to an isomorphism on lowest conformal weight spaces, and it then follows that $p_W$ is an isomorphism provided $P_W$ contains $W$ as a submodule (and not just as a quotient). But by \cite{ENO}, this is indeed a property of projective covers in factorizable finite ribbon categories. So Main Theorem \ref{mthm:rationality} reduces to Main Theorems \ref{mthm:factorizability} and \ref{mthm:rigidity}.

To prove Main Theorems \ref{mthm:factorizability} and \ref{mthm:rigidity}, we need to establish (under the relevant hypotheses on $V$ and its module category $\cC$), that $\cC$ is rigid and has nondegenerate braiding. To prove that $\cC$ is rigid, we need to show in particular that for any $V$-module $W$ in $\cC$, there are suitably-compatible evaluation and coevaluation maps
\begin{equation*}
 e_W: W'\tens W\rightarrow V,\qquad i_W: V\rightarrow W\tens W',
\end{equation*}
where $W'$ is the contragredient module of $W$ defined in \cite{FHL}. Since $V$ is self-contragredient, the evaluation can be obtained from the vertex operator $Y_W$ using symmetries of intertwining operators from \cite{FHL} and the intertwining operator universal property (see for example \cite{HLZ3}) that defines vertex algebraic tensor products. We can similarly obtain a homomorphism $\til{e}_W: W\tens W'\rightarrow V$. However, since $\cC$ is not assumed semisimple, it is not so easy to get maps into vertex algebraic tensor products, which means we do not have a coevaluation \textit{a priori}. In fact, the \textit{a posteriori} formula for the coevaluation shown in Remark \ref{rem:coevaluation} is rather complicated and not obviously a $V$-module homomorphism. Thus we are forced to use somewhat indirect methods to prove rigidity.

Given a $V$-module $W$ with evaluation maps $e_W: W'\tens W\rightarrow V$ and $\til{e}_W: W\tens W'\rightarrow V$, there is a natural map $(W\tens W')\tens(W\tens W')\rightarrow V$ given by evaluating both the inner and outer pairs of $W$ and $W'$. This map induces a homomorphism $\Phi_W: W\tens W'\rightarrow(W\tens W')'$, and a straightforward tensor-categorical argument shows that $W$ is rigid if and only if $\Phi_W$ is an isomorphism. Moreover, $\Phi_W$ is an isomorphism if and only if it is injective, since $W\tens W'$ and $(W\tens W')'$ are isomorphic as graded vector spaces. We then use an argument similar to the proof of \cite[Theorem 4.7]{McR} to show that $\Phi_W$ is injective if it is ``surjective enough,'' that is, $\im\Phi_W$ should contain the image of the contragredient map $(\til{e}_W)': V\cong V'\rightarrow(W\tens W')'$. In this case, the coevaluation $i_W$ is obtained by composing $(\til{e}_W)'$ with $\Phi_W^{-1}$.

To prove that $\im\,(\til{e}_W)'\subseteq\im\Phi_W$, we shift from tensor-categorical to vertex-algebraic methods. Specifically, as in the rigidity proofs of \cite{Hu-Verlinde, Hu-rig-mod, CM}, we use the method of Moore-Seiberg \cite{MS} and Huang \cite{Hu-Verlinde} to derive identities of two-point genus-one correlation functions associated to a $C_2$-cofinite vertex operator algebra $V$. These functions are (multivalued analytic extensions of) trace functions of the form
\begin{equation*}
 \tr_X \cY_1(\cU(q_{z_1})w_1, q_{z_1})\cY_2(\cU(q_{z_2})w_2,q_{z_2})q_\tau^{L(0)-c/24},
\end{equation*}
where $\cY_1$, $\cY_2$ are intertwining operators involving $V$-modules $W_1$, $W_2$, and $X$ with $w_1\in W_1$, $w_2\in W_2$; where $z_1$, $z_2$ are punctures on the torus $\CC/(\ZZ+\ZZ\tau)$ and $\cU(q_{z_1})$, $\cU(q_{z_2})$ are certain linear automorphisms of $W_1$, $W_2$ related to local coordinates at the punctures; and where $q_z=e^{2\pi i z}$ for any $z\in\CC$. For Main Theorem \ref{mthm:factorizability}, we also need pseudo-traces of products of intertwining operators \cite{Mi2, Arike, AN, Fi}. The convergence of two-point (pseudo-)trace functions in suitable regions was proved in \cite{Hu-mod_inv, Fi}, and they extend to multivalued analytic functions with singularities possible only where $z_1-z_2\in\ZZ+\ZZ\tau$. 

The Moore-Seiberg-Huang method is based on expanding two-point genus-one functions as series about the singularities $z_1-z_2=m+n\tau$, $m,n\in\ZZ$, and then using the $S$-transformation $\tau\mapsto-\frac{1}{\tau}$ to relate these series expansions and derive identities of genus-one functions. In particular, to prove the rigidity assertion of Main Theorem \ref{mthm:rigidity}, one uses the pentagon and triangle identities for tensor categories and the cyclic symmetry of traces to identify a genus-one function that involves $\til{e}_W$ with another that involves $\Phi_W$ for any $V$-module $W$ (see Theorem \ref{thm:MSH_reln_1}). Using this identity, we then realize vectors in $\mathrm{Im}\,(\til{e}_W)'$ as the images under $\Phi_W$ of rather complicated vectors defined in terms of intertwining operators involving $W$ and modules $X$ appearing in the $S$-transformation of the character of $V$ (see Remark \ref{rem:coevaluation}). As discussed above, this proves Main Theorem \ref{mthm:rigidity}.

The assumption in Main Theorem \ref{mthm:rigidity} that $S(\mathrm{ch}_V)$ involves only characters of $V$-modules, and no pseudo-traces, probably should not be necessary. That is, the module category of a strongly finite vertex operator algebra should be rigid without this condition, as is conjectured in \cite{Hu-conj} and known for the triplet vertex operator algebras and related examples. But removing this condition may require a better understanding of two-point pseudo-trace functions than is currently available: A key step in proving the identities for Main Theorem \ref{mthm:rigidity} uses cyclic symmetry to switch the order of two intertwining operators in a trace, and it is not clear if this step remains valid in the pseudo-trace generality. See Remark \ref{rem:pt_obstruction} below for a discussion of this issue.

Now to prove Main Theorem \ref{mthm:factorizability}, we need to show that the braiding on the category $\cC$ of $V$-modules is nondegenerate if $\cC$ is rigid. In particular, the only rigid simple $V$-module $W$ such that $\cR_{W,X}^2=\Id_{W\tens X}$ for all $V$-modules $X$ should be $V$ itself. Showing this uses the fact that the series expansions of two-point genus-one functions about the singularities $z_1-z_2=\pm 1$ contain information about the double braiding on $\cC$. More precisely, we prove (in Theorem \ref{thm:S_ch_W}) a logarithmic conformal field theory generalization of a well-known formula from rational conformal field theory that relates the $S$-transformation $S(\mathrm{ch}_W)$ for any $V$-module $W$ to $S(\mathrm{ch}_V)$ (which may involve pseudo-traces) and the open Hopf link endomorphisms $h_{W,X}\in\mathrm{End}_V(X)$ for $V$-modules $X$ appearing in $S(\mathrm{ch}_V)$. See \eqref{eqn:open_Hopf_link} for a graphical definition of $h_{W,X}$. Once this formula is proved, it is easy to see that if $\cR_{W,X}^2=\Id_{W\tens X}$ for all $X$ appearing in $S(\mathrm{ch}_V)$, then $\mathrm{ch}_W$ is a multiple of $\mathrm{ch}_V$. As characters of distinct irreducible $V$-modules are linearly independent, this proves Main Theorem \ref{mthm:factorizability}.

\subsection{Outline}

We now summarize the remaining contents of this paper. Section \ref{sec:prelim} recalls definitions and notation from the theories of tensor categories, vertex operator algebras, and vertex algebraic tensor categories. To prepare for $W$-algebras in Subsection \ref{subsec:W-algebras}, Subsections \ref{subsec:VOAs} and \ref{subsec:vrtx_tens_cat} also contain results on the effects (or not) that different conformal vectors in a vertex operator algebra have on that vertex operator algebra's representation category. In Section \ref{sec:contragredient}, we assume $V$ is a self-contragredient vertex operator algebra and $\cC$ is a braided tensor category of $V$-modules closed under contragredients, and we derive a criterion for $V$-modules in $\cC$ to be rigid, with duals given by contragredients. This is the most heavily tensor-categorical section of the paper, containing substantial graphical calculus computation.

Section \ref{sec:expansions} reviews convergence and series expansion results for (pseudo-)traces of compositions of intertwining operators among modules for a $C_2$-cofinite vertex operator algebra $V$, which are genus-one correlation functions in the conformal field theory associated to $V$. Most of these results have appeared before, especially in \cite{Zh, Mi2, Hu-mod_inv, Hu-Verlinde, Fi, CM}, but we give detailed proofs for some series expansion results, to emphasize which single-valued branches of multivalued analytic functions are equal on intersections of convergence regions. 

We prove Main Theorems \ref{mthm:factorizability} through \ref{mthm:rationality} in Subsections \ref{subsec:factorizability} through \ref{subsec:rat}, respectively. Actually, we prove somewhat more general theorems in Subsections \ref{subsec:rig} and \ref{subsec:rat}: we assume that $V$ is the fixed-point subalgebra under a finite-order automorphism of a larger $C_2$-cofinite vertex operator algebra $A$ that has a semisimple Zhu algebra. This generalization is needed to prove rationality for $\frac{1}{2}\NN$-graded affine $W$-algebras in Subsection \ref{subsec:W-algebras}, and it also recovers Carnahan and Miyamoto's solution \cite{CM} to the orbifold rationality problem for finite cyclic automorphism groups.

Section \ref{sec:applications} applies Main Theorems \ref{mthm:rigidity} and \ref{mthm:rationality} to prove rationality for $C_2$-cofinite affine $W$-algebras and coset vertex operator algebras. In Subsection \ref{subsec:W-algebras}, we prove the conjecture of Kac-Wakimoto and Arakawa that $C_2$-cofinite affine $W$-algebras obtained via quantum Drinfeld-Sokolov reduction of admissible-level affine vertex operator algebras are strongly rational. For $\frac{1}{2}\NN$-graded $W$-algebras, the proof requires case-by-case checking of a Lie algebraic condition on certain nilpotent orbits, which is carried out in Appendix \ref{app:nilpotent_check}. In Subsection \ref{subsec:cosets}, we reduce the coset rationality problem to the problem of $C_2$-cofiniteness for the coset.

Finally, we note for the interested reader that an expanded version of this paper is available at \texttt{arXiv:2108.01898v2}. The expanded version contains a somewhat different exposition of Main Theorem \ref{mthm:rationality} as well as detailed proofs of a few elementary results from differential equations and Lie theory that are used in this paper.

\medskip

\noindent\textbf{Acknowledgments.} I would like to thank Thomas Creutzig for discussions on $W$-algebras and Yi-Zhi Huang for discussions on differential equations and genus-one correlation functions, and I would like to thank Florencia Orosz-Hunziker, Corina Calinescu, and Christopher Sadowski for opportunities to present preliminary versions of this work. I would also like to thank Toshiyuki Abe, Tomoyuki Arakawa, Thomas Creutzig, Bin Gui, Yi-Zhi Huang, and Jinwei Yang for comments on the literature. Finally, I would like to thank the referees for constructive comments and suggestions.

\section{Preliminaries}\label{sec:prelim}

In this section, we review definitions and notation from the theories of braided tensor categories, vertex operator algebras, and vertex tensor categories.

\subsection{Braided tensor categories}

In this paper, we use (\textit{$\CC$-linear}) \textit{tensor category} to mean a ($\CC$-linear) abelian category $\cC$ with a monoidal category structure such that the tensor product bifunctor $\tens$ is bilinear on morphisms; this is less restrictive than the definition of tensor category in \cite{EGNO}. We use $l$ and $r$ to denote the left and right unit isomorphisms in $\cC$,  and $\cA$ will denote the associativity isomorphisms. A tensor category $\cC$ is \textit{braided} if it has natural braiding isomorphisms $\cR$ that satisfy the hexagon axioms. In this subsection, $\vac$ will denote the unit object of a tensor category, though in later sections the unit object will be a vertex operator algebra denoted $V$ (and $\vac$ will be the vacuum vector in $V$).

If $\cC$ is a tensor category, then we say that an object $W$ in $\cC$ is \textit{rigid} if it has a \textit{dual} $(W^*, e_W, i_W)$ where the \textit{evaluation} $e_W: W^*\tens W\rightarrow\vac$ and \textit{coevaluation} $i_W: \vac\rightarrow W\tens W^*$ are such that both compositions
\begin{equation*}
 W\xrightarrow{l_W^{-1}} \vac\tens W\xrightarrow{i_W\tens\Id_W} (W\tens W^*)\tens W\xrightarrow{\cA_{W,W^*,W}^{-1}} W\tens(W^*\tens W)\xrightarrow{\Id_W\tens e_W} W\tens\vac\xrightarrow{r_W} W
\end{equation*}
and
\begin{align*}
 W^*\xrightarrow{r_{W^*}^{-1}}  W^*\tens\vac & \xrightarrow{\Id_{W^*}\tens i_W} W^*\tens (W\tens W^*)\nonumber\\
 &\xrightarrow{\cA_{W^*,W,W^*}} (W^*\tens W)\tens W^*\xrightarrow{e_W\tens\Id_{W^*}} \vac\tens W^*\xrightarrow{l_{W^*}} W^*
\end{align*}
are identities. We call the tensor category $\cC$ \textit{rigid} if every object in $\cC$ is rigid. Note that what we have here called a dual is more properly termed a \textit{left dual}, and in general a rigid object in a tensor category should have a similarly defined right dual as well.
However, in this paper, we will only work with braided tensor categories, in which left duals are also right duals. Thus in this paper, we make no distinction between left and right duals. In a rigid tensor category, duals define a contravariant functor from $\cC$ to itself, with the dual $f^*: W_2^*\rightarrow W_1^*$ of a morphism $f: W_1\rightarrow W_2$ defined to be the composition
\begin{align*}
 W_2^*\xrightarrow{r_{W_2^*}^{-1}} W_2^*\tens\vac\xrightarrow{\Id_{W_2^*}\tens i_{W_1}} & W_2^*\tens(W_1\tens W_1^*)\xrightarrow{\Id_{W_2^*}\tens(f\tens\Id_{W_1^*})} W_2^*\tens(W_2\tens W_1^*)\nonumber\\
 & \xrightarrow{\cA_{W_2^*,W_2,W_1^*}} (W_2^*\tens W_2)\tens W_1^*\xrightarrow{e_{W_2}\tens\Id_{W_1}^*} \vac\tens W_1^*\xrightarrow{l_{W_1^*}} W_1^*.
\end{align*}
We will, however, use a different characterization of dual morphisms in Section \ref{sec:contragredient} below.

In a braided tensor category $\cC$, the \textit{monodromy} is the natural double braiding isomorphism, which we denote $\cR^2$ by slight abuse of notation, that is, $\cR_{W_1,W_2}^2$ is the composition
\begin{equation*}
 W_1\tens W_2\xrightarrow{\cR_{W_1,W_2}} W_2\tens W_1\xrightarrow{\cR_{W_2,W_1}} W_1\tens W_2
\end{equation*}
for objects $W_1$ and $W_2$ in $\cC$. We then define a \textit{twist} to be a natural isomorphism $\theta:\Id_\cC\rightarrow\Id_\cC$ such that $\theta_\vac=\Id_\vac$ and the \textit{balancing equation} holds:
\begin{equation*}
 \theta_{W_1\tens W_2}=\cR^2_{W_1,W_2}\circ(\theta_{W_1}\tens\theta_{W_2})
\end{equation*}
for objects $W_1$ and $W_2$ in $\cC$. If $\cC$ is a rigid braided tensor category with twist $\theta$, we say that $\cC$ is a \textit{ribbon category} if $\theta_{W^*}=\theta_W^*$ for all objects $W$ in $\cC$.

Now following \cite{EGNO}, we define the notion of finite tensor category as follows:
\begin{defi}
 A $\CC$-linear tensor category $\cC$ is \textit{finite} if the following conditions hold:
 \begin{enumerate}
  \item The category $\cC$ is finite as a $\CC$-linear abelian category. This means:
  \begin{enumerate}
   \item All morphism spaces in $\cC$ are finite dimensional as $\CC$-vector spaces.
   \item Every object of $\cC$ has finite length.
   \item There are finitely many equivalence classes of simple objects in $\cC$.
   \item Every simple object in $\cC$ has a projective cover.
  \end{enumerate}
\item The tensor category $\cC$ is rigid, and moreover $\Endo_\cC \vac=\CC\,\Id_\vac$.
 \end{enumerate}
 \end{defi}
 
 A finite braided ribbon category is \textit{modular} if it is semisimple and its braiding is nondegenerate. There are several equivalent ways to define nondegeneracy, and Shimizu has recently shown \cite{Sh} that these formulations, suitably interpreted, remain equivalent when the semisimplicity assumption is dropped. We shall use the terminology of \cite{ENO} and call a finite braided ribbon category that satisfies any of these nondegeneracy conditions \textit{factorizable}; for our purposes, the most convenient definition of nondegeneracy is the following:
 \begin{defi}
A \textit{factorizable finite ribbon category} is a finite braided ribbon category $\cC$ with trivial M\"{u}ger center. That is, if an object $W$ in $\cC$ satisfies $\cR_{W,X}^2=\Id_{W\tens X}$ for all objects $X$ in $\cC$, then $W\cong\vac^{\oplus n}$ for some $n\in\NN$.

A \textit{modular tensor category} is a semisimple factorizable finite ribbon category.
\end{defi}

\begin{rem}
 Factorizable finite ribbon categories could also be called non-semisimple modular tensor categories, but we avoid this terminology because most of the factorizable finite ribbon categories we shall consider will turn out to be semisimple (though they will not be so \textit{a priori}).
\end{rem}

\subsection{Vertex operator algebras and modules}\label{subsec:VOAs}

We use the definition of vertex operator algebra $(V,Y_V,\vac,\omega)$ as well as standard vertex algebraic notation from \cite{FLM, LL}, except that in general we allow $V$ to be $\frac{1}{2}\ZZ$-graded by conformal weights: $V=\bigoplus_{n\in\frac{1}{2}\ZZ} V_{(n)}$. Note, however, that our main focus will be on $\NN$-graded vertex operator algebras. For $v\in V_{(n)}$, we write $\mathrm{wt}\,v=n$. The linear map 
\begin{align*}
 Y_V: V\otimes V & \rightarrow V((x))\nonumber\\
 u\otimes v & \mapsto Y_V(u,x)v=\sum_{n\in\ZZ} u_n v\,x^{-n-1}
\end{align*}
is the \textit{vertex operator map} of $V$, $\vac\in V_{(0)}$ is the \textit{vacuum vector} of $V$, and $\omega\in V_{(2)}$ is the \textit{conformal vector}. The modes of the vertex operator
\begin{equation*}
 Y(\omega, x)=\sum_{n\in\ZZ} L(n)\,x^{-n-2}
\end{equation*}
span a representation of the Virasoro Lie algebra on $V$ with \textit{central charge} $c\in\CC$, that is,
\begin{equation*}
 [L(m),L(n)]=(m-n)L(m+n)+\frac{m^3-m}{12}\delta_{m+n,0} c
\end{equation*}
for $m,n\in\ZZ$. For $n\in\frac{1}{2}\ZZ$, $V_{(n)}$ is the $L(0)$-eigenspace of eigenvalue $n$: $L(0)v=(\mathrm{wt}\,v)v$ for homogeneous $v\in V$. We recall the notion of grading-restricted generalized $V$-module:
\begin{defi}
 Let $V$ be a vertex operator algebra. A \textit{grading-restricted generalized $V$-module} $(W,Y_W)$ is a graded vector space $W=\bigoplus_{h\in\CC} W_{[h]}$ equipped with a vertex operator
 \begin{align*}
  Y_W: V & \rightarrow (\Endo W)[[x,x^{-1}]]\nonumber\\
   v & \mapsto Y_W(v,x)=\sum_{n\in\ZZ} v_n\,x^{-n-1}
 \end{align*}
which satisfies the following axioms:
\begin{enumerate}
 \item The \textit{grading restriction conditions}: For any $h\in\CC$, $W_{[h+n]}=0$ for $n\in\ZZ$ sufficiently negative, and moreover $\dim W_{[h]}<\infty$.
 \item \textit{Lower truncation}: For all $v\in V$ and $w\in W$, $v_n w=0$ for $n\in\ZZ$ sufficiently positive, that is, $Y_W(v,x)w\in W((x))$.
 \item The \textit{vacuum property}: $Y_W(\vac,x)=\Id_W$.
 \item The \textit{Jacobi identity}: For $v_1,v_2\in V$,
 \begin{align*}
  x_0^{-1}\delta\left(\frac{x_1-x_2}{x_0}\right)Y_W(v_1,x_1)Y_W(v_2,x_2) & -x_0^{-1}\delta\left(\frac{-x_2+x_1}{x_0}\right)Y_W(v_2,x_2)Y_W(v_1,x_1)\nonumber\\
  & = x_2^{-1}\delta\left(\frac{x_1-x_0}{x_2}\right)Y_W(Y_V(v_1,x_0)v_2,x_2).
 \end{align*}
 
 \item The \textit{$L(0)$-grading condition}: For all $h\in\CC$, $W_{[h]}$ is the generalized $L(0)$-eigenspace with generalized eigenvalue $h$.
\item The \textit{$L(-1)$-derivative property}: For $v\in V$,
\begin{equation*}
 \dfrac{d}{dx}Y_W(v,x)=Y_W(L(-1)v,x).
\end{equation*}
\end{enumerate}
\end{defi}

We shall frequently abbreviate the term ``grading-restricted generalized $V$-module'' as ``$V$-module,'' though note that some definitions require $L(0)$ to act semisimply on a $V$-module. In fact, $L(0)$ necessarily acts semisimply on irreducible $V$-modules, so we shall sometimes speak of ``irreducible grading-restricted $V$-modules.''

Given a grading-restricted generalized $V$-module $W$, the \textit{contragredient} $V$-module is constructed on the graded dual $W'=\bigoplus_{h\in\CC} W_{[h]}^*$ as in \cite{FHL}:
\begin{equation}\label{eqn:contra_vrtx_op}
 \left\langle Y_{W'}(v,x)w', w\right\rangle=\big\langle w', Y_W(e^{xL(1)}(-x^{-2})^{L(0)} v,x^{-1})w\big\rangle
\end{equation}
for $v\in V$, $w'\in W'$, and $w\in W$.

A \textit{weak $V$-module} $W$ satisfies all axioms of a grading-restricted generalized $V$-module except for all properties related to the grading. That is, a weak $V$-module need not be graded by generalized $L(0)$-eigenvalues. A weak $V$-module $W$ is \textit{$\NN$-gradable}, or \textit{admissible}, if it admits an $\NN$-grading $W=\bigoplus_{n\in\NN} W(n)$ such that
\begin{equation*}
v_n W(m)\subseteq W(m+\mathrm{wt}\,v-n-1)
\end{equation*}
for all homogeneous $v\in V$, $n\in\ZZ$, and $m\in\NN$ (such an $\NN$-grading need not be unique). All grading-restricted generalized $V$-modules are $\NN$-gradable since a suitable $\NN$-grading can be constructed out of the conformal weight $\CC$-grading. For $\frac{1}{2}\ZZ$-graded vertex operator algebras, the notion of $\frac{1}{2}\NN$-gradable weak $V$-module is defined similarly.

In Subsection \ref{subsec:W-algebras}, we will study affine $W$-algebras which are $\ZZ$- or $\frac{1}{2}\ZZ$-graded with respect to more than one conformal vector. Thus suppose $\omega$ and $\til{\omega}$ are two conformal vectors in a vertex operator algebra $V$ whose zero-modes $L(0)$ and $\til{L}(0)$, respectively, both give $V$ (possibly different) $\frac{1}{2}\ZZ$-gradings. Since both conformal vectors satisfy the $L(-1)$-derivative property of a vertex operator algebra, we have
\begin{equation*}
 \mathrm{Res}_x\,Y_V(\omega,x)=L(-1)=\mathrm{Res}_x\,Y_V(\til{\omega},x)
\end{equation*}
with $L(-1)v=v_{-2}\vac$ for $v\in V$ (see for example \cite[Proposition 3.1.18]{LL}). But we also want
\begin{equation*}
 \mathrm{Res}_x\,Y_W(\omega,x)=\mathrm{Res}_x\,Y_W(\til{\omega},x)
\end{equation*}
on any weak $V$-module $W$. In fact the $L(-1)$-derivative property for $W$,
\begin{equation*}
 Y_W(v_{-2}\vac,x)=\frac{d}{dx}Y_W(v,x),
\end{equation*}
implies that indeed $\mathrm{Res}_x\,Y_W(\til{\omega}-\omega,x)=0$ if $\til{\omega}=\omega+v_{-2}\vac$ for some $v\in V$. Thus we will always assume this relation between different conformal vectors in a vertex operator algebra.

We will also want generalized $L(0)$- and $\til{L}(0)$-eigenvalue gradings of weak $V$-modules to be compatible, and for this it will be sufficient to assume
\begin{equation*}
 [L(0),\til{L}(0)]=0
\end{equation*}
on any weak $V$-module. Given $\til{\omega}=\omega+v_{-2}\vac$ for some $v\in V$, the Jacobi identity implies that 
\begin{equation*}
 [L(0),\til{L}(0)]=[v_0,L(0)]=(v_0\omega)_1.
\end{equation*}
So it is enough to assume $v_0\omega=0$. Under this assumption, we have:
\begin{lem}\label{lem:diff_omega_gradings}
 Let $V$ be a $\frac{1}{2}\ZZ$-graded vertex operator algebra with respect to either of two conformal vectors $\omega$ and $\til{\omega}=\omega+v_{-2}\vac$, where $v\in V$ satisfies $v_0\omega=0$. If $W$ is a grading-restricted generalized $V$-module with respect to $\omega$, then $W$ is the direct sum of simultaneous generalized $L(0)$- and $\til{L}(0)$-eigenspaces, where $L(0) =\mathrm{Res}_x\,xY_W(\omega,x)$ and $\til{L}(0)=\mathrm{Res}_x\,xY_W(\til{\omega},x)$.
\end{lem}
\begin{proof}
 By hypothesis, $W=\bigoplus_{h\in\CC} W_{[h]}$ where $W_{[h]}$ is the finite-dimensional generalized $L(0)$-eigenspace with generalized eigenvalue $h$. Then because our assumptions on $\til{\omega}$ imply that $\til{L}(0)$ commutes with $L(0)$, each $W_{[h]}$ is $\til{L}(0)$-stable. Thus because $W_{[h]}$ is finite dimensional, it decomposes as the direct sum of generalized $\til{L}(0)$-eigenspaces.
\end{proof}

Lemma \ref{lem:diff_omega_gradings} will show that the graded dual vector space of a grading-restricted generalized $V$-module is independent of the conformal vector. Thus in the setting of the lemma, let $W$ be a grading-restricted generalized $V$-module with respect to both conformal vectors $\omega$ and $\til{\omega}$, with generalized $L(0)$-eigenspace decomposition $W=\bigoplus_{h\in\CC} W_{[h]}$ and generalized $\til{L}(0)$-eigenspace decomposition $W=\bigoplus_{k\in\CC} W_{\lbrace k\rbrace}$. Then both graded duals $\bigoplus_{h\in\CC} W_{[h]}^*$ and $\bigoplus_{k\in\CC} W_{\lbrace k\rbrace}^*$ embed into the full dual vector space $W^*$ in the obvious way, and we have:
\begin{prop}\label{prop:graded_duals_same}
 In the setting of Lemma \ref{lem:diff_omega_gradings}, let $W$ be a grading-restricted generalized $V$-module with respect to both conformal vectors $\omega$ and $\til{\omega}$. Then the graded duals of $W$ with respect to the generalized $L(0)$- and $\til{L}(0)$-eigenvalue gradings embed as the same vector subspace of $W^*$.
\end{prop}
\begin{proof}
 By Lemma \ref{lem:diff_omega_gradings}, $W=\bigoplus_{h,k\in\CC} W_{[h]}\cap W_{\lbrace k\rbrace}$, so for any $h\in\CC$, $W_{[h]}=\bigoplus_{k\in\CC} W_{[h]}\cap W_{\lbrace k\rbrace}$. Since $\dim W_{[h]}<\infty$, the set of $k$ such that $W_{[h]}\cap W_{\lbrace k\rbrace}\neq 0$ is finite, so as subspaces of $W^*$,
 \begin{equation*}
  W_{[h]}^* =\bigoplus_{k\in\CC} (W_{[h]}\cap W_{\lbrace k\rbrace})^*\subseteq\bigoplus_{k\in\CC} W_{\lbrace k\rbrace}^*.
 \end{equation*}
So $\bigoplus_{h\in\CC} W_{[h]}^*\subseteq\bigoplus_{k\in\CC} W_{\lbrace k\rbrace}^*$ as subspaces of $W^*$, and $\bigoplus_{k\in\CC} W_{\lbrace k\rbrace}^*\subseteq\bigoplus_{h\in\CC} W_{[h]}^*$  similarly.
\end{proof}

Given a $\frac{1}{2}\ZZ$-graded vertex operator algebra $(V,Y_V,\vac,\omega)$, there is a well-known construction of a conformal vector $\til{\omega}$ satisfying the assumptions of Lemma \ref{lem:diff_omega_gradings}: we take $v$ such that
\begin{equation}\label{eqn:new_conf_cond}
 L(n)v=\delta_{n,0} v,\quad v_n v=k\delta_{n,1}\vac
\end{equation}
for all $n\geq 0$ and for some $k\in\CC$. In particular, $v$ is a Virasoro primary vector of conformal weight $1$ (for the representation of the Virasoro algebra on $V$ induced by $\omega$), and moreover $v$ generates a Heisenberg vertex subalgebra of $V$ of level $k$. The condition $v_0\omega=0$ holds by skew-symmetry for vertex operator algebras (see for example \cite[Proposition 3.1.19]{LL}):
\begin{align*}
 v_0\omega=\sum_{i\geq 0} (-1)^{i+1}\frac{L(-1)^i}{i!}\omega_i v =-L(-1)v+L(-1)L(0)v=0.
\end{align*}
Now $\til{\omega}=\omega+v_{-2}\vac$ has vertex operator $Y_V(\til{\omega},x)=\sum_{n\in\ZZ} \til{L}(n)\,x^{-n-2}$ where
\begin{equation*}
 \til{L}(n)=L(n)-(n+1)v_n
\end{equation*}
for $n\in\ZZ$, and a straightforward computation shows that the operators $\til{L}(n)$ satisfy the Virasoro algebra commutation relations with central charge $c-12k$ (where $c$ is the central charge of $V$ with respect to $\omega$). In particular, as long as $V$ is still suitably $\frac{1}{2}\ZZ$-graded by $\til{L}(0)$-eigenvalues, $\til{\omega}$ is another conformal vector in $V$.

We now discuss the Zhu algebra of a vertex operator algebra. For $V$ an $\NN$-graded vertex operator algebra, Zhu showed \cite{Zh} that irreducible $\NN$-gradable weak $V$-modules are in one-to-one correspondence with irreducible modules for a certain associative algebra $A(V)$. As a vector space, $A(V)=V/O(V)$ where
\begin{equation*}
 O(V)=\mathrm{span}\left\lbrace \mathrm{Res}_x\,x^{-2} Y_V\big((1+x)^{L(0)}u,x\big)v\,\,\big\vert\,\,u,v\in V\right\rbrace,
\end{equation*}
and the associative product on $A(V)$ is induced from the product
\begin{equation*}
 u * v = \mathrm{Res}_x\,x^{-1}Y_V\big((1+x)^{L(0)} u,x\big)v.
\end{equation*}
on $V$. The unit of $A(V)$ is $\vac+O(V)$. If $W=\bigoplus_{n\in\NN} W(n)$ is an $\NN$-gradable weak $V$-module, then \cite[Theorem 2.1.2]{Zh} shows that the \textit{top level} $W(0)$ is an $A(V)$-module with action
\begin{equation*}
 (v+O(V))\cdot w= o(v)w
\end{equation*}
for $v\in V$ and $w\in W(0)$, where $o(v)=\mathrm{Res}_x\,x^{-1}Y_W(x^{L(0)}v,x)$ is the degree-preserving component of $Y_W(v,x)$. If $v$ is homogeneous, then $o(v)=v_{\mathrm{wt}\,v-1}$. Since an indecomposable grading-restricted generalized $V$-module has a conformal weight grading of the form $W=\bigoplus_{n\in\NN} W_{[h_W+n]}$ for a unique \textit{conformal dimension} $h_W\in\CC$, the lowest conformal weight space $W_{[h_W]}$ is a finite-dimensional $A(V)$-module. If $W$ is an irreducible $V$-module, then $W_{[h_W]}$ is an irreducible $A(V)$-module.

The Zhu algebra $A(V)$ can be defined in the same way when $V$ is a $\frac{1}{2}\NN$-graded vertex operator algebra, but now there is one-to-one correspondence between irreducible $A(V)$-modules and irreducible $\NN$-gradable weak Ramond twisted $V$-modules \cite{DSK}. These are twisted $V$-modules for the automorphism of $V$ that acts as the identity on the $\NN$-graded vertex operator subalgebra and as $-1$ on the $(\NN+\frac{1}{2})$-graded subspace.

Because the definition of the Zhu algebra of $V$ depends on $L(0)$, and thus also on the choice of conformal vector in $V$, we will sometimes denote it by $A(V,\omega)$ to emphasize this dependence. However:
\begin{prop}\label{prop:Zhu_algebra_iso}
 Let $V$ be a vertex operator algebra which is $\frac{1}{2}\NN$-graded with respect to two conformal vectors $\omega$ and $\til{\omega}=\omega+v_{-2}\vac$, where $v\in V$ satisfies \eqref{eqn:new_conf_cond}. Then the Zhu algebras $A(V,\omega)$ and $A(V,\til{\omega})$ are isomorphic as unital associative algebras.
\end{prop}

See \cite[Section 5]{Ar-rat} for a proof of this proposition. The isomorphism $A(V,\til{\omega})\rightarrow A(V,\omega)$ is induced by Li's $\Delta$-operator
\begin{equation*}
\Delta(v,1) =\exp\bigg(\sum_{n=1}^\infty\frac{(-1)^{n+1}}{n}v_n\bigg)
\end{equation*}
introduced in \cite[Section 2]{Li-spectral-flow}.

A $\frac{1}{2}\ZZ$-graded vertex operator algebra $V$ is \textit{strongly rational} if it satisfies the following:
\begin{enumerate}
 \item $V$ has \textit{positive energy}, or \textit{CFT type}: $V_{(n)}=0$ for all $n<0$ and $V_{(0)}=\CC\vac$.
 \item $V$ is \textit{self-contragredient}: As $V$-modules, $V\cong V'$; equivalently, there is a nondegenerate invariant bilinear form $V\times V\rightarrow\CC$.
 \item $V$ is \textit{$C_2$-cofinite}: $\dim V/C_2(V)<\infty$, where
 \begin{equation*}
  C_2(V)=\mathrm{span}\left\lbrace u_{-2} v\,\vert\,u,v\in V\right\rbrace.
 \end{equation*}
 \item $V$ is \textit{rational}: Every $\frac{1}{2}\NN$-gradable weak $V$-module is isomorphic to a direct sum of irreducible grading-restricted $V$-modules.
\end{enumerate}
The first two conditions imply that $V$ is a simple vertex operator algebra: Any proper ideal $I\subsetneq V$ must intersect the generating subspace $V_{(0)}=\CC\vac$ trivially. From the characterization of invariant bilinear forms on $V$ in \cite[Theorem 3.1]{Li}, this means that $I$ is contained in the radical of any such bilinear form. But $V$ has a nondegenerate invariant bilinear form since $V$ is self-contragredient, so $I=0$ and $V$ is simple. 

If we drop rationality from the definition of strongly rational vertex operator algebra, we get vertex operator algebras which are called ``strongly finite'' in \cite{CG}. Since we will sometimes need to relax the positive energy condition, in this paper, we will say that a vertex operator algebra is \textit{strongly finite} if it is $\NN$-graded, simple, self-contragredient, and $C_2$-cofinite.

Zhu showed in \cite[Theorem 2.2.3]{Zh} that if $V$ is rational, then $A(V)$ is a finite-dimensional semisimple associative algebra. If $V$ is $C_2$-cofinite, then $A(V)$ is finite dimensional \cite[Proposition 3.6]{DLM} but not necessarily semisimple. Our main result in this paper will be a partial converse of the first statement: If $V$ is strongly finite and $A(V)$ is semisimple, then $V$ is rational. Thus we now review some properties of $\NN$-graded $C_2$-cofinite vertex operator algebras. The following spanning set result is \cite[Lemma 2.4]{Mi2}, and the $W=V$, $w=\mathbf{1}$ case is \cite[Proposition 8]{GN}:
\begin{lem}\label{lem:Miy_spanning_set}
 Let $V$ be an $\NN$-graded $C_2$-cofinite vertex operator algebra and let $T\subseteq V$ be a finite-dimensional graded subspace such that $V=T+C_2(V)$. Then for any weak $V$-module $W$ and any $w\in W$, the submodule generated by $w$ has the following spanning set:
 \begin{equation*}
  \langle w\rangle =\mathrm{span}\big\lbrace v^{(1)}_{n_1}\cdots v^{(k)}_{n_k} w\,\,\big\vert\,\, v^{(1)},\ldots, v^{(k)}\in T,\, n_1<\ldots < n_k\big\rbrace.
 \end{equation*}
In particular,
\begin{equation*}
 V =\mathrm{span}\big\lbrace v^{(1)}_{-n_1}\cdots v^{(k)}_{-n_k}\vac\,\,\big\vert\,\,v^{(1)},\ldots,v^{(k)}\in T,\,n_1>\cdots>n_k>0\big\rbrace.
\end{equation*}
\end{lem}

\begin{rem}\label{rem:span_set_general}
 The proof of Lemma \ref{lem:Miy_spanning_set} given in \cite{Mi2} remains valid when $V$ is a $\frac{1}{2}\NN$-graded $C_2$-cofinite vertex operator algebra, and also when the conformal weight spaces of $V$ are \textit{a priori} possibly infinite dimensional. Then the $W=V$ case of Lemma \ref{lem:Miy_spanning_set} shows that the conformal weight spaces of a $\frac{1}{2}\NN$-graded $C_2$-cofinite vertex operator algebra are actually always finite dimensional.
\end{rem}

Now as defined in \cite{CM}, $V$ is \textit{$C_2^0$-cofinite} if $\dim V/C_2^0(V)<\infty$, where 
\begin{equation*}
 C_2^0(V)=\mathrm{span}\bigg\lbrace u_{-2} v\,\,\bigg\vert\,\,u\in\bigoplus_{n>0} V_{(n)}, v\in V\bigg\rbrace.
\end{equation*}
Although $C_2^0$-cofiniteness might appear to be stronger than $C_2$-cofiniteness in general, Lemma \ref{lem:Miy_spanning_set} shows they are equivalent at least for $\NN$-graded vertex operator algebras:
\begin{lem}\label{lem:C20_cofin}
 If $V$ is an $\NN$-graded $C_2$-cofinite vertex operator algebra, then $V$ is $C_2^0$-cofinite.
\end{lem}
\begin{proof}
 Let $T\subseteq V$ be as in Lemma \ref{lem:Miy_spanning_set}, so that Lemma \ref{lem:Miy_spanning_set} shows that
 \begin{equation*}
  V=\til{T}+\sum_{n\geq 3}\mathrm{span}\lbrace u_{-n} v\,\,\vert\,\, u\in T,\,v\in V\rbrace,
 \end{equation*}
where
\begin{equation*}
 \til{T}=\mathrm{span}\big\lbrace v^{(1)}_{-n_1}\cdots v^{(k)}_{-n_k}\vac\,\,\big\vert\,\,v^{(1)},\cdots,v^{(k)}\in T,\,3> n_1>\cdots >n_k>0\big\rbrace
\end{equation*}
is finite dimensional. Now since $u_{-n}=\frac{1}{n-1}(L(-1)u)_{-n+1}$ by the $L(-1)$-derivative property and since $V$ is $\NN$-graded,
\begin{equation*}
 \mathrm{span}\lbrace u_{-n} v\,\,\vert\,\, u\in T,\,v\in V\rbrace\subseteq C_2^0(V)
\end{equation*}
for $n\geq 3$. It follows that $V$ is $C_2^0$-cofinite.
\end{proof}

As another consequence of Lemma \ref{lem:Miy_spanning_set}, as well as results from \cite{DLM, Mi2, Hu-C2}, we now show that if $V$ is $C_2$-cofinite and $\NN$-graded (with respect to some conformal vector), then the grading-restricted generalized $V$-module category depends only on $V$ as a vertex algebra, and not on the choice of conformal vector in $V$:
\begin{prop}\label{prop:C2-cofin-mod-cat}
 If $V$ is an $\NN$-graded $C_2$-cofinite vertex operator algebra, then the category of grading-restricted generalized $V$-modules equals the category of finitely-generated weak $V$-modules.
\end{prop}
\begin{proof}
Since $V$ is $C_2$-cofinite, $A(V)$ is finite dimensional by \cite[Proposition 3.6]{DLM} (see also \cite[Theorem 2.5]{Mi2} and \cite[Proposition 2.14]{Hu-C2}). Then \cite[Corollary 3.16]{Hu-C2} shows that any grading-restricted generalized $V$-module has finite length and thus is finitely generated.

Conversely, if $W$ is a finitely generated weak $V$-module, then because $V$ is $\NN$-graded and $C_2$-cofinite, \cite[Theorem 2.7(3)]{Mi2} implies that $W$ has a finite generating set consisting of generalized $L(0)$-eigenvectors. Then the spanning set of Lemma \ref{lem:Miy_spanning_set} shows that the submodule of $W$ generated by each of the finitely many homogeneous generators is a generalized $V$-module with a lower bound on the conformal weights (see for example the second assertion in \cite[Lemma 2.4]{Mi2}) and such that each conformal weight space is finite dimensional. Thus $W$ is a grading-restricted generalized $V$-module since it is a finite sum of such submodules.
\end{proof}

\subsection{Vertex tensor categories}\label{subsec:vrtx_tens_cat}

Huang showed in \cite{Hu-C2} that the category of grading-restricted generalized modules for a positive-energy $C_2$-cofinite vertex operator algebra is a finite abelian category and is also a braided tensor category as constructed in \cite{HLZ1}-\cite{HLZ8}. Then in \cite[Lemma 3.3]{CM}, Carnahan and Miyamoto checked that the results of \cite{Hu-C2} also apply to $\NN$-graded $C_2^0$-cofinite vertex operator algebras. So by Lemma \ref{lem:C20_cofin}, the full module category of an $\NN$-graded $C_2$-cofinite vertex operator algebra is a vertex algebraic braided tensor category. In this subsection, we review this braided tensor category structure.

Tensor products of modules for a vertex operator algebra $V$ are defined in terms of (possibly logarithmic) intertwining operators:
\begin{defi}
 Let $W_1$, $W_2$, and $W_3$ be grading-restricted generalized $V$-modules. An \textit{intertwining operator} of type $\binom{W_3}{W_1\,W_2}$ is a linear map
 \begin{align*}
  \cY: W_1\otimes W_2 & \rightarrow W_3[\log x]\lbrace x\rbrace\nonumber\\
  w_1\otimes w_2 & \mapsto \cY(w_1,x)w_2=\sum_{h\in\CC}\sum_{k\in\NN} (w_1)_{h,k}w_2\,x^{-h-1}(\log x)^k
 \end{align*}
satisfying the following properties:
\begin{enumerate}
 \item \textit{Lower truncation}: For any $w_1\in W_1$, $w_2\in W_2$, and $h\in\CC$, $(w_1)_{h+n,k} w_2=0$ for $n\in\NN$ sufficiently large, independently of $k$.
 \item The \textit{Jacobi identity}: For $v\in V$ and $w_1\in W_1$,
 \begin{align*}
  x_0^{-1}\delta\left(\frac{x_1-x_2}{x_0}\right)Y_{W_3}(v,x_1)\cY(w_1,x_2) & -x_0^{-1}\delta\left(\frac{-x_2+x_1}{x_0}\right)\cY(w_1,x_2)Y_{W_2}(v,x_1)\nonumber\\
  & = x_2^{-1}\delta\left(\frac{x_1-x_0}{x_2}\right)\cY(Y_{W_1}(v,x_0)w_1,x_2).
 \end{align*}
 
 \item The \textit{$L(-1)$-derivative property}: For $w_1\in W_1$,
 \begin{equation*}
  \dfrac{d}{dx}\cY(w_1,x)=\cY(L(-1)w_1,x).
 \end{equation*}
\end{enumerate}
\end{defi}

For any $V$-module $W$, the vertex operator $Y_W$ is an intertwining operator of type $\binom{W}{V\,W}$. If $f_i: W_i\rightarrow X_i$ are $V$-module homomorphisms for $i=1,2,3$ and $\cY$ is an intertwining operator of type $\binom{W_3}{X_1\,X_2}$, then $f_3\circ\cY\circ(f_1\otimes f_2)$ is an intertwining operator of type $\binom{X_3}{W_1\,W_2}$. We can also obtain new intertwining operators from old ones using the skew-symmetry and adjoint constructions of \cite{FHL, HL-old-tens-1, HLZ2}: Let $\cY$ be an intertwining operator of type $\binom{W_3}{W_1\,W_2}$ and fix $r\in\ZZ$. We then define an intertwining operator $\Omega_r(\cY)$ of type $\binom{W_3}{W_2\,W_1}$ by
\begin{equation}\label{eqn:Omega_intw_op_def}
 \Omega_r(\cY)(w_2,x)w_1=e^{xL(-1)}\cY(w_1,e^{(2r+1)\pi i} x)w_2
\end{equation}
for $w_1\in W_1$, $w_2\in W_2$. We can also define $A_r(\cY)$ of type $\binom{W_2'}{W_1\,W_3'}$ by
\begin{align}\label{eqn:Adjoint_intw_op_def}
 \left\langle A_r(\cY)(w_1,x)w_3', w_2\right\rangle =\big\langle w_3', \cY(e^{xL(1)} e^{(2r+1)\pi i L(0)}x^{-2 L(0)} w_1,x^{-1})w_2 \big\rangle
\end{align}
for $w_1\in W_1$, $w_2'\in W_2'$, and $w_3'\in W_3'$. For a $V$-module $W$, $A_r(Y_W)$ and $\Omega_r(Y_W)$ are independent of $r$, so we will denote $\Omega_r(Y_W)$ simply by $\Omega(Y_W)$ (and $A_r(Y_W)=Y_{W'}$).

\begin{defi}\label{def:tens_prod}
 Let $\cC$ be a category of grading-restricted generalized $V$-modules and let $W_1$ and $W_2$ be modules in $\cC$. A \textit{tensor product} of $W_1$ and $W_2$ in $\cC$ is a pair $(W_1\tens W_2,\cY_{W_1,W_2})$, with $W_1\tens W_2$ a $V$-module in $\cC$ and $\cY_{W_1,W_2}$ an intertwining operator of type $\binom{W_1\tens W_2}{W_1\,W_2}$, that satisfies the following universal property: If $W_3$ is a $V$-module in $\cC$ and $\cY$ is an intertwining operator of type $\binom{W_3}{W_1\,W_2}$, then there is a unique $V$-module homomorphism $f: W_1\tens W_2\rightarrow W_3$ such that $\cY=f\circ\cY_{W_1,W_2}$.
\end{defi}

\begin{rem}
 Given an arbitrary category $\cC$ of $V$-modules, it is not typically clear that tensor products of modules in $\cC$ exist.
\end{rem}

\begin{rem}
 From a geometric point of view, it is more natural to consider \textit{$P(z)$-tensor products}, defined in \cite{HLZ3} in terms of \textit{$P(z)$-intertwining maps}, which are intertwining operators with the formal variable $x$ suitably specialized to a non-zero complex number $z$. This is because the geometric formulation of conformal field theory introduced by Segal \cite{Se} and Vafa \cite{Va} requires a tensor product functor associated to each point in the moduli space of spheres with two positively-oriented punctures, one negatively-oriented puncture, and local coordinates at the punctures. In particular, $P(z)$ is the sphere with positively-oriented punctures at $0$ and $z$, negatively-oriented puncture at $\infty$, and local coordinates $w\mapsto w$, $w\mapsto w-z$, $w\mapsto 1/w$ at the punctures.
 
 A \textit{vertex tensor category} \cite{HL-vrtx-tens} is a category of $V$-modules with a tensor product functor for each conformal equivalence class of spheres with punctures and local coordinates, together with suitable natural isomorphisms that satisfy suitable coherence properties. As shown in \cite{HLZ8}, vertex tensor category structure on a category $\cC$ of $V$-modules induces a braided tensor category structure that can be completely described using the $P(z)$-tensor products for $z\in\CC^\times$. Moreover, all $P(z)$-tensor products are naturally isomorphic to each other via parallel transport isomorphisms. In particular, we may take the tensor product functor for the braided tensor category $\cC$ to be the $P(1)$-tensor product, and then tensor product modules satisfy the universal property of Definition \ref{def:tens_prod}. We describe this braided tensor category structure in more detail below.
\end{rem}

Fix a category $\cC$ of grading-restricted generalized $V$-modules that contains $V$ itself and is closed under tensor products, that is, every pair of modules $W_1$ and $W_2$ in $\cC$ has a tensor product $(W_1\tens W_2,\cY_{W_1,W_2})$ in $\cC$. It is then easy to show that tensor products define a functor $\boxtimes:\cC\times\cC\rightarrow\cC$, that $V$ is a unit object in $\cC$, and that the tensor product is commutative:
\begin{itemize}
 \item Given morphisms $f_1: W_1\rightarrow X_1$ and $f_2: W_2\rightarrow X_2$ in $\cC$, the tensor product 
 \begin{equation*}
  f_1\tens f_2: W_1\tens W_2\rightarrow X_1\tens X_2
 \end{equation*}
 is the unique morphism induced by the intertwining operator $\cY_{X_1,X_2}\circ(f_1\otimes f_2)$ of type $\binom{X_1\tens X_2}{W_1\,W_2}$ and the universal property of $(W_1\tens W_2,\cY_{W_1,W_2})$. That is,
 \begin{equation*}
  (f_1\tens f_2)\left(\cY_{W_1,W_2}(w_1,x)w_2\right)=\cY_{X_1,X_2}(f_1(w_1),x)f_2(w_2)
 \end{equation*}
for $w_1\in W_1$, $w_2\in W_2$.

\item Given a $V$-module $(W,Y_W)$ in $\cC$, the left unit isomorphism $l_W: V\tens W\rightarrow W$ is induced by the intertwining operator $Y_W$ and the universal property of $(V\tens W, \cY_{V,W})$:
\begin{equation*}
 l_W\left(\cY_{V,W}(v,x)w\right)=Y_W(v,x)w
\end{equation*}
for $v\in V$, $w\in W$. Similarly, the right unit isomorphism $r_W: W\tens V\rightarrow W$ is induced by the intertwining operator $\Omega(Y_W)$:
\begin{equation*}
 r_W\left(\cY_{W,V}(w,x)v\right) =\Omega(Y_W)(w,x)v=e^{xL(-1)}Y_W(v,-x)w
\end{equation*}
for $v\in V$, $w\in W$.

\item Given modules $W_1$ and $W_2$ in $\cC$, the braiding isomorphism 
\begin{equation*}
 \cR_{W_1,W_2}: W_1\tens W_2\rightarrow W_2\tens W_1
\end{equation*}
is induced by the intertwining operator $\Omega_0(\cY_{W_2,W_1})$ of type $\binom{W_2\tens W_1}{W_1\,W_2}$ and the universal property of $(W_1\tens W_2,\cY_{W_1,W_2})$:
\begin{equation*}
 \cR_{W_1,W_2}\left(\cY_{W_1,W_2}(w_1,x)w_2\right)=\Omega_0(\cY_{W_2,W_1})(w_1,x)w_2 =e^{xL(-1)}\cY_{W_2,W_1}(w_2,e^{\pi i} x)w_1
\end{equation*}
for $w_1\in W_1$, $w_2\in W_2$. The inverse braiding isomorphism is characterized by 
\begin{equation*}
 \cR_{W_1,W_2}^{-1}\left(\cY_{W_2,W_1}(w_2,x)w_1\right)=e^{xL(-1)}\cY_{W_1,W_2}(w_1,e^{-\pi i}x)w_2,
\end{equation*}
and the monodromy isomorphism by
\begin{equation*}
 \cR^2_{W_1,W_2}\left(\cY_{W_1,W_2}(w_1,x)w_2\right)=\cY_{W_1,W_2}(w_1,e^{2\pi i}x)w_2,
\end{equation*}
for $w_1\in W_1$, $w_2\in W_2$.

\item The category $\cC$ also has a twist given by $\theta_W=e^{2\pi i L(0)}$ for $W$ in $\cC$. The twist satisfies the balancing equation by the $L(0)$-conjugation property for intertwining operators \cite[Proposition 3.36(b)]{HLZ2}. It also satisfies $\theta_V=\Id_V$ as long as $V$ is $\ZZ$-graded.
\end{itemize}
Associativity isomorphisms in $\cC$ are much more difficult to construct and do not follow simply from the universal property of tensor products. Their description, assuming they exist, requires some analytic preparation.

Given a $V$-module $W=\bigoplus_{h\in\CC} W_{[h]}$, the \textit{algebraic completion} is $\overline{W}=\prod_{h\in\CC} W_{[h]}$; it is the vector space dual of the graded dual $W'$, since each conformal weight space $W_{[h]}$ is finite dimensional. For $h\in\CC$, $\pi_h: \overline{W}\rightarrow W_{[h]}$ denotes the canonical projection. Then if $\cY$ is a $V$-module intertwining operator of type $\binom{W_3}{W_1\,W_2}$, the substitution $x\mapsto z$ in $\cY$ for any $z\in\CC^\times$ induces a linear map
\begin{align*}
 I_\cY: W_1\otimes W_2 & \rightarrow \overline{W_3}\nonumber\\
 w_1\otimes w_2 & \mapsto \cY(w_1,z)w_2,
\end{align*}
called a $P(z)$-intertwining map (see for instance \cite{HLZ3}); since the intertwining operator $\cY$ may involve non-integral powers of $x$, as well as powers of $\log x$, we need to choose a branch of $\log z$ to make the substitution $x\mapsto z$ precise. In general, to emphasize that we are substituting formal variables using some branch of logarithm $\ell(z)$, we use the notation $\cY(w_1,e^{\ell(z)})w_2$, that is,
\begin{equation*}
 \cY(w_1,e^{\ell(z)})w_2 = \cY(w_1,x)w_2\big\vert_{x^h=e^{h\ell(z)},\,\log x=\ell(z)}.
\end{equation*}
We will most frequently use the principal branch of logarithm  on $\CC\setminus(-\infty,0]$, which we will denote by $\log z$, that is,
\begin{equation}\label{eqn:log_branch}
 \log z=\ln\vert z\vert+i\arg z
\end{equation}
where $-\pi<\arg z\leq\pi$.

Now existence of associativity isomorphisms for tensor products in a category $\cC$ of $V$-modules requires the convergence of products and iterates of intertwining operators. Given $V$-modules $W_1$, $W_2$, $W_3$, $W_4$, $M$ and intertwining operators $\cY_1$, $\cY_2$ of types $\binom{W_4}{W_1\,M}$ and $\binom{M}{W_2\,W_3}$, respectively, we say that the product of $\cY_1$ and $\cY_2$ converges if the series
\begin{equation*}
 \sum_{h\in\CC} \left\langle w_4',\cY_1(w_1,z_1)\pi_h(\cY_2(w_2,z_2)w_3)\right\rangle
\end{equation*}
converges absolutely for all $w_1\in W_1$, $w_2\in W_2$, $w_3\in W_3$, $w_4'\in W_4'$ and for all $z_1,z_2\in\CC^\times$ such that $\vert z_1\vert>\vert z_2\vert>0$, where the substitutions $x_1\mapsto z_1$, $x_2\mapsto z_2$ are accomplished using any choices for branch of logarithm. If the product of $\cY_1$ and $\cY_2$ converges, then it defines an element
\begin{equation*}
 \cY_1(w_1,z_1)\cY_2(w_2,z_2)w_3=\sum_{h\in\CC}\left\langle\cdot,\cY_1(w_1,z_1)\pi_h(\cY_2(w_2,z_2)w_3)\right\rangle\in (W_4')^*=\overline{W_4}.
\end{equation*}
Moreover, \cite[Proposition 7.20]{HLZ5} shows that for any branches of logarithm $\ell_1$ and $\ell_2$ defined on simply-connected domains of $\CC^\times$, the equality
\begin{align*}
 \langle & w_4',\cY_1(w_1,  e^{\ell_1(z_1)}) \cY_2(w_2,e^{\ell_2(z_2)})w_3\rangle\nonumber\\
 &=\sum_{h_1,h_2\in\CC}\sum_{k_1,k_2\in\NN} \left\langle w_4',(w_1)_{h_1,k_1}\left[(w_2)_{h_2,k_2} w_3\right]\right\rangle e^{(-h_1-1)\ell_1(z_1)}(\log z_1)^{k_1} e^{(-h_2-1)\ell_2(z_2)}(\log z_2)^{k_2}
\end{align*}
holds, so that in particular the product of $\cY_1$ and $\cY_2$ defines a multivalued analytic function on the region $\vert z_1\vert>\vert z_2\vert>0$.

Similarly, suppose $\cY^1$ and $\cY^2$ are intertwining operators of types $\binom{W_4}{M\,W_3}$ and $\binom{M}{W_1\,W_2}$, respectively. Then we say that the iterate of $\cY^1$ and $\cY^2$ converges if the series
\begin{equation*}
 \cY^1(\cY^2(w_1,z_0)w_2,z_2)w_3 =\sum_{h\in\CC}\left\langle\cdot,\cY^1(\pi_h(\cY^2(w_1,z_0)w_2),z_2)w_3\right\rangle\in(W_4')^*=\overline{W_4}
\end{equation*}
converges absolutely for $z_0,z_2\in\CC^\times$ such that $\vert z_2\vert>\vert z_0\vert>0$. As for products, if the iterate of $\cY^1$ and $\cY^2$ converges, then it defines a multivalued analytic function on the region $\vert z_2\vert>\vert z_0\vert>0$.

Now we can describe the associativity isomorphisms in a braided tensor category $\cC$ of $V$-modules. The basic assumption needed for their existence is convergence of products and iterates of intertwining operators among $V$-modules in $\cC$; for the remaining conditions, see the assumptions in \cite{HLZ6, HLZ8}.
\begin{itemize}
 \item Fix $r_1,r_2\in\RR$ such that $r_1>r_2>r_1-r_2>0$. Then for modules $W_1$, $W_2$, and $W_3$ in $\cC$, the associativity isomorphism
 \begin{equation*}
  \cA_{W_1,W_2,W_3}: W_1\tens(W_2\tens W_3)\rightarrow (W_1\tens W_2)\tens W_3
 \end{equation*}
is characterized by the relation
\begin{align*}
 \overline{\cA_{W_1,W_2,W_3}} & \left(\cY_{W_1,W_2\tens W_3}(w_1,e^{\ln r_1})\cY_{W_2,W_3}(w_2,e^{\ln r_2})w_3\right)\nonumber\\
 &=\cY_{W_1\tens W_2,W_3}(\cY_{W_1,W_2}(w_1,e^{\ln(r_1-r_2)})w_2,e^{\ln r_2})w_3
\end{align*}
for $w_1\in W_1$, $w_2\in W_2$, $w_3\in W_3$, where $\overline{\cA_{W_1,W_2,W_3}}$ denotes the obvious extension of $\cA_{W_1,W_2,W_3}$ to algebraic completions.
\end{itemize}
\begin{rem}
 The associativity isomorphisms are independent of the choice of $r_1,r_2\in\RR$ because the subset of $(r_1,r_2)\in(\RR^\times)^2$ defined by the condition $r_1>r_2>r_1-r_2>0$ is simply connected (see \cite[Proposition 3.32]{CKM1}).
\end{rem}

The general assumptions on the category $\cC$ of $V$-modules that were used in \cite{HLZ1}-\cite{HLZ8}  to construct braided tensor category structure are rather extensive. Nevertheless, by \cite{Hu-C2}, they hold when $V$ is $\NN$-graded $C_2$-cofinite and $\cC$ is the full category of grading-restricted generalized $V$-modules.
 Thus in this case $\cC$ is a braided tensor category with a twist, as described above. 
 
 Since we will sometimes need to consider $V$ as an $\NN$-graded vertex operator algebra with respect to two different conformal vectors $\omega$ and $\til{\omega}$, we need to consider whether the braided tensor category $\cC$ depends on the conformal vector. The operator $L(0)$ appears in the definition of the category $\cC$ itself and in the definition of graded dual vector spaces and algebraic completions, while $L(-1)$ appears in the $L(-1)$-derivative property for intertwining operators and the braiding isomorphisms. However, if $V$ is $C_2$-cofinite and $\NN$-graded with respect to both conformal vectors, Proposition \ref{prop:C2-cofin-mod-cat} shows that the category $\cC$ of grading-restricted generalized $V$-modules is independent of the conformal vector. Moreover, as noted in the previous subsection, the operator $L(-1)$ on $V$-modules in $\cC$ is also independent of the conformal vector provided $\til{\omega}=\omega+v_{-2}\vac$ for some $v\in V$. Finally, Lemma \ref{lem:diff_omega_gradings} and Proposition \ref{prop:graded_duals_same} show that algebraic completions and graded duals of $V$-modules in $\cC$ do not depend on which $L(0)$ we use, provided that $v_0\omega=0$. Thus we have:
 \begin{prop}\label{prop:ten_cat_ind_of_omega}
  Let $V$ be a $C_2$-cofinite vertex operator algebra which is $\NN$-graded with respect to either of two conformal vectors $\omega$ and $\til{\omega}=\omega+v_{-2}\vac$, where $v\in V$ satisfies $v_0\omega=0$. Then the category $\cC$ of grading-restricted generalized $V$-modules and the braided tensor category structure on $\cC$ do not depend on whether $\omega$ or $\til{\omega}$ is used as the conformal vector of $V$.
 \end{prop}
 \begin{rem}
  In the setting of Proposition \ref{prop:ten_cat_ind_of_omega}, the twist $e^{2\pi i L(0)}$ on the braided tensor category $\cC$ does depend on whether $\omega$ or $\til{\omega}$ is used as the conformal vector of $V$, as does the $V$-module structure on contragredient modules.
 \end{rem}

\section{Contragredients in vertex tensor categories}\label{sec:contragredient}

Let $V$ be a vertex operator algebra and $\cC$ a category of grading-restricted generalized $V$-modules that is closed under contragredients and admits the vertex and braided tensor category structures of \cite{HLZ1}-\cite{HLZ8}, as described in Subsection \ref{subsec:vrtx_tens_cat}. When $V$ is self-contragredient, contragredient modules give $\cC$ the structure of what is sometimes called a \textit{weakly rigid} tensor category \cite[Appendix]{KL} or an \textit{$r$-category} \cite{BD}; this is shown in \cite{CKM2} for example. When $V$ is not necessarily self-contragredient, $\cC$ still has the structure of what \cite{BD} calls a \textit{Grothendieck-Verdier category} with dualizing object $V'$; this was shown recently in \cite[Theorem 2.12]{ALSW}. But it is not usually obvious whether contragredients are duals satisfying rigidity. In this section we use tensor-categorical methods to find a criterion under which $\cC$ is in fact rigid with duals given by contragredient modules.

Now assume that $V$ is $\ZZ$-graded and self-contragredient. Since $V$ may be suitably $\ZZ$- or $\frac{1}{2}\ZZ$-graded with respect to several different conformal vectors, and since \eqref{eqn:contra_vrtx_op} shows that the contragredient of $V$ may depend on the conformal vector, our assumption here is simply that $V$ has a conformal vector with respect to which $V$ is $\ZZ$-graded and self-contragredient. We will thus use such a conformal vector, together with the accompanying contragredient modules, throughout this section.

Fix a $V$-module isomorphism $\varphi: V\rightarrow V'$. We will also use the notation of Subsection \ref{subsec:vrtx_tens_cat} in what follows. Contragredients in $\cC$ define an exact contravariant functor such that the contragredient of a homomorphism $f: W_1\rightarrow W_2$ is given by
\begin{equation*}
 \langle f'(w_2'),w_1\rangle =\langle w_2',f(w_1)\rangle
\end{equation*}
for $w_1\in W_1$, $w_2'\in W_2'$. Moreover, for any module $W$ in $\cC$, there is a natural isomorphism $\delta_W: W\rightarrow W''$ defined by
\begin{equation*}
 \langle\delta_W(w),w'\rangle =\langle w',w\rangle.
\end{equation*}
Using the skew-symmetry \eqref{eqn:Omega_intw_op_def} and adjoint \eqref{eqn:Adjoint_intw_op_def} intertwining operator constructions, for any module $W$ in $\cC$, there is an evaluation homomorphism
\begin{equation*}
 e_W: W'\tens W\rightarrow V
\end{equation*}
induced by the intertwining operator
\begin{equation*}
 \mathcal{E}_W=\varphi^{-1}\circ\Omega_0(A_0(\Omega(Y_W)))
\end{equation*}
of type $\binom{V}{W'\,W}$. A calculation shows that $\mathcal{E}_W$ is characterized by the relation
\begin{equation*}
 \big\langle(\varphi'\circ\delta_V)(v), \cE_W(w',x)w\big\rangle =\big\langle e^{-x^{-1} L(1)}w',Y_W(e^{xL(1)} v, x^{-1})e^{-xL(1)} e^{-\pi i L(0)} x^{-2L(0)} w\big\rangle
\end{equation*}
for $v\in V$, $w'\in W'$, and $w\in W$. We also define the homomorphism $\til{e}_W: W\tens W'\rightarrow V$ by
\begin{equation*}
 \til{e}_W=e_W\circ\cR_{W,W'}\circ(\theta_W\tens\Id_{W'})
\end{equation*}
and the intertwining operator $\til{\cE}_W=\til{e}_W\circ\cY_{W,W'}$ of type $\binom{V}{W\,W'}$. A calculation shows that $\til{\cE}_W=\varphi^{-1}\circ A_{-1}(\Omega(Y_W))$.

As explained in \cite[Sections 2.1 and 5.1]{CKM2}, the pair $(W', e_W)$ for any module $W$ in $\cC$ satisfies a universal property (thus making $\cC$ an $r$-category in the sense of \cite{BD}): for any module $X$ in $\cC$ and homomorphism $f: X\tens W\rightarrow V$, there is a unique map $g: X\rightarrow W'$ such that the diagram
\begin{equation*}
\xymatrixcolsep{3pc}
 \xymatrix{
 X\tens W \ar[rd]^{f} \ar[d]_{g\tens\Id_W} & \\
W'\tens W \ar[r]_(.6){e_W} & V \\
 }
\end{equation*}
commutes. This universal property gives a convenient way of characterizing the $V$-module isomorphisms $\varphi$ and $\delta_W$, as well as contragredient homomorphisms:
\begin{prop}
 In the setting of this section:
 \begin{itemize}
  \item The $V$-module isomorphism $\varphi: V\rightarrow V'$ is the unique map such that the diagram
  \begin{equation}\label{diag:phi_char}
   \xymatrixcolsep{3pc}
   \begin{matrix}
   \xymatrix{
   V\tens V \ar[rd]^{l_V=r_V} \ar[d]_{\varphi\tens\Id_V} & \\
   V'\tens V \ar[r]_(.6){e_V} & V\\}
   \end{matrix}
  \end{equation}
commutes.

  \item For any homomorphism $f: W_1\rightarrow W_2$ in $\cC$, $f': W_2'\rightarrow W_1'$ is the unique homomorphism such that the diagram
  \begin{equation}\label{diag:f'_def}
 \xymatrixcolsep{4pc}
 \begin{matrix}
 \xymatrix{
 W_2'\boxtimes W_1 \ar[r]^{\Id_{W_2'}\boxtimes f} \ar[d]_{f'\boxtimes\Id_{W_1}} & W_2'\boxtimes W_2 \ar[d]^{e_{W_2}} \\
 W_1'\boxtimes W_1 \ar[r]_(.57){e_{W_1}} & V \\
 }
 \end{matrix}
\end{equation}
commutes.

\item For any module $W$ in $\cC$, $\delta_W: W\rightarrow W''$ is the unique homomorphism such that the diagram
\begin{equation}\label{diag:delta_def}
  \xymatrixcolsep{4pc}
  \begin{matrix}
  \xymatrix{
  W\boxtimes W' \ar[d]_{\delta_W\boxtimes\Id_{W'}} \ar[rd]^{\qquad\til{e}_W=e_W\circ\cR_{W,W'}\circ(\theta_W\boxtimes\Id_{W'})} &  \\
W'' \boxtimes W' \ar[r]_(.6){e_{W'}} & V \\
  }
  \end{matrix}
 \end{equation}
commutes.
 \end{itemize}
\end{prop}
\begin{proof}
 Due to the uniqueness assertion in the universal property of contragredients, it is enough to show that all three diagrams in the proposition commute. For \eqref{diag:f'_def} and \eqref{diag:delta_def}, this was done in \cite[Lemma 4.3.4]{CMY2}. For \eqref{diag:phi_char}, we calculate
 \begin{align*}
 \big\langle (\varphi'\circ\delta_V)(u), [e_V & \circ(\varphi\tens\Id_V)] \cY_{V,V}(v_1, x)v_2\big\rangle = \big\langle (\varphi'\circ\delta_V)(u), \cE_V(\varphi(v_1),x)v_2\big\rangle\nonumber\\
 & = \big\langle e^{-x^{-1} L(1)}\varphi(v_1), Y_V(e^{x L(1)} u, x^{-1})e^{-xL(1)}(-x^{-2})^{L(0)} v_2\big\rangle\nonumber\\
 & =\big\langle \varphi(v_1), Y_V(e^{-xL(1)}(-x^{-2})^{L(0)} v_2, -x^{-1})e^{xL(1)} u\big\rangle\nonumber\\
 & = \big\langle Y_{V'}(v_2,-x)\varphi(v_1), e^{xL(1)}u\big\rangle\nonumber\\
 & =\big\langle\varphi(e^{xL(-1)}Y_V(v_2,-x)v_1), u\big\rangle\nonumber\\
 & =\big\langle \delta_V(u),\varphi(Y_V(v_1,x)v_2)\big\rangle =\big\langle(\varphi'\circ\delta_V)(u),l_V(\cY_{V,V}(v_1,x)v_2)\big\rangle,
\end{align*}
for $u, v_1,v_2\in V$, using skew-symmetry of the vertex operator $Y_V$ twice as well as the fact that $\varphi$ is a $V$-module homomorphism. Since $\cY_{V,V}$ is surjective and $\varphi'\circ\delta_V$ is an isomorphism, this shows $l_V=e_W\circ(\varphi\tens\Id_V)$.
\end{proof}

For brevity and clarity, it will be convenient to use graphical calculus in subsequent proofs in this section. To illustrate the graphical calculus conventions that we will use, we express \eqref{diag:phi_char}, \eqref{diag:f'_def}, and \eqref{diag:delta_def}, respectively, as follows:
\begin{equation*}
\begin{matrix}
    \begin{tikzpicture}[scale = 1, baseline = {(current bounding box.center)}, line width=0.75pt] 
    \draw[dashed] (0,0) -- (0,1.25) .. controls (0,2.15) and (1.5,2.15) .. (1.5,1.25) -- (1.5,0);
    \draw[dashed] (.75,1.75) -- (.75,2.5);
    \node at (0,-.25) {$V$};
    \node at (1.5,-.25) {$V$};
    \node at (.75,2.75) {$V$};
    \node at (0,.75) [draw,minimum width=20pt,minimum height=10pt,thick, fill=white] {$\varphi$};
    \node at (.75,1.85) [draw,minimum width=20pt,minimum height=10pt,thick, fill=white] {$e_V$};
    \end{tikzpicture}
   \end{matrix} =
   \begin{matrix}
    \begin{tikzpicture}[scale = 1, baseline = {(current bounding box.center)}, line width=0.75pt] 
    \draw[dashed] (0,0) -- (.75,1) -- (1.5,0);
    \draw[dashed] (.75,1) -- (.75,2);
    \node at (0,-.25) {$V$};
    \node at (1.5,-.25) {$V$};
    \node at (.75,2.25) {$V$};
    \end{tikzpicture}
   \end{matrix},\qquad\qquad
 \begin{matrix}
    \begin{tikzpicture}[scale = 1, baseline = {(current bounding box.center)}, line width=0.75pt] 
    \draw[white, double=black, line width = 3pt ] (0,0) -- (0,1.25) .. controls (0,2.15) and (1.5,2.15) .. (1.5,1.25) -- (1.5,0);
    \draw[dashed] (.75,1.75) -- (.75,2.5);
    \node at (0,-.25) {$W_2'$};
    \node at (1.5,-.25) {$W_1$};
    \node at (.75,2.75) {$V$};
    \node at (0,.75) [draw,minimum width=20pt,minimum height=10pt,thick, fill=white] {$f'$};
    \node at (.75,1.75) [draw,minimum width=20pt,minimum height=10pt,thick, fill=white] {$e_{W_1}$};
    \end{tikzpicture}
   \end{matrix} =
  \begin{matrix}
    \begin{tikzpicture}[scale = 1, baseline = {(current bounding box.center)}, line width=0.75pt] 
    \draw[white, double=black, line width = 3pt ] (0,0) -- (0,1.25) .. controls (0,2.15) and (1.5,2.15) .. (1.5,1.25) -- (1.5,0);
    \draw[dashed] (.75,1.75) -- (.75,2.5);
    \node at (0,-.25) {$W_2'$};
    \node at (1.5,-.25) {$W_1$};
    \node at (.75,2.75) {$V$};
    \node at (1.5,.75) [draw,minimum width=20pt,minimum height=10pt,thick, fill=white] {$f$};
    \node at (.75,1.75) [draw,minimum width=20pt,minimum height=10pt,thick, fill=white] {$e_{W_2}$};
    \end{tikzpicture}
   \end{matrix},
\end{equation*}
and
\begin{equation*}
 \begin{matrix}
    \begin{tikzpicture}[scale = 1, baseline = {(current bounding box.center)}, line width=0.75pt] 
    \draw[white, double=black, line width = 3pt ] (0,0) -- (0,1.25) .. controls (0,2.15) and (1.5,2.15) .. (1.5,1.25) -- (1.5,0);
    \draw[dashed] (.75,1.75) -- (.75,2.5);
    \node at (0,-.25) {$W$};
    \node at (1.5,-.25) {$W'$};
    \node at (.75,2.75) {$V$};
    \node at (0,.75) [draw,minimum width=20pt,minimum height=10pt,thick, fill=white] {$\delta_W$};
    \node at (.75,1.75) [draw,minimum width=20pt,minimum height=10pt,thick, fill=white] {$e_{W'}$};
    \end{tikzpicture}
   \end{matrix} =
   \begin{matrix}
    \begin{tikzpicture}[scale = 1, baseline = {(current bounding box.center)}, line width=0.75pt] 
    \draw[white, double=black, line width = 3pt ] (0,0) -- (0,.5) .. controls (0,1.4) and (1.5,1.4) .. (1.5,.5) -- (1.5,0);
    \draw[dashed] (.75,1) -- (.75,2);
    \node at (0,-.25) {$W$};
    \node at (1.5,-.25) {$W'$};
    \node at (.75,2.25) {$V$};
    \node at (.75,1) [draw,minimum width=20pt,minimum height=10pt,thick, fill=white] {$\til{e}_{W}$};
    \end{tikzpicture}
   \end{matrix}
   =
  \begin{matrix}
    \begin{tikzpicture}[scale = 1, baseline = {(current bounding box.center)}, line width=0.75pt] 
    \draw (1.5,0) -- (1.5,1.15) .. controls (1.5,1.65) and (0,1.4) .. (0,1.9) -- (0,2);
    \draw[white, double=black, line width = 3pt ] (0,0) -- (0,1.15) .. controls (0,1.65) and (1.5,1.4) .. (1.5,1.9) -- (1.5,2);
    \draw (0,2) .. controls (0,2.85) and (1.5,2.85) .. (1.5,2);
    \draw[dashed] (.75,2.55) -- (.75,3.25);
    \node at (0,-.25) {$W$};
    \node at (1.5,-.25) {$W'$};
    \node at (.75,3.5) {$V$};
    \node at (0,.65) [draw,minimum width=20pt,minimum height=10pt,thick, fill=white] {$\theta_W$};
    \node at (.75,2.5) [draw,minimum width=20pt,minimum height=10pt,thick, fill=white] {$e_{W}$};
    \end{tikzpicture}
   \end{matrix}.
\end{equation*}
It is straightforward, although perhaps tedious, to convert all graphical proofs presented below into more detailed arguments through appropriate use of the triangle, pentagon, and hexagon axioms and their consequences.

Now for any module $W$ in $\cC$, we define
\begin{equation*}
 \Phi_W: W\tens W'\rightarrow(W\tens W')'
\end{equation*}
to be the unique homomorphism such that the diagram
\begin{align*}
\xymatrixcolsep{.6pc}
\xymatrix{
& ((W\tens W')\tens W)\tens W'  \ar[rrrrr]^(.51){\cA_{W,W',W}^{-1}\tens\Id_{W'}} &&&&& (W\tens(W'\tens W))\tens W' \ar[d]^{(\Id_W\tens e_W)\tens\Id_{W'}} \\
 (W\tens W')\tens(W\tens W') \ar[ru]^(.45){\cA_{W\tens W',W,W'}} \ar[d]^{\Phi_W\tens\Id_{W\tens W'}} & &&&&& (W\tens V)\tens W' \ar[d]^{r_W\tens\Id_{W'}} \\
 (W\tens W')'\tens(W\tens W') \ar[r]_(.7){e_{W\tens W'}} & V &&&&& \ar[lllll]^{\til{e}_W} W\tens W'\\
 }
\end{align*}
commutes. Graphically,
\begin{equation*}
 \begin{matrix}
    \begin{tikzpicture}[scale = 1, baseline = {(current bounding box.center)}, line width=0.75pt]
    \draw (0,0) -- (0,.5) .. controls (0,.8) .. (.5,.8) .. controls (1,.8) .. (1,.5) -- (1,0);
    \draw[dashed] (1.5, 2.25) -- (1.5,3);
    \draw (.5,1) -- (.5,1.6) .. controls (.5,2.6) and (2.5,2.6) .. (2.5,1.6) -- (2.5, 1);
    \draw (2,0) -- (2,.5) -- (2.5,1) -- (3,.5) -- (3,0);
    \node at (0,-.25) {$W$};
    \node at (1,-.25) {$W'$};
    \node at (2,-.25) {$W$};
    \node at (3,-.25) {$W'$};
    \node at (1.5,3.25) {$V$};
    \node at (.5,1) [draw,minimum width=20pt,minimum height=10pt,thick, fill=white] {$\Phi_W$};
    \node at (1.5,2.25) [draw,minimum width=20pt,minimum height=10pt,thick, fill=white] {$e_{W\boxtimes W'}$};
    \end{tikzpicture}
   \end{matrix} =
   \begin{matrix}
    \begin{tikzpicture}[scale = 1, baseline = {(current bounding box.center)}, line width=0.75pt]
     \draw[white, double=black, line width = 3pt ] (1,0) -- (1,.5) .. controls (1,1.4) and (2,1.4) .. (2,0.5) -- (2,0);                      
    \draw[white, double=black, line width = 3pt ] (0,0) -- (0,1.5) .. controls (0,2.5) and (3,2.5) .. (3,1.5) -- (3,0);
    \draw[dashed] (1.5,1) .. controls (1.5,1.4) .. (0,1.5);
    \draw[dashed] (1.5,2.25) -- (1.5,3);
    \node at (0,-.25) {$W$};
    \node at (1,-.25) {$W'$};
    \node at (2,-.25) {$W$};
    \node at (3,-.25) {$W'$};
    \node at (1.5,3.25) {$V$};
    \node at (1.5,1) [draw,minimum height=10pt,thick, fill=white] {$e_W$};
    \node at (1.5,2.25) [draw,minimum width=20pt,minimum height=10pt,thick, fill=white] {$\widetilde{e}_W$};
    \end{tikzpicture}
   \end{matrix}.
\end{equation*}
The following result may be found, for example, in \cite[Proposition 4.4.5]{CMY2}, but here we give a more streamlined proof:
\begin{prop}\label{prop:Phi_iso_rigid}
 If $\Phi_W$ is an isomorphism, then $W$ is rigid with dual $W'$.
\end{prop}
\begin{proof}
 Assuming $\Phi_W$ is an isomorphism, we first define the coevaluation
 \begin{equation*}
  i_W: V\xrightarrow{\varphi} V'\xrightarrow{(\til{e}_W)'} (W\tens W')'\xrightarrow{\Phi_W^{-1}} W\tens W'.
 \end{equation*}
To show that $(W',e_W,i_W)$ is a dual of $W$ in $\cC$, \cite[Lemma 4.2.1]{CMY3} implies it is enough to show that the rigidity composition $\mathfrak{R}_W$ given by
\begin{equation*}
 W\xrightarrow{l_W^{-1}} V\tens W\xrightarrow{i_W\tens\Id_W} (W\tens W')\tens W\xrightarrow{\cA_{W,W',W}^{-1}} W\tens(W'\tens W)\xrightarrow{\Id_W\tens e_W} W\tens V\xrightarrow{r_W} W
\end{equation*}
is the identity on $W$. Then since \eqref{diag:delta_def} shows that $(W,\til{e}_W)$ is a contragredient of $W'$, the universal property of contragredients implies that it is enough to show
\begin{equation*}
 \til{e}_W\circ(\mathfrak{R}_W\tens\Id_{W'}) =\til{e}_W.
\end{equation*}
This can be proved graphically as follows:
\begin{align*}
 \til{e}_W\circ(\mathfrak{R}_W\tens\Id_{W'}) & =\begin{matrix}
    \begin{tikzpicture}[scale = 1, baseline = {(current bounding box.center)}, line width=0.75pt]
     \draw[white, double=black, line width = 3pt ] (0,2.5) .. controls (0,1.6) and (1,1.6) .. (1,2.5) .. controls (1,3.4) and (2,3.4) .. (2,2.5) -- (2,0);                      
    \draw[white, double=black, line width = 3pt ] (0,2.5) -- (0,3.5) .. controls (0,4.5) and (3,4.5) .. (3,3.5) -- (3,0);
    \draw (.5,1) -- (.5,2);
    \draw[dashed] (2,.25) .. controls (.75,.25) .. (.5,1);
    \draw[dashed] (1.5,3) .. controls (1.5,3.4) .. (0,3.5);
    \draw[dashed] (1.5,4.25) -- (1.5,5);
    \node at (2,-.25) {$W$};
    \node at (3,-.25) {$W'$};
    \node at (1.5,5.25) {$V$};
    \node at (1.5,3) [draw,minimum height=10pt,thick, fill=white] {$e_W$};
    \node at (1.5,4.25) [draw,minimum width=20pt,minimum height=10pt,thick, fill=white] {$\widetilde{e}_W$};
    \node at (.5,1) [draw,minimum height=10pt,thick, fill=white] {$(\til{e}_W)'\circ\varphi$};
    \node at (.5,2) [draw,minimum height=10pt,thick, fill=white] {$\Phi_W^{-1}$};
    \end{tikzpicture}
   \end{matrix} =
   \begin{matrix}
    \begin{tikzpicture}[scale = 1, baseline = {(current bounding box.center)}, line width=0.75pt]                     
    \draw[white, double=black, line width = 3pt ] (.5,2) -- (.5,2.5) .. controls (.5,3.85) and (2.5,3.85) .. (2.5,2.5) -- (2.5,1);
    \draw[dashed] (2,.25) .. controls (.75,.25) .. (.5,1) -- (.5,2);
    \draw (2,0) -- (2,.5) -- (2.5,1) -- (3,.5) -- (3,0);
    \draw[dashed] (1.5,3.25) -- (1.5,4);
    \node at (2,-.25) {$W$};
    \node at (3,-.25) {$W'$};
    \node at (1.5,4.25) {$V$};
    \node at (1.5,3.35) [draw,minimum width=20pt,minimum height=10pt,thick, fill=white] {$e_{W\tens W'}$};
    \node at (.5,1) [draw,minimum height=10pt,thick, fill=white] {$\varphi$};
    \node at (.5,2) [draw,minimum height=10pt,thick, fill=white] {$(\til{e}_W)'$};
    \end{tikzpicture}
   \end{matrix} \nonumber\\
   &= 
   \begin{matrix}
    \begin{tikzpicture}[scale = 1, baseline = {(current bounding box.center)}, line width=0.75pt]                     
    \draw[dashed] (.5,1.25) -- (.5,1.5) .. controls (.5,2.85) and (2.5,2.85) .. (2.5,1.5) -- (2.5,1.25);
    \draw[dashed] (2,.25) .. controls (.75,.25) .. (.5,1) -- (.5,1.25);
    \draw (2,0) -- (2,.5) .. controls (2,1.8) and (3,1.8) .. (3,.5) -- (3,0);
    \draw[dashed] (1.5,2.5) -- (1.5,3.25);
    \node at (2,-.25) {$W$};
    \node at (3,-.25) {$W'$};
    \node at (1.5,3.5) {$V$};
    \node at (1.5,2.5) [draw,minimum width=20pt,minimum height=10pt,thick, fill=white] {$e_{V}$};
    \node at (.5,1.25) [draw,minimum height=10pt,thick, fill=white] {$\varphi$};
    \node at (2.5,1.25) [draw,minimum height=10pt,thick, fill=white] {$\til{e}_W$};
    \end{tikzpicture}
   \end{matrix} =
   \begin{matrix}
    \begin{tikzpicture}[scale = 1, baseline = {(current bounding box.center)}, line width=0.75pt]                     
    \draw[dashed] (2,.25) .. controls (1,.25) .. (1,1.25) .. controls (1,1.8) .. (1.75,2.25) .. controls (2.5,1.8) .. (2.5,1.25);
    \draw (2,0) -- (2,.5) .. controls (2,1.8) and (3,1.8) .. (3,.5) -- (3,0);
    \draw[dashed] (1.75,2.25) -- (1.75,3);
    \node at (2,-.25) {$W$};
    \node at (3,-.25) {$W'$};
    \node at (1.75,3.25) {$V$};
    \node at (2.5,1.25) [draw,minimum height=10pt,thick, fill=white] {$\til{e}_W$};
    \end{tikzpicture}
   \end{matrix} =
   \begin{matrix}
    \begin{tikzpicture}[scale = 1, baseline = {(current bounding box.center)}, line width=0.75pt] 
    \draw[white, double=black, line width = 3pt ] (0,0) -- (0,.5) .. controls (0,1.4) and (1.5,1.4) .. (1.5,.5) -- (1.5,0);
    \draw[dashed] (.75,1) -- (.75,2);
    \node at (0,-.25) {$W$};
    \node at (1.5,-.25) {$W'$};
    \node at (.75,2.25) {$V$};
    \node at (.75,1) [draw,minimum width=20pt,minimum height=10pt,thick, fill=white] {$\til{e}_{W}$};
    \end{tikzpicture}
   \end{matrix},
\end{align*}
where the first equality is the definition of $\mathfrak{R}_W$, the second is the definition of $\Phi_W$, the third follows from \eqref{diag:f'_def}, the fourth from \eqref{diag:phi_char}, and the last equality comes from naturality of the left unit isomorphisms.
\end{proof}

To find a sufficient condition for $\Phi_W$ to be an isomorphism, and thus for $W$ to be rigid, we define two more homomorphisms using the universal property of contragredients. First, let 
\begin{equation*}
 \Psi_W: W'\tens(W\tens W')'\rightarrow W'
\end{equation*}
be the unique morphism such that the diagram
\begin{align*}
 \xymatrixcolsep{.5pc}
\xymatrix{
& (W'\tens(W\tens W')')\tens W  \ar[rrrrr]^(.51){\cR_{(W\tens W')',W'}^{-1}\tens\Id_{W}} &&&&& ((W\tens W')'\tens W')\tens W \ar[d]^{\cA^{-1}_{(W\tens W')',W',W}} \\
 (W'\tens(W\tens W')')\tens W \ar[ru]^(.45){(\theta_{W'}^{-1}\tens\Id_{(W\tens W')'})\tens\Id_{W}\qquad} \ar[d]^{\Psi_W\tens\Id_{W\tens W'}} & &&&&& (W\tens W')'\tens(W'\tens W) \ar[d]^{\Id_{(W\tens W')'}\tens\cR_{W,W'}^{-1}} \\
 W'\tens W \ar[r]_(.52){e_{W}} & V &&&&& \ar[lllll]^(.6){e_{W\tens W'}} (W\tens W')'\tens(W\tens W')\\
 }
\end{align*}
commutes. Graphically,
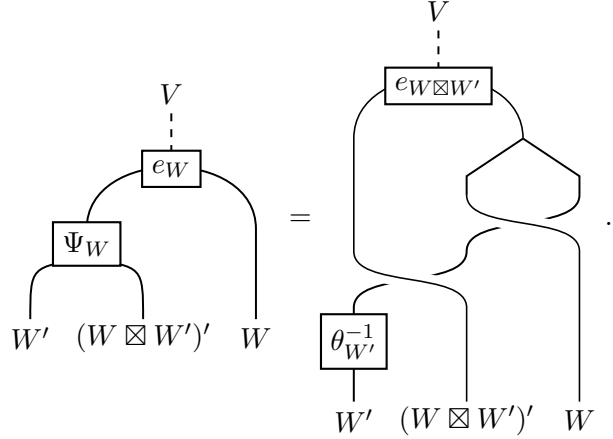
\begin{equation*}
  \begin{matrix}
     \begin{tikzpicture}[scale = 1, baseline = {(current bounding box.center)}, line width=0.75pt]
     \draw (1.5,1) .. controls (1.5,1.7) .. (.75,1.8);
     \draw (0,1) .. controls (0,1.7) .. (.75,1.8);
     \draw (.75,2.2) .. controls (.75,3.3) and (3,3.3) .. (3,2.2) -- (3,1);
     \draw[dashed] (1.875,3) -- (1.875,3.8);
    \node at (0,.75) {$W'$};
    \node at (3,.75) {$W$};
    \node at (1.875,4) {$V$};
    \node at (1.5,.75) {$(W\tens W')'$};
    \node at (.75,2) [draw,minimum width=20pt,minimum height=10pt,thick, fill=white] {$\Psi_W$};
    \node at (1.875,3) [draw,minimum width=20pt,minimum height=10pt,thick, fill=white] {$e_W$};   
     \end{tikzpicture}
   \end{matrix}   = 
   \begin{matrix}
    \begin{tikzpicture}[scale = 1, baseline = {(current bounding box.center)}, line width=0.75pt]
    \draw (0,0) -- (0,1);
     \draw (0,1) -- (0,1.25) .. controls (0, 1.75) and (1.5, 1.5) .. (1.5, 2) .. controls (1.5, 2.5) and (3, 2.25) .. (3,2.75);
     \draw[white, double=black, line width = 3pt ] (1.5,0) -- (1.5,1.25) .. controls (1.5, 1.75) and (0,1.5) .. (0,2) -- (0,3.5) .. controls (0,4.6) and (2.25,4.6) .. (2.25,3.5); 
    \draw[white, double=black, line width = 3pt ] (3,0) -- (3,2) .. controls (3,2.5) and (1.5,2.25) .. (1.5,2.75);
    \draw (1.5,2.75) -- (1.5, 3) -- (2.25, 3.5);
    \draw (2.25,3.5) -- (3,3) -- (3,2.75);
    \draw[dashed] (1.125, 4.2) -- (1.125,5);
    \node at (0,-.25) {$W'$};
    \node at (3,-.25) {$W$};
    \node at (1.125,5.2) {$V$};
    \node at (1.5,-.25) {$(W\tens W')'$};
    \node at (0,.8) [draw,minimum width=20pt,minimum height=10pt,thick, fill=white] {$\theta_{W'}^{-1}$};
    \node at (1.125,4.2) [draw,minimum width=20pt,minimum height=10pt,thick, fill=white] {$e_{W\boxtimes W'}$};
    \end{tikzpicture}
   \end{matrix} .
\end{equation*}
Secondly, we define
\begin{equation*}
 \til{\Psi}_W: (W\tens W')'\tens W\rightarrow W
\end{equation*}
to be the unique morphism such that the diagram
\begin{align*}
 \xymatrixcolsep{5pc}
 \xymatrix{
 ((W\tens W')'\tens W)\tens W' \ar[d]_{\til{\Psi}_W\tens\Id_{W'}} \ar[r]^{\cA_{(W\tens W')',W,W'}^{-1}} & (W\tens W')'\tens(W\tens W') \ar[d]^{e_{W\tens W'}}\\
 W\tens W' \ar[r]_{\til{e}_{W}} & V\\
 }
\end{align*}
commutes (recall from \eqref{diag:delta_def} that $(W,\til{e}_W)$ is a contragredient of $W'$ in $\cC$).
\begin{lem}\label{lem:PsiPhi}
 For any $V$-module $W$ in $\cC$,
 \begin{enumerate}
  \item  $\Psi_{W}\circ(\Id_{W'}\tens\Phi_W) = l_{W'}\circ(e_W\tens\Id_{W'})\circ\cA_{W',W,W'}$ as morphisms $W'\tens(W\tens W')\rightarrow W'$.
  \item $\til{\Psi}_W\circ(\Phi_W\tens\Id_W) = r_W\circ(\Id_W\tens e_W)\circ\cA_{W,W',W}^{-1}$ as morphisms $(W\tens W')\tens W\rightarrow W$.
 \end{enumerate}
\end{lem}

\begin{proof}
 To prove the first assertion, the universal property of the contragredient $(W',e_W)$ of $W$ implies that it is enough to show 
 \begin{equation}\label{eqn:PsiPhi_reln_1}
  e_W\circ([\Psi_W\circ(\Id_{W'}\tens\Phi_W)]\tens\Id_W)=e_W\circ([l_{W'}\circ(e_W\tens\Id_{W'})\circ\cA_{W',W,W'}]\tens\Id_W).
 \end{equation}
Beginning with the left side, we calculate 
 
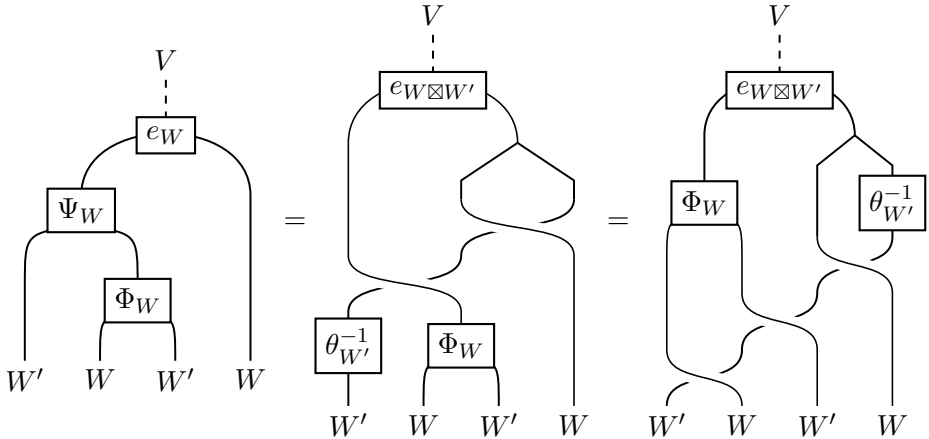
\begin{align*}
   \begin{matrix}
     \begin{tikzpicture}[scale = 1, baseline = {(current bounding box.center)}, line width=0.75pt]
     \draw (1,0) .. controls (1,.6) .. (1.5,.6);
     \draw (2,0) .. controls (2,.6) .. (1.5,.6);
     \draw (1.5,1) .. controls (1.5,1.7) .. (.75,1.8);
     \draw (0,0) -- (0,1) .. controls (0,1.7) .. (.75,1.8);
     \draw (.75,2.2) .. controls (.75,3.3) and (3,3.3) .. (3,2.2) -- (3,0);
     \draw[dashed] (1.875,3) -- (1.875,3.8);
    \node at (0,-.25) {$W'$};
    \node at (1,-.25) {$W$};
    \node at (2,-.25) {$W'$};
    \node at (3,-.25) {$W$};
    \node at (1.875,4) {$V$};
    \node at (1.5,.8) [draw,minimum width=20pt,minimum height=10pt,thick, fill=white] {$\Phi_W$};
    \node at (.75,2) [draw,minimum width=20pt,minimum height=10pt,thick, fill=white] {$\Psi_W$};
    \node at (1.875,3) [draw,minimum width=20pt,minimum height=10pt,thick, fill=white] {$e_W$};   
     \end{tikzpicture}
   \end{matrix} &  = 
   \begin{matrix}
    \begin{tikzpicture}[scale = 1, baseline = {(current bounding box.center)}, line width=0.75pt]
    \draw (0,0) -- (0,1);
    \draw (1,0) .. controls (1,.6) .. (1.5,.6);
     \draw (2,0) .. controls (2,.6) .. (1.5,.6);
     \draw (0,1) -- (0,1.25) .. controls (0, 1.75) and (1.5, 1.5) .. (1.5, 2) .. controls (1.5, 2.5) and (3, 2.25) .. (3,2.75);
     \draw[white, double=black, line width = 3pt ] (1.5,1) -- (1.5,1.25) .. controls (1.5, 1.75) and (0,1.5) .. (0,2) -- (0,3.5) .. controls (0,4.6) and (2.25,4.6) .. (2.25,3.5); 
    \draw[white, double=black, line width = 3pt ] (3,0) -- (3,2) .. controls (3,2.5) and (1.5,2.25) .. (1.5,2.75);
    \draw (1.5,2.75) -- (1.5, 3) -- (2.25, 3.5);
    \draw (2.25,3.5) -- (3,3) -- (3,2.75);
    \draw[dashed] (1.125, 4.2) -- (1.125,5);
    \node at (0,-.25) {$W'$};
    \node at (1,-.25) {$W$};
    \node at (2,-.25) {$W'$};
    \node at (3,-.25) {$W$};
    \node at (1.125,5.2) {$V$};
    \node at (1.5,.8) [draw,minimum width=20pt,minimum height=10pt,thick, fill=white] {$\Phi_W$};
    \node at (0,.8) [draw,minimum width=20pt,minimum height=10pt,thick, fill=white] {$\theta_{W'}^{-1}$};
    \node at (1.125,4.2) [draw,minimum width=20pt,minimum height=10pt,thick, fill=white] {$e_{W\boxtimes W'}$};
    \end{tikzpicture}
   \end{matrix} =
    \begin{matrix}
    \begin{tikzpicture}[scale = 1, baseline = {(current bounding box.center)}, line width=0.75pt]
     \draw (0,0) .. controls (0, .5) and (1, .25) .. (1, .75) .. controls (1, 1.25) and (2, 1) .. (2,1.5) .. controls (2,2) and (3,1.75) .. (3,2.25) -- (3,3.2) -- (2.5, 3.6) -- (2, 3.2) -- (2,2.25);
     \draw[white, double=black, line width = 3pt ] (2,0) -- (2,.75) .. controls (2, 1.25) and (1,1) .. (1,1.5) -- (1,1.9) .. controls (1,2.5) .. (.5,2.5); 
    \draw[white, double=black, line width = 3pt ] (3,0) -- (3,1.5) .. controls (3,2) and (2,1.75) .. (2,2.25);
    \draw[dashed] (1.5, 4.2) -- (1.5,5);
    \draw[white, double=black, line width = 3pt ] (1,0) .. controls (1,.5) and (0,.25) .. (0,.75) -- (0,1.9) .. controls (0,2.5) .. (.5,2.5);
    \draw (.5,2.5) -- (.5,3.6) .. controls (.5,4.6) and (2.5,4.6) .. (2.5,3.6);
    \node at (0,-.25) {$W'$};
    \node at (1,-.25) {$W$};
    \node at (2,-.25) {$W'$};
    \node at (3,-.25) {$W$};
    \node at (1.5,5.2) {$V$};
    \node at (.5,2.7) [draw,minimum width=20pt,minimum height=10pt,thick, fill=white] {$\Phi_W$};
    \node at (3,2.7) [draw,minimum width=20pt,minimum height=10pt,thick, fill=white] {$\theta_{W'}^{-1}$};
    \node at (1.5,4.2) [draw,minimum width=20pt,minimum height=10pt,thick, fill=white] {$e_{W\boxtimes W'}$};
    \end{tikzpicture}
   \end{matrix}\nonumber\\
   &\hspace{-5em} =
   \begin{matrix}
    \begin{tikzpicture}[scale = 1, baseline = {(current bounding box.center)}, line width=0.75pt]
     \draw (0,0) .. controls (0, .5) and (1, .25) .. (1, .75) .. controls (1, 1.25) and (3, 1) .. (3,1.5);
     \draw[white, double=black, line width = 3pt ] (2,0) -- (2,.75) .. controls (2, 1.25) and (1,1) .. (1,1.5) .. controls (1,2.4) and (2,2.4) .. (2,1.5) .. controls (2,1) and (3,1.25) .. (3,.75) -- (3,0); 
    \draw[white, double=black, line width = 3pt ] (1,0) .. controls (1,.5) and (0,.25) .. (0,.75) -- (0,2.5) .. controls (0,3.5) and (3,3.5) .. (3,2.5) -- (3,1.5);
    \draw[dashed] (1.5,2) .. controls (1.5,2.4) .. (0,2.5);
    \draw[dashed] (1.5,3.25) -- (1.5,4);
    \node at (0,-.25) {$W'$};
    \node at (1,-.25) {$W$};
    \node at (2,-.25) {$W'$};
    \node at (3,-.25) {$W$};
    \node at (1.5,4.25) {$V$};
    \node at (1.5,2) [draw,minimum height=10pt,thick, fill=white] {$e_W$};
    \node at (3,2) [draw,minimum width=20pt,minimum height=10pt,thick, fill=white] {$\theta_{W'}^{-1}$};
    \node at (1.5,3.25) [draw,minimum width=20pt,minimum height=10pt,thick, fill=white] {$\widetilde{e}_W$};
    \end{tikzpicture}
   \end{matrix} =
   \begin{matrix}
    \begin{tikzpicture}[scale = 1, baseline = {(current bounding box.center)}, line width=0.75pt]
     \draw (0,0) .. controls (0, .5) and (1, .25) .. (1, .75);
     \draw[white, double=black, line width = 3pt ] (2,0) -- (2,.25) .. controls (2,1.15) and (3,1.15) .. (3,.25) -- (3,0); 
    \draw[white, double=black, line width = 3pt ] (1,0) .. controls (1,.5) and (0,.25) .. (0,.75) -- (0,2.25) .. controls (0,3.15) and (1,3.15) .. (1,2.25) -- (1,.75);
    \draw[dashed] (2.5,.75) .. controls (2.5,1.15) .. (1,1.25);
    \draw[dashed] (.5,2.75) -- (.5,3.5);
    \node at (0,-.25) {$W'$};
    \node at (1,-.25) {$W$};
    \node at (2,-.25) {$W'$};
    \node at (3,-.25) {$W$};
    \node at (.5,3.75) {$V$};
    \node at (2.5,.75) [draw,minimum height=10pt,thick, fill=white] {$e_W$};
    \node at (1,1.75) [draw,minimum width=20pt,minimum height=10pt,thick, fill=white] {$\theta_{W'}^{-1}$};
    \node at (.5,2.75) [draw,minimum width=20pt,minimum height=10pt,thick, fill=white] {$\widetilde{e}_W$};
    \end{tikzpicture}
   \end{matrix} =
   \begin{matrix}
    \begin{tikzpicture}[scale = 1, baseline = {(current bounding box.center)}, line width=0.75pt]
     \draw[white, double=black, line width = 3pt ] (2,0) -- (2,.25) .. controls (2,1.15) and (3,1.15) .. (3,.25) -- (3,0); 
    \draw[white, double=black, line width = 3pt ] (0,0) -- (0,2.25) .. controls (0,3.15) and (1,3.15) .. (1,2.25) -- (1,0);
    \draw[dashed] (2.5,.75) .. controls (2.5,1.15) .. (1,1.25);
    \draw[dashed] (.5,2.75) -- (.5,3.5);
    \node at (0,-.25) {$W'$};
    \node at (1,-.25) {$W$};
    \node at (2,-.25) {$W'$};
    \node at (3,-.25) {$W$};
    \node at (.5,3.75) {$V$};
    \node at (2.5,.75) [draw,minimum height=10pt,thick, fill=white] {$e_W$};
    \node at (1,1.75) [draw,minimum width=20pt,minimum height=10pt,thick, fill=white] {$\theta_{W}$};
    \node at (0,1.75) [draw,minimum width=20pt,minimum height=10pt,thick, fill=white] {$\theta_{W'}^{-1}$};
    \node at (.5,2.75) [draw,minimum width=20pt,minimum height=10pt,thick, fill=white] {$e_W$};
    \end{tikzpicture}
   \end{matrix}
  \end{align*}
  using the definitions of $\Psi_W$ and $\Phi_W$, properties of the braiding and unit isomorphisms (such as the triangle axiom and the braiding identity $l_{W'}\circ\cR_{V,W'}^{-1}=r_{W'}$ in the fourth equality), and the definition of $\til{e}_W$. Now since $\theta_W=e^{2\pi i L(0)}$, it is easy to see that $\theta_{W'}=\theta_W'$. Thus \eqref{diag:f'_def} implies $e_W\circ(\theta_{W'}^{-1}\tens\theta_W)=e_W$ so that the above composition becomes
  \begin{equation*}
   \begin{matrix}
    \begin{tikzpicture}[scale = 1, baseline = {(current bounding box.center)}, line width=0.75pt]
     \draw[white, double=black, line width = 3pt ] (2,0) -- (2,.25) .. controls (2,1.15) and (3,1.15) .. (3,.25) -- (3,0); 
    \draw[white, double=black, line width = 3pt ] (0,0) -- (0,1.25) .. controls (0,2.15) and (1,2.15) .. (1,1.25) -- (1,0);
    \draw[dashed] (2.5,.75) .. controls (2.5,1.15) .. (1,1.25);
    \draw[dashed] (.5,1.75) -- (.5,2.5);
    \node at (0,-.25) {$W'$};
    \node at (1,-.25) {$W$};
    \node at (2,-.25) {$W'$};
    \node at (3,-.25) {$W$};
    \node at (.5,2.75) {$V$};
    \node at (2.5,.75) [draw,minimum height=10pt,thick, fill=white] {$e_W$};
    \node at (.5,1.75) [draw,minimum width=20pt,minimum height=10pt,thick, fill=white] {$e_W$};
    \end{tikzpicture}
   \end{matrix} =
   \begin{matrix}
    \begin{tikzpicture}[scale = 1, baseline = {(current bounding box.center)}, line width=0.75pt]
     \draw[white, double=black, line width = 3pt ] (2,0) -- (2,.25) .. controls (2,1.15) and (3,1.15) .. (3,.25) -- (3,0); 
    \draw[white, double=black, line width = 3pt ] (0,0) -- (0,.25) .. controls (0,1.15) and (1,1.15) .. (1,.25) -- (1,0);
    \draw[dashed] (2.5,.75) .. controls (2.5,1.15) .. (1.5,1.75) .. controls (.5,1.15) .. (.5,.75);
    \draw[dashed] (1.5,1.75) -- (1.5,2.5);
    \node at (0,-.25) {$W'$};
    \node at (1,-.25) {$W$};
    \node at (2,-.25) {$W'$};
    \node at (3,-.25) {$W$};
    \node at (1.5,2.75) {$V$};
    \node at (2.5,.75) [draw,minimum height=10pt,thick, fill=white] {$e_W$};
    \node at (.5,.75) [draw,minimum width=20pt,minimum height=10pt,thick, fill=white] {$e_W$};
    \end{tikzpicture} 
   \end{matrix} =
   \begin{matrix}
    \begin{tikzpicture}[scale = 1, baseline = {(current bounding box.center)}, line width=0.75pt]
     \draw[white, double=black, line width = 3pt ] (0,0) -- (0,.25) .. controls (0,1.15) and (1,1.15) .. (1,.25) -- (1,0); 
    \draw[white, double=black, line width = 3pt ] (2,0) -- (2,1.25) .. controls (2,2.15) and (3,2.15) .. (3,1.25) -- (3,0);
    \draw[dashed] (.5,.75) .. controls (.5,1.15) .. (2,1.25);
    \draw[dashed] (2.5,1.75) -- (2.5,2.5);
    \node at (0,-.25) {$W'$};
    \node at (1,-.25) {$W$};
    \node at (2,-.25) {$W'$};
    \node at (3,-.25) {$W$};
    \node at (2.5,2.75) {$V$};
    \node at (.5,.75) [draw,minimum height=10pt,thick, fill=white] {$e_W$};
    \node at (2.5,1.75) [draw,minimum width=20pt,minimum height=10pt,thick, fill=white] {$e_W$};
    \end{tikzpicture}
   \end{matrix},
  \end{equation*}
which is the right side of \eqref{eqn:PsiPhi_reln_1}

For the second assertion, it is enough to show
\begin{equation*}
 \til{e}_W\circ([\til{\Psi}_W\circ(\Phi_W\tens\Id_W)]\tens\Id_{W'}) =\til{e}_W\circ([r_W\circ(\Id_W\tens e_W)\circ\cA_{W,W',W}^{-1}]\tens\Id_{W'}),
\end{equation*}
which has the straightforward graphical proof
\begin{equation*}
 \begin{matrix}
     \begin{tikzpicture}[scale = 1, baseline = {(current bounding box.center)}, line width=0.75pt]
     \draw (0,0) .. controls (0,.6) .. (.5,.6);
     \draw (1,0) .. controls (1,.6) .. (.5,.6);
     \draw (.5,1) .. controls (.5,1.7) .. (1.25,1.8);
     \draw (2,0) -- (2,1) .. controls (2,1.7) .. (1.25,1.8);
     \draw (1.25,2.3) .. controls (1.25,3.4) and (3,3.4) .. (3,2.3) -- (3,0);
     \draw[dashed] (2.125,3) -- (2.125,3.8);
    \node at (0,-.25) {$W$};
    \node at (1,-.25) {$W'$};
    \node at (2,-.25) {$W$};
    \node at (3,-.25) {$W'$};
    \node at (2.125,4) {$V$};
    \node at (.5,.8) [draw,minimum width=20pt,minimum height=10pt,thick, fill=white] {$\Phi_W$};
    \node at (1.25,2) [draw,minimum width=20pt,minimum height=10pt,thick, fill=white] {$\til{\Psi}_W$};
    \node at (2.125,3) [draw,minimum width=20pt,minimum height=10pt,thick, fill=white] {$\til{e}_W$};   
     \end{tikzpicture}
   \end{matrix} = \begin{matrix}
    \begin{tikzpicture}[scale = 1, baseline = {(current bounding box.center)}, line width=0.75pt]
    \draw (0,0) -- (0,.5) .. controls (0,.8) .. (.5,.8) .. controls (1,.8) .. (1,.5) -- (1,0);
    \draw[dashed] (1.5, 2.25) -- (1.5,3);
    \draw (.5,1) -- (.5,1.6) .. controls (.5,2.6) and (2.5,2.6) .. (2.5,1.6) -- (2.5, 1);
    \draw (2,0) -- (2,.5) -- (2.5,1) -- (3,.5) -- (3,0);
    \node at (0,-.25) {$W$};
    \node at (1,-.25) {$W'$};
    \node at (2,-.25) {$W$};
    \node at (3,-.25) {$W'$};
    \node at (1.5,3.25) {$V$};
    \node at (.5,1) [draw,minimum width=20pt,minimum height=10pt,thick, fill=white] {$\Phi_W$};
    \node at (1.5,2.25) [draw,minimum width=20pt,minimum height=10pt,thick, fill=white] {$e_{W\boxtimes W'}$};
    \end{tikzpicture}
   \end{matrix} =
   \begin{matrix}
    \begin{tikzpicture}[scale = 1, baseline = {(current bounding box.center)}, line width=0.75pt]
     \draw[white, double=black, line width = 3pt ] (1,0) -- (1,.5) .. controls (1,1.4) and (2,1.4) .. (2,0.5) -- (2,0);                      
    \draw[white, double=black, line width = 3pt ] (0,0) -- (0,1.5) .. controls (0,2.5) and (3,2.5) .. (3,1.5) -- (3,0);
    \draw[dashed] (1.5,1) .. controls (1.5,1.4) .. (0,1.5);
    \draw[dashed] (1.5,2.25) -- (1.5,3);
    \node at (0,-.25) {$W$};
    \node at (1,-.25) {$W'$};
    \node at (2,-.25) {$W$};
    \node at (3,-.25) {$W'$};
    \node at (1.5,3.25) {$V$};
    \node at (1.5,1) [draw,minimum height=10pt,thick, fill=white] {$e_W$};
    \node at (1.5,2.25) [draw,minimum width=20pt,minimum height=10pt,thick, fill=white] {$\widetilde{e}_W$};
    \end{tikzpicture}
   \end{matrix}
\end{equation*}
using the definitions of $\til{\Psi}_W$ and $\Phi_W$.
\end{proof}

Using this lemma, we can now prove the following theorem; a related result in terms of a notion called ``semi-rigidity'' was obtained in \cite[Lemma 19]{Mi-semi-rigidity}:
\begin{thm}\label{thm:rigidity_criterion}
 Suppose $W$ is a $V$-module in $\cC$ and $\im\,(\til{e}_W)' \subseteq\im\Phi_W$. Then $\Phi_W$ is an isomorphism, and thus also $W$ is rigid.
\end{thm}
\begin{proof}
 Since $W\tens W'$ and $(W\tens W')'$ are grading-restricted generalized $V$-modules which are isomorphic as graded vector spaces, it is enough to show that $\Phi_W$ is injective. We will prove injectivity similar to \cite[Theorem 4.7]{McR}.
 
 Let $K$ be the kernel of $\Phi_W$ with $k: K\rightarrow W\tens W'$ the kernel morphism; we need to show that $k=0$. To this end, we define the morphism
 \begin{equation*}
  F: (W\tens W')'\tens K\rightarrow W\tens W'
 \end{equation*}
to be the composition
\begin{align*}
 (W\tens W')'\tens K & \xrightarrow{\Id_{(W\tens W')'}\tens k} (W\tens W')'\tens(W\tens W')\nonumber\\
 &\xrightarrow{\cA_{(W\tens W')',W,W'}} ((W\tens W')'\tens W)\tens W'\xrightarrow{\til{\Psi}_W\tens\Id_{W'}} W\tens W'.
\end{align*}
We shall show:
\begin{enumerate}
 \item $F\circ((\til{e}_W)'\tens\Id_K)\circ(\varphi\tens\Id_K)\circ l_K^{-1} = k$, 
 \item $F\circ(\Phi_W\tens\Id_K)=0$.
\end{enumerate}
To show why these two claims imply that $k=0$, we first factorize
\begin{align*}
\Phi_W :  W\tens W' \xrightarrow{p} \im\Phi_W\xrightarrow{q} (W\tens W')',\qquad
 (\til{e}_W)':V'\xrightarrow{\til{p}}\im\,(\til{e}_W)'\xrightarrow{\til{q}} (W\tens W')',
\end{align*}
where $p$, $\til{p}$ are surjective and $q$, $\til{q}$ are injective. By hypothesis, there is an injection $j: \im\,(\til{e}_W)'\rightarrow\im\Phi_W$ such that $q\circ j=\til{q}$.

Since $p$ is surjective and the tensoring functor $\bullet\tens K$ is right exact (see \cite[Proposition 4.26]{HLZ3}), $p\tens\Id_K$ is also surjective. This surjectivity combined with Claim (2) implies that 
\begin{equation*}
 F\circ(q\tens\Id_K)=0.
\end{equation*}
But then using Claim (1),
\begin{align*}
k & = F\circ((\til{e}_W)'\tens\Id_K)\circ(\varphi\tens\Id_K)\circ l_K^{-1}\nonumber\\
&= F\circ((\til{q}\circ\til{p})\tens\Id_K)\circ(\varphi\tens\Id_K)\circ l_K^{-1}\nonumber\\
& =F\circ((q\circ j)\tens\Id_K)\circ((\til{p}\circ\varphi)\tens\Id_K)\circ l_K^{-1}\nonumber\\
& = F\circ(q\tens\Id_K)\circ((j\circ\til{p}\circ\varphi)\tens\Id_K)\circ l_K^{-1} =0,
\end{align*}
which will prove the theorem once we have proven Claims (1) and (2).

To prove Claim (1), naturality of the left unit isomorphisms implies that it is enough to show that the composition
\begin{align*}
 W\tens W'  \xrightarrow{l_{W\tens W'}^{-1}} & V\tens(W\tens W')\xrightarrow{[(\til{e}_W)'\circ\varphi]\tens\Id_{W\tens W'}} (W\tens W')'\tens(W\tens W')\nonumber\\
 &\xrightarrow{\cA_{(W\tens W')',W,W'}} ((W\tens W')'\tens W)\tens W' \xrightarrow{\til{\Psi}_W\tens\Id_{W'}} W\tens W'
\end{align*}
is the identity on $W\tens W'$. Then naturality of the associativity isomorphisms and properties of the left unit reduce us to showing that
\begin{align*}
 W\xrightarrow{l_W^{-1}} V\tens W\xrightarrow{[(\til{e}_W)'\circ\varphi]\tens\Id_W} (W\tens W')'\tens W\xrightarrow{\til{\Psi}_W} W
\end{align*}
is the identity on $W$. For this, the universal property of contragredients implies that it is enough to show
\begin{equation*}
 \til{e}_W\circ\left([\til{\Psi}_W\circ(((\til{e}_W)'\circ\varphi)\tens\Id_W)\circ l_W^{-1}]\tens\Id_{W'}\right)=\til{e}_W.
\end{equation*}
In fact, the definition of $\til{\Psi}_W$ implies that the left side is
\begin{align*}
 W\tens W'\xrightarrow{l_{W}^{-1}\tens\Id_{W'}} (V\tens W)\tens W' & \xrightarrow{([(\til{e}_W)'\circ\varphi]\tens\Id_W)\tens\Id_{W'}} ((W\tens W')'\tens W)\tens W'\nonumber\\ &\xrightarrow{\cA_{(W\tens W')',W,W'}^{-1}} (W\tens W')'\tens (W\tens W')\xrightarrow{e_{W\tens W'}} V.
\end{align*}
We apply naturality of associativity, and then use properties of the left unit isomorphisms as well as \eqref{diag:f'_def}, to get
\begin{align*}
 W\tens W'\xrightarrow{l_{W\tens W'}^{-1}} V\tens(W\tens W')\xrightarrow{\varphi\tens \til{e}_W} V'\tens V\xrightarrow{e_V} V.
\end{align*}
Then \eqref{diag:phi_char} together with naturality of the unit isomorphisms reduces this composition to $\til{e}_W$, completing the proof of Claim (1).

To prove Claim (2), we calculate
\begin{align*}
 F\circ(\Phi_W & \tens\Id_K)  =(\til{\Psi}_W\tens\Id_{W'})\circ\cA_{(W\tens W')',W,W'}\circ(\Id_{(W\tens W')'}\tens k)\circ(\Phi_W\tens\Id_K)\nonumber\\
 & =(\til{\Psi}_W\tens\Id_{W'})\circ((\Phi_W\tens\Id_W)\tens\Id_{W'})\circ\cA_{W\tens W',W,W'}\circ(\Id_{W\tens W'}\tens k)\nonumber\\
 & =([r_W\circ(\Id_W\tens e_W)\circ\cA_{W,W',W}^{-1}]\tens\Id_{W'})\circ\cA_{W\tens W',W,W'}\circ(\Id_{W\tens W'}\tens k)\nonumber\\
 & =(\Id_W\tens l_{W'})\circ(\Id_W\tens(e_W\tens\Id_{W'}))\circ\cA_{W,W'\tens W,W'}^{-1}\circ(\cA_{W,W',W}^{-1}\tens\Id_{W'})\nonumber\\
 &\qquad\qquad\qquad\quad \circ\cA_{W\tens W',W,W'}\circ(\Id_{W\tens W'}\tens k)\nonumber\\
 & = (\Id_W\tens[l_{W'}\circ(e_W\tens\Id_{W'})\circ\cA_{W',W,W'}])\circ\cA_{W,W',W\tens W'}^{-1}\circ(\Id_{W\tens W'}\tens k)\nonumber\\
 & =(\Id_W\tens\Psi_W)\circ(\Id_W\tens(\Id_{W'}\tens\Phi_W))\circ(\Id_W\tens(\Id_{W'}\tens k))\circ\cA_{W,W',K}^{-1} = 0,
\end{align*}
using the definition of $F$, naturality of the associativity isomorphisms, Lemma \ref{lem:PsiPhi}(2), the triangle axiom and naturality of associativity, the pentagon axiom, Lemma \ref{lem:PsiPhi}(1) and naturality of associativity, and finally the fact that $\Phi_W\circ k=0$ since $(K,k)$ is the kernel of $\Phi_W$. This proves Claim (2) and thus also the theorem.
\end{proof}

To verify the condition of Theorem \ref{thm:rigidity_criterion} for a $V$-module $W$ in Section \ref{sec:main_thms}, we will need to know the relation between $\Phi_W$ and its contragredient $\Phi_W': (W\tens W')''\rightarrow(W\tens W')'$:
\begin{prop}\label{prop:Phi_and_Phi_prime}
 For $W$ a $V$-module in $\cC$, $\Phi_W=\Phi_W'\circ\delta_{W\tens W'}$.
\end{prop}
\begin{proof}
It is enough to show
 \begin{align}\label{eqn:Phi_and_Phi_prime}
  & e_{W\tens W'}  \circ([\Phi_W'\circ\delta_{W\tens W'}]\tens\Id_{W\tens W'})\nonumber\\ &\quad=\til{e}_W\circ(r_W\tens\Id_{W'})\circ((\Id_W\tens e_W)\tens\Id_{W'})\circ(\cA_{W,W',W}^{-1}\tens\Id_{W'})\circ\cA_{W\tens W',W,W'},
 \end{align}
by the definition of $\Phi_W$ and the universal property of $((W\tens W')',e_{W\tens W'})$. The left side is
\begin{align*}
 e_{W\tens W'}  \circ([\Phi_W' & \circ\delta_{W\tens W'}]  \tens\Id_{W\tens W'})= e_{(W\tens W')'}\circ(\delta_{W\tens W'}\tens\Phi_W)\nonumber\\
 &=e_{W\tens W'}\circ\cR_{W\tens W',(W\tens W')'}\circ(\theta_{W\tens W'}\tens\Phi_W)\nonumber\\
 & = e_{W\tens W'}\circ(\Phi_W\tens\Id_{W\tens W'})\circ\cR_{W\tens W',W\tens W'}\circ(\theta_{W\tens W'}\tens\Id_{(W\tens W')'}).
\end{align*}
by the definitions and \eqref{diag:delta_def}. We continue analyzing this composition graphically, using
 the definition of $\Phi_W$, the hexagon axioms, and the balancing property of $\theta_{W\tens W'}$ to begin:
 \begin{align}\label{Phi_prime_calc}
  \begin{matrix}
   \begin{tikzpicture}[scale = 1, baseline = {(current bounding box.center)}, line width=0.75pt]
    \draw (3,0) -- (3,3.25) .. controls (3,3.75) and (1,3.5) .. (1,4);
    \draw (2,0) -- (2,2.5) .. controls (2,3) and (0,2.75) .. (0,3.25);
    \draw (1,0) -- (1,1) .. controls (1,1.5) and (0,1.25) .. (0,1.75);
     \draw[white, double=black, line width = 3pt ] (0,0) -- (0,1) .. controls (0, 1.5) and (1, 1.25) .. (1, 1.75) .. controls (1, 2.25) and (0, 2) .. (0,2.5);
     \draw[white, double=black, line width = 3pt ] (1,4) .. controls (1,4.9) and (2,4.9) .. (2,4) .. controls (2,3.5) and (1,3.75) .. (1,3.25) .. controls (1,2.75) and (0,3) .. (0,2.5); 
    \draw[white, double=black, line width = 3pt ] (0,3.25) -- (0,5) .. controls (0,6) and (3,6) .. (3,5) -- (3,4) .. controls (3,3.5) and (2,3.75) .. (2,3.25) .. controls (2,2.75) and (1,3) .. (1,2.5) .. controls (1,2) and (0,2.25) .. (0,1.75);
    \draw[dashed] (1.5,4.5) .. controls (1.5,4.9) .. (0,5);
    \draw[dashed] (1.5,5.75) -- (1.5,6.5);
    \node at (0,-.25) {$W$};
    \node at (1,-.25) {$W'$};
    \node at (2,-.25) {$W$};
    \node at (3,-.25) {$W'$};
    \node at (1.5,6.75) {$V$};
    \node at (1.5,4.5) [draw,minimum height=10pt,thick, fill=white] {$e_W$};
    \node at (1,.5) [draw,thick, fill=white] {$\theta_{W'}$};
    \node at (0,.5) [draw, thick, fill=white] {$\theta_{W}$};
    \node at (1.5,5.75) [draw,minimum width=20pt,minimum height=10pt,thick, fill=white] {$\widetilde{e}_W$};
    \end{tikzpicture}
   \end{matrix} & =
   \begin{matrix}
    \begin{tikzpicture}[scale = 1, baseline = {(current bounding box.center)}, line width=0.75pt]
     \draw (3,0) -- (3,3.25) .. controls (3,3.75) and (2,3.5) .. (2,4) .. controls (2,4.5) and (1,4.25) .. (1,4.75);
     \draw (2,0) -- (2,1) .. controls (2,1.5) and (0,1.25) .. (0,1.75);
     \draw[white, double=black, line width = 3pt ] (1,0) -- (1,1) .. controls (1,1.5) and (2,1.25) .. (2,1.75) .. controls (2,2.25) and (1,2) .. (1,2.5);
      \draw[white, double=black, line width = 3pt ] (0,0) -- (0,1) .. controls (0, 1.5) and (1, 1.25) .. (1, 1.75) .. controls (1, 2.25) and (2, 2) .. (2,2.5) .. controls (2,3) and (1,2.75) .. (1,3.25); 
     \draw[white, double=black, line width = 3pt ] (0,1.75) -- (0,5.75) .. controls (0,6.75) and (3,6.75) .. (3,5.75) -- (3,4) .. controls (3,3.5) and (2,3.75) .. (2,3.25) .. controls (2,2.75) and (1,3) .. (1,2.5);
     \draw[white, double=black, line width = 3pt ] (1,4.75) .. controls (1,5.65) and (2,5.65) .. (2,4.75) .. controls (2,4.25) and (1,4.5) .. (1,4) -- (1,3.25);
    \draw[dashed] (1.5,5.25) .. controls (1.5,5.65) .. (3,5.75);
    \draw[dashed] (1.5,6.5) -- (1.5,7.25);
    \node at (0,-.25) {$W$};
    \node at (1,-.25) {$W'$};
    \node at (2,-.25) {$W$};
    \node at (3,-.25) {$W'$};
    \node at (1.5,7.5) {$V$};
    \node at (1.5,5.25) [draw,minimum height=10pt,thick, fill=white] {$e_W$};
    \node at (1,.5) [draw,thick, fill=white] {$\theta_{W'}$};
    \node at (0,.5) [draw, thick, fill=white] {$\theta_{W}$};
    \node at (1.5,6.5) [draw,minimum width=20pt,minimum height=10pt,thick, fill=white] {$\widetilde{e}_W$};
    \end{tikzpicture}
    \end{matrix} =
    \begin{matrix}
    \begin{tikzpicture}[scale = 1, baseline = {(current bounding box.center)}, line width=0.75pt]
      \draw (3,0) -- (3,2.5) .. controls (3,3) and (2,2.75) .. (2,3.25) .. controls (2,4.25) and (1,4.25) .. (1,3.25) .. controls (1,2.75) and (2,3) .. (2,2.5);
      \draw (2,0) -- (2,1) .. controls (2,1.5) and (0,1.25) .. (0,1.75);
      \draw[white, double=black, line width = 3pt ] (1,0) -- (1,1) .. controls (1,1.5) and (2,1.25) .. (2,1.75) .. controls (2,2.25) and (1,2) .. (1,2.5);
       \draw[white, double=black, line width = 3pt ] (0,0) -- (0,1) .. controls (0, 1.5) and (1, 1.25) .. (1, 1.75) .. controls (1, 2.25) and (2, 2) .. (2,2.5) .. controls (2,3) and (1,2.75) .. (1,3.25); 
     \draw[white, double=black, line width = 3pt ] (0,1.75) -- (0,4.25) .. controls (0,5.25) and (3,5.25) .. (3,4.25) -- (3,3.25) .. controls (3,2.75) and (1,3) .. (1,2.5);
    \draw[dashed] (1.5,3.75) .. controls (1.5,4.15) .. (3,4.25);
    \draw[dashed] (1.5,5) -- (1.5,5.75);
    \node at (0,-.25) {$W$};
    \node at (1,-.25) {$W'$};
    \node at (2,-.25) {$W$};
    \node at (3,-.25) {$W'$};
    \node at (1.5,6) {$V$};
    \node at (1.5,3.75) [draw,minimum height=10pt,thick, fill=white] {$\widetilde{e}_W$};
    \node at (1,.5) [draw,thick, fill=white] {$\theta_{W'}$};
    \node at (1.5,5) [draw,minimum width=20pt,minimum height=10pt,thick, fill=white] {$\widetilde{e}_W$};
    \end{tikzpicture}
    \end{matrix}\nonumber\\
    &\hspace{-8em} =
        \begin{matrix}
    \begin{tikzpicture}[scale = 1, baseline = {(current bounding box.center)}, line width=0.75pt]
       \draw (2,0) -- (2,1) .. controls (2,1.5) and (1,1.25) .. (1,1.75) .. controls (1,2.25) and (0,2) .. (0,2.5) -- (0,4.5) .. controls (0,5.5) and (1,5.5) .. (1,4.5) -- (1,3.25) .. controls (1,2.75) and (2,3) .. (2,2.5) -- (2,1.75);
       \draw[white, double=black, line width = 3pt ] (1,0) -- (1,1) .. controls (1,1.5) and (2,1.25) .. (2,1.75);
       \draw[white, double=black, line width = 3pt ] (3,0) -- (3,3.25) .. controls (3,4.25) and (2,4.25) .. (2,3.25) .. controls (2,2.75) and (1,3) .. (1,2.5) .. controls (1,2) and (0,2.25) .. (0,1.75) -- (0,0);
    \draw[dashed] (2.5,3.75) .. controls (2.5,4.4) .. (1,4.5);
    \draw[dashed] (.5,5) -- (.5,5.75);
    \node at (0,-.25) {$W$};
    \node at (1,-.25) {$W'$};
    \node at (2,-.25) {$W$};
    \node at (3,-.25) {$W'$};
    \node at (.5,6) {$V$};
    \node at (2.5,3.75) [draw,minimum height=10pt,thick, fill=white] {$\widetilde{e}_W$};
    \node at (1,.5) [draw,thick, fill=white] {$\theta_{W'}$};
    \node at (.5,5) [draw,minimum width=20pt,minimum height=10pt,thick, fill=white] {$\widetilde{e}_W$};
    \end{tikzpicture}
    \end{matrix} =
    \begin{matrix}
    \begin{tikzpicture}[scale = 1, baseline = {(current bounding box.center)}, line width=0.75pt]
     \draw (2,0) -- (2,1) .. controls (2,1.5) and (1,1.25) .. (1,1.75) .. controls (1,2.25) and (0,2) .. (0,2.5) .. controls (0,3.5) and (1,3.5) .. (1,2.5) .. controls (1,2) and (2,2.25) .. (2,1.75); 
     \draw[white, double=black, line width = 3pt ] (1,0) -- (1,1) .. controls (1,1.5) and (2,1.25) .. (2,1.75);
    \draw[white, double=black, line width = 3pt ] (0,0) -- (0,1.75) .. controls (0,2.25) and (2,2) .. (2,2.5) .. controls (2,3.5) and (3,3.5) .. (3,2.5) -- (3,0);
    \draw[dashed] (2.5,3) .. controls (2.5,3.4) .. (1.5,4) .. controls (.5,3.4) .. (.5,3);
    \draw[dashed] (1.5,4) -- (1.5,4.75);
    \node at (0,-.25) {$W$};
    \node at (1,-.25) {$W'$};
    \node at (2,-.25) {$W$};
    \node at (3,-.25) {$W'$};
    \node at (1.5,5) {$V$};
    \node at (2.5,3) [draw,minimum height=10pt,thick, fill=white] {$\widetilde{e}_W$};
    \node at (.5,3) [draw,minimum width=20pt,minimum height=10pt,thick, fill=white] {$\widetilde{e}_W$};
    \node at (1,.5) [draw,minimum width=20pt,minimum height=10pt,thick, fill=white] {$\theta_{W'}$};
    \end{tikzpicture} 
   \end{matrix} =
   \begin{matrix}
    \begin{tikzpicture}[scale = 1, baseline = {(current bounding box.center)}, line width=0.75pt]
    \draw (2,0) -- (2,1) .. controls (2,1.5) and (1,1.25) .. (1,1.75);
     \draw[white, double=black, line width = 3pt ] (1,0) -- (1,1) .. controls (1,1.5) and (2,1.25) .. (2,1.75) .. controls (2,2.75) and (1,2.75) .. (1,1.75);                      
    \draw[white, double=black, line width = 3pt ] (0,0) -- (0,2.75) .. controls (0,3.75) and (3,3.75) .. (3,2.75) -- (3,0);
    \draw[dashed] (1.5,2.25) .. controls (1.5,2.65) .. (0,2.75);
    \draw[dashed] (1.5,3.75) -- (1.5,4.25);
    \node at (0,-.25) {$W$};
    \node at (1,-.25) {$W'$};
    \node at (2,-.25) {$W$};
    \node at (3,-.25) {$W'$};
    \node at (1.5,4.5) {$V$};
    \node at (1.5,2.25) [draw,minimum height=10pt,thick, fill=white] {$\widetilde{e}_W$};
    \node at (1.5,3.5) [draw,minimum width=20pt,minimum height=10pt,thick, fill=white] {$\widetilde{e}_W$};
    \node at (1,.5) [draw,minimum width=20pt,minimum height=10pt,thick, fill=white] {$\theta_{W'}$};
    \end{tikzpicture}
   \end{matrix} .
 \end{align}
Now by the definition of $\til{e}_W$, the balancing equation, the naturality of $\theta$, and $\theta_V=\Id_V$,
\begin{align*}
 \til{e}_W\circ\cR_{W',W}\circ(\theta_{W'}\tens\Id_W) = e_W\circ\cR_{W',W}^2\circ(\theta_{W'}\tens\theta_W)=e_W\circ\theta_{W'\tens W}=\theta_V\circ e_W=e_W,
\end{align*}
so the right side of \eqref{Phi_prime_calc} reduces to the right side of \eqref{eqn:Phi_and_Phi_prime}.
\end{proof}

\section{Series expansions of correlation functions}\label{sec:expansions}

Here we present results on the convergence and modular invariance of genus-one correlation functions associated to a vertex operator algebra that we will use in the next section. Most of the results in this section have appeared at some level of generality in one or more of the works of Zhu \cite{Zh}, Dong-Li-Mason \cite{DLM}, Miyamoto \cite{Mi1, Mi2}, Huang \cite{Hu-mod_inv, Hu-Verlinde}, Fiordalisi \cite{Fi}, and Carnahan-Miyamoto \cite{CM}. However, some of the results in \cite{Hu-mod_inv, Hu-Verlinde, CM} on certain multivalued analytic functions need to be enhanced to include specific information about convergence regions for single-valued branches, and about relations between single-valued branches on different but overlapping simply-connected domains.

We assume $V$ is an $\NN$-graded $C_2$-cofinite vertex operator algebra. As mentioned in Subsection \ref{subsec:vrtx_tens_cat}, the category $\cC$ of grading-restricted generalized $V$-modules has vertex algebraic braided tensor category structure. Also, all results from \cite{Fi} and from Sections 1 through 5 of \cite{Hu-mod_inv} that we quote below hold when $V$ is $\NN$-graded and $C_2$-cofinite.

\subsection{Geometrically-modified genus-zero correlation functions}

We will use the linear automorphism $\cU(1)$ of $V$, and indeed of any $V$-module, introduced in \cite[Section 1]{Hu-mod_inv}:
\begin{equation}\label{eqn:U1_def}
 \cU(1)=(2\pi i)^{L(0)}\exp\bigg(-\sum_{j=1}^\infty A_j L(j)\bigg),
\end{equation}
where the $A_j$ are the unique complex numbers such that
\begin{equation*}
 \frac{1}{2\pi i}(e^{2\pi i x}-1) =\exp\bigg(-\sum_{j=1}^\infty A_j x^{j+1}\frac{d}{dx}\bigg)\cdot x.
\end{equation*}
When $\cU(1)$ acts on generalized $V$-modules with non-integral conformal weights, or with non-semisimple $L(0)$-actions, we must interpret $(2\pi i)^{L(0)}$ using a fixed choice of branch of logarithm. Recall the principal branch $\log z$ from \eqref{eqn:log_branch}, so that in particular  $\log 2\pi i=\ln 2\pi +i\,\frac{\pi}{2}$. We interpret $(2\pi i)^{L(0)}$ to mean $e^{(\log 2\pi i)L(0)}$, which is a well-defined operator on each (finite-dimensional) conformal weight space of a grading-restricted generalized $V$-module (see the discussion surrounding \cite[Equation 3.69]{HLZ2}).

Because $\cU(1)$ is constructed from Virasoro operators, $\cU(1)$ commutes with any $V$-module homomorphism. Moreover, \cite[Proposition 1.2]{Hu-mod_inv} shows how $\cU(1)$ conjugates intertwining operators: If $\cY$ is a $V$-module intertwining operator of type $\binom{W_3}{W_1\,W_2}$, then
\begin{equation}\label{eqn:U(1)_intw_op}
 \cU(1)\cY(w_1,x)w_2=\cY(e^{2\pi i x L(0)}\cU(1)w_1,e^{2\pi i x}-1)\cU(1)w_2
\end{equation}
for $w_1\in W_1$, $w_2\in W_2$; the logarithmic intertwining operator generality is discussed in \cite[Lemma 1.25]{Fi} and \cite[Lemma 5.5]{CM}. When $\cY$ involves non-integral powers of $x$, or powers of $\log x$, we must interpret the substitution of $e^{2\pi i x}-1$ in $\cY$ properly. If $\cY$ involves powers of $\log x$, we interpret $\log(e^{2\pi i x}-1)$ to mean
\begin{equation*}
 \log(e^{2\pi i x}-1) = \log 2\pi i+\log x+\log\left(\frac{e^{2\pi i x}-1}{2\pi i x}\right),
\end{equation*}
where the third term on the right is the standard power series for $\log(1+y)$ centered at $y=0$, with $y=\frac{e^{2\pi i x}-1}{2\pi i x}-1$. (Since $y$ is a power series in $x$ with no constant term, $\log(1+y)$ is also a well-defined power series in $x$ with no constant term.)
 Similarly, if $\cY$ contains non-integral powers $x^h$, we interpret $(e^{2\pi i x}-1)^h$ to mean
\begin{equation*}
(e^{2\pi i x}-1)^h = e^{(\log 2\pi i)h}x^h\left(\frac{e^{2\pi i x}-1}{2\pi i x}\right)^h,
\end{equation*}
where the third factor on the right is the binomial expansion of $(1+y)^h$ at $y=\frac{e^{2\pi i x}-1}{2\pi i x}-1$.

Now for $z\in\CC$, we introduce the notation $q_z=e^{2\pi i z}$ and define $\cU(q_z)=e^{2\pi i z L(0)} \cU(1)$. For substituting the complex number $q_z$ in an intertwining operator, we use the conventions 
$$\log x\vert_{x=q_z} =2\pi i z,\qquad x^h\vert_{x=q_z} =e^{2\pi i z h}.$$
(With this convention, $\log x\vert_{x=q_z}$ is the principal branch $\log q_z$ only when $-\frac{1}{2}<\mathrm{Re}\,z\leq\frac{1}{2}$.) Now if $\cY_1$ and $\cY_2$ are $V$-module intertwining operators of types $\binom{W_4}{W_1\,M}$ and $\binom{M}{W_2\,W_3}$, respectively, then the product
\begin{align*}
 &\left\langle w_4',\cY_1(\cU(q_{z_1})w_1,q_{z_1})\cY_2(\cU(q_{z_2})w_2,q_{z_2})w_3\right\rangle\nonumber\\
 &\hspace{5em}=\sum_{h\in\CC} \left\langle w_4',\cY_1(\cU(q_{z_1})w_1,q_{z_1})\pi_h\left(\cY_2(\cU(q_{z_2})w_2,q_{z_2})w_3\right)\right\rangle
\end{align*}
for $w_4'\in W_4'$, $w_1\in W_1$, $w_2\in W_2$, $w_3\in W_3$ converges absolutely when $\vert q_{z_1}\vert >\vert q_{z_2}\vert >0$, where $\pi_h$ is the projection $\overline{M}\rightarrow M_{[h]}$. For simplicity, we will always take $z_2=0$ and set $z_1=z$; moreover, we will use the simplified notation $\cY_2(\cU(1)w_2,1)$ to denote $\cY_2(\cU(q_0)w_2,q_0)$. Thus,
\begin{equation*}
P_{\cY_1,\cY_2}(w_4',w_1,w_2,w_3;z)= \left\langle w_4',\cY_1(\cU(q_{z})w_1,q_{z})\cY_2(\cU(1)w_2,1)w_3\right\rangle
\end{equation*}
defines a (single-valued) analytic function on the simply-connected basic convergence region for products of geometrically-modified intertwining operators:
\begin{equation*}
 U_0^{prod} =\lbrace z\in\CC\,\,\vert\,\,\mathrm{Im}\,z<0\rbrace.
\end{equation*}
The subscript in $U_0^{prod}$ indicates that $P_{\cY_1,\cY_2}$ is (a single-valued branch of) a genus-zero correlation function. By the genus-zero differential equations in \cite{Hu-diff-eqs}, $P_{\cY_1,\cY_2}$ can be extended to a multivalued analytic function on the region where $q_{z}\neq 1$, that is,
\begin{equation*}
U_0 = \CC\setminus\ZZ,
\end{equation*}
with principal branch given by $P_{\cY_1,\cY_2}$ on $U_0^{prod}$. We will use $\overline{P}_{\cY_1,\cY_2}$ to denote the multivalued extension of $P_{\cY_1,\cY_2}$ to $U_0$.

We will now use associativity of intertwining operators and \eqref{eqn:U(1)_intw_op} to find a series expansion for $\overline{P}_{\cY_1,\cY_2}$ about its singularity $z=0$. First, the universal property of $V$-module tensor products implies that there are unique $V$-module homomorphisms $f_1: W_1\tens M\rightarrow W_3$ and $f_2: W_2\tens W_3\rightarrow M$ such that
\begin{equation*}
 f_1\circ\cY_{W_1,M} =\cY_1,\qquad f_2\circ\cY_{W_2,W_3}=\cY_2.
\end{equation*}
Then by definition of the associativity isomorphisms in the category of $V$-modules, for $r_1,r_2\in\RR_+$ such that $r_1>r_2>r_1-r_2>0$, we have
\begin{align}\label{eqn:intw_op_assoc}
 \cY_1(w_1,e^{\ln r_1})&\cY_2(w_2,e^{\ln r_2}) w_3 = (f_1\circ\cY_{W_1,M})(w_1,e^{\ln r_1})(f_2\circ\cY_{W_2,W_3})(w_2,e^{\ln r_2})w_3\nonumber\\
 & =\overline{f_1\circ(\Id_{W_1}\tens f_2)}\left(\cY_{W_1,W_2\tens W_3}(w_1,e^{\ln r_1})\cY_{W_2,W_3}(w_2,e^{\ln r_2})w_3\right)\nonumber\\
 & =\overline{f_1\circ(\Id_{W_1}\tens f_2)\circ\cA_{W_1,W_2,W_3}^{-1}}\left(\cY_{W_1\tens W_2,W_3}(\cY_{W_1,W_2}(w_1,e^{\ln(r_1-r_2)})w_2,e^{\ln r_2})w_3\right)\nonumber\\
 & =\cY^0\left(\cY_{W_1,W_2}(w_1,e^{\ln(r_1-r_2)})w_2,e^{\ln r_2}\right)w_3
\end{align}
for $w_1\in W_1$, $w_2\in W_2$, and $w_3\in W_3$, where
\begin{equation*}
 \cY^0=f_1\circ(\Id_{W_1}\tens f_2)\circ\cA_{W_1,W_2,W_3}^{-1}\circ\cY_{W_1\tens W_2,W_3}
\end{equation*}
is an intertwining operator of type $\binom{W_4}{W_1\tens W_2\,W_3}$. We also introduce some notation for the expansion of $\cY_{W_1,W_2}$ as a formal series:
\begin{equation*}
 \cY_{W_1,W_2}(w_1,x)w_2=\sum_{h\in\QQ}\sum_{k=0}^K (w_1\tens_{h,k} w_2)\,x^{-h-1}(\log x)^k,
\end{equation*}
where $w_1\tens_{h,k} w_2\in W_1\tens W_2$ for each $h,k$ (with conformal weight $\mathrm{wt}\,w_1+\mathrm{wt}\,w_2-h-1$ if $w_1$ and $w_2$ are homogeneous).

\begin{rem}\label{rem:K_def}
The outer sum in the expansion of $\cY_{W_1,W_2}$ is over $\QQ$ because the conformal weights of $W_1$, $W_2$, and $W_1\tens W_2$ are rational numbers (see \cite[Corollary 5.10]{Mi2}). Moreover, the bound $K\in\NN$ on the maximum power of $\log x$ in $\cY_{W_1,W_2}(w_1,x)w_2$ can be fixed independently of $w_1\in W_1$, $w_2\in W_2$, and $h\in\QQ$: Using \cite[Proposition 3.20(a)]{HLZ2} and using $L(0)_{nil}$ to denote the nilpotent part of $L(0)$, we can take $K=k_1+k_2+k_3$ where $k_1$, $k_2$, and $k_3$ are the greatest non-negative integers such that $L(0)_{nil}^{k_i}\neq 0$ on $W_1$, $W_2$, and $W_1\tens W_2$, respectively. Note that each $k_i<\infty$ because $W_1$, $W_2$, and $W_1\tens W_2$ are finitely generated with finite-dimensional weight spaces.
\end{rem}

Using the above notation and the intertwining operator $\cY^0$, we can obtain a series expansion for $\overline{P}_{\cY_1,\cY_2}$ about $z=0$ from \cite[Proposition 1.4]{Hu-mod_inv}. However, we shall need to enhance this result to include specific branch and convergence region information. Thus we now introduce the basic simply-connected convergence region for iterates of geometrically-modified intertwining operators:
\begin{equation*}
  U_0^{it}=\left\lbrace z\in\CC\,\,\vert\,\,0<\vert z\vert<1,\, \arg z\neq\pi\right\rbrace.
 \end{equation*}
 Now the extension of $P_{\cY_1,\cY_2}$ to $U_0^{it}$ is as follows: 
\begin{prop}\label{prop:geom_mod_assoc}
 The analytic extension of the principal branch $P_{\cY_1,\cY_2}$ of $\overline{P}_{\cY_1,\cY_2}$ to the simply-connected domain $U_0^{it}$ is given by
\begin{align}\label{eqn:I_series}
 I_{\cY^0,\cY_{W_1,W_2}}(  w_4', & \,w_1,w_2,w_3;z) =\left\langle w_4',\cY^0\left(\cU(1)\cY_{W_1,W_2}(w_1, e^{\log z})w_2,1\right)w_3\right\rangle\nonumber\\
 & :=\sum_{h\in\QQ}\sum_{k=0}^K \left.\left\langle w_4',\cY^0\left(\cU(1)(w_1\tens_{h,k} w_2),1\right)w_3\right\rangle x^{-h-1}(\log x)^k\right\vert_{\log x=\log z}.
\end{align}
In particular, the indicated series converges absolutely for $0<\vert z\vert <1$.
\end{prop}
\begin{proof}
 Take $z=ir$, with $r\in\RR_+$, such that
 \begin{equation*}
  \vert q_{z}\vert> 1 >\vert q_{z}-1\vert >0,
 \end{equation*}
so $q_{z}$, $1$, and $q_{z}-1$ are positive real numbers satisfying the conditions for associativity of intertwining operators in \eqref{eqn:intw_op_assoc}. That is, we assume $r$ satisfies
\begin{equation*}
 1 < e^{-2\pi r} < 2,\quad\text{equivalently},\quad -\frac{\ln 2}{2\pi} < r <0.
\end{equation*}
Then by \eqref{eqn:intw_op_assoc},
\begin{align}\label{eqn:PY_1Y_2_calc}
  P_{\cY_1,\cY_2}  ( & w_4',w_1,w_2,w_3;z) =\langle w_4',\cY_1(\cU(q_{ir}) w_1,e^{-2\pi r})\cY_2(\cU(1)w_2,1)w_3\rangle\nonumber\\
 & =\left\langle w_4',\cY^0\big(\cY_{W_1,W_2}\big(\cU(q_{ir})w_1,e^{\ln(e^{-2\pi r}-1)}\big)\cU(1)w_2,1\big)w_3\right\rangle\nonumber\\
 & =\sum_{h\in\QQ}\left\langle w_4',\cY^0\big(\pi_h\big(\cY_{W_1,W_2}\big(\cU(q_{z})w_1, e^{\log(e^{2\pi i z}-1)}\big)\cU(1)w_2\big),1\big)w_3\right\rangle.
\end{align} 
We would like to treat the right side of this equation as a series in powers of $z$ and $\log z$ and then apply \eqref{eqn:U(1)_intw_op}, but we must first justify rearranging the sum. To do so, we introduce some more notation for the series expansion of $\cY_{W_1,W_2}$.

By projecting to the indecomposable summands of $W_1$, $W_2$, and $W_1\tens W_2$, we can write $\cY_{W_1,W_2}=\sum_{j=1}^J \cY^{(j)}$ such that the powers of $x$ in each intertwining operator $\cY^{(j)}$ of type $\binom{W_1\tens W_2}{W_1\,W_2}$ are all congruent mod $\ZZ$ (since the conformal weights of an indecomposable $V$-module are all congruent mod $\ZZ$). Thus there are $h_j\in\QQ$ for $1\leq j\leq J$ such that
\begin{equation}\label{eqn:Y_tens_decomp}
 \cY_{W_1,W_2}(w_1,x)w_2 =\sum_{j=1}^J\sum_{k=0}^K x^{h_j}(\log x)^k\cY^{(j)}_k(w_1,x)w_2,
\end{equation}
with each $\cY^{(j)}_k(w_1,x)w_2\in(W_1\tens W_2)((x))$ for $w_1\in W_1$, $w_2\in W_2$. We use the notation
\begin{equation*}
 \cY^{(j)}_k(w_1,x)w_2=\sum_{n\in\ZZ} w_1(j,k; n)w_2\,x^n
\end{equation*}
for $w_1\in W_1$, $w_2\in W_2$.

We now allow $z$ to vary in the open set such that $0<\vert q_{z}-1\vert<1$, since the right side of \eqref{eqn:PY_1Y_2_calc} is absolutely convergent in this region. By \cite[Proposition 7.20]{HLZ5}, the right side of \eqref{eqn:PY_1Y_2_calc} also equals the absolutely-convergent series
\begin{align}\label{eqn:gjk_obtain}
 &\sum_{j=1}^J\sum_{k=0}^K \sum_{n\in\ZZ}\left\langle w_4',\cY^0\big([\cU(q_{z})w_1](j,k;n)\cU(1)w_2,1\big)w_3\right\rangle (e^{2\pi i z}-1)^{h_j+n} \left(\log(e^{2\pi i z}-1)\right)^k.
\end{align}
Since this converges absolutely, so does each series 
\begin{align}\label{eqn:gjk_def}
 g_{j,k}(z) & =\sum_{n\in\ZZ}\left\langle w_4',\cY^0\big([\cU(q_{z})w_1](j,k;n)\cU(1)w_2,1\big)w_3\right\rangle(e^{2\pi i z}-1)^{n}
\end{align}
when $0<\vert q_{z}-1\vert<1$. Assuming as we may that $\cU(1)w_1$ is $L(0)$-homogeneous, we can write
\begin{align*}
 g_{j,k}(z)&=e^{2\pi i(\mathrm{wt}\,\cU(1)w_1)z}\sum_{n\in\ZZ}(e^{2\pi i z}-1)^n\cdot\nonumber\\
 &\hspace{3em}\cdot\sum_{l=0}^K \frac{(2\pi i z)^l}{l!}\left\langle w_4',\cY^0\big([L(0)_{nil}^l \cU(1)w_1](j,k;n)\cU(1)w_2,1\big)w_3\right\rangle. 
\end{align*}
Here, the double series (and not just the iterated series) converges absolutely because for each $l$, the series
\begin{align}\label{eqn:double_abs_conv}
 \sum_{n\in\ZZ} & \left\langle w_4',\cY^0\big([L(0)_{nil}^l \cU(1)w_1](j,k;n)\cU(1)w_2,1\big)w_3\right\rangle (e^{2\pi i z}-1)^n\nonumber\\
 &\qquad=\sum_{n\in\ZZ} \left\langle w_4',\cY^0\big([\cU(q_{z})\til{w}_1^{(l)}](j,k;n)\cU(1)w_2,1\big)w_3\right\rangle(e^{2\pi i z}-1)^{n},
\end{align}
with $\til{w}_1^{(l)}=\cU(q_{z})^{-1} L(0)_{nil}^l\cU(1)w_1$, converges absolutely just as \eqref{eqn:gjk_def} does. From the absolute convergence of \eqref{eqn:double_abs_conv}, it is easy to show that \eqref{eqn:gjk_def} converges uniformly on compact subsets of the region $0<\vert q_{z}-1\vert<1$ and therefore converges to a single-valued analytic function. 

The coefficients of the Laurent series expansion of $g_{j,k}(z)$ about $z=0$ can be computed by contour integrals, and by uniform convergence, these contour integrals commute with the sum in \eqref{eqn:gjk_def}. Thus for $z$ close enough to $0$, say $0<\vert z\vert<\varepsilon$, $g_{j,k}(z)$ equals the (absolutely-convergent) series obtained by rearranging \eqref{eqn:gjk_def} into a (well-defined) Laurent series in $z$. Returning to \eqref{eqn:gjk_obtain}, $(e^{2\pi i z}-1)^{h_j}$ and $\log(e^{2\pi i z}-1)$ also have expansions as series in powers of $z$ and $\log z$ that converge absolutely for small enough $\vert z\vert$. Thus for $0<\vert z\vert<\varepsilon$ with $\varepsilon$ small,
 the series \eqref{eqn:gjk_obtain} agrees with its rearrangement as a well-defined and absolutely-convergent series in powers of $z$ and $\log z$. However, for comparing with \eqref{eqn:PY_1Y_2_calc}, we must be careful to use the correct series expansion of $\log(e^{2\pi i z}-1)$ that yields the principal branch of logarithm.
 
 We now return to fixing $z=ir$ as in the beginning of the proof; we may assume in addition that $\vert z\vert =-r<\varepsilon$. As our assumptions on $r$ imply that $-2\pi r>0$, the principal branch of logarithm equals:
\begin{align}\label{eqn:ln_calc}
\log(e^{2\pi i z}-1) & = \ln\left(-2\pi r\cdot \frac{e^{-2\pi r}-1}{-2\pi r}\right) =\ln 2\pi +\ln(-r)+\ln\left(\frac{e^{-2\pi r}-1}{-2\pi r}\right).
\end{align}
Then because $0<-2\pi r<\ln 2$ and $\frac{e^x-1}{x}$ is increasing for $0<x<\ln 2$, we have
\begin{equation*}
0< \frac{e^{-2\pi r}-1}{-2\pi r}<\frac{1}{\ln 2}<2,
\end{equation*}
so we can use the standard power series for $\log(1+y)$ at $y=\frac{e^{-2\pi r}-1}{-2\pi r}-1$ to evaluate the third term in \eqref{eqn:ln_calc}. Therefore,
\begin{align*}
 \log(e^{2\pi i z}-1) &  = \ln 2\pi +\ln(-r)+\log(1+y)\big\vert_{y=\frac{e^{-2\pi r}-1}{-2\pi r}-1}\nonumber\\
 & =\left(\ln 2\pi+i\,\frac{\pi}{2}\right)+\left(\ln(-r)-i\,\frac{\pi}{2}\right)+\log(1+y)\big\vert_{y=\frac{e^{-2\pi r}-1}{-2\pi r}-1}\nonumber\\
 & = \log 2\pi i+\log z+\log(1+y)\big\vert_{y=\frac{e^{2\pi i z}-1}{2\pi i z}-1}.
\end{align*}
This is the correct series expansion to use in \eqref{eqn:gjk_obtain}, given $z=i r$; it agrees with the expansion used for the formal series $\log(e^{2\pi ix}-1)$ in \eqref{eqn:U(1)_intw_op}.

The above discussion now shows that for $z=ir$ such that $0<\vert q_z-1\vert<1<\vert q_z\vert$ and $0<\vert z\vert<\varepsilon$, we can rewrite \eqref{eqn:PY_1Y_2_calc} as
\begin{align*}
 P_{\cY_1,\cY_2} & (w_4',w_1,w_2,w_3;z)\nonumber\\
 &=\left.\left\langle w_4',\cY^0\big(\cY_{W_1,W_2}(e^{2\pi i x L(0)}\cU(1)w_1,e^{2\pi i x}-1)\cU(1)w_2,1\big)w_3\right\rangle\right\vert_{\log x=\log z},
\end{align*}
where the right side is a (well-defined and absolutely-convergent) series in powers of $z$ and $\log z$, expanded using the convention described in the paragraph following \eqref{eqn:U(1)_intw_op}. By \eqref{eqn:U(1)_intw_op}, this is precisely the series \eqref{eqn:I_series} that defines $I_{\cY^0,\cY_{W_1,W_2}}$. In particular, \eqref{eqn:I_series} converges absolutely for all $z=ir$ such that $0<\vert z\vert<\varepsilon$, for some sufficiently small $\varepsilon$; from this we see (using \cite[Lemma 7.7]{HLZ5} and its proof, for example) that $I_{\cY^0,\cY_{W_1,W_2}}$ converges absolutely to a (single-valued) analytic function on the region $0<\vert z\vert<\varepsilon$, $\arg z\neq \pi$. Moreover, for $z=ir$,
\begin{equation}\label{eqn:P_equals_I}
 P_{\cY_1,\cY_2}(w_4',w_1, w_2, w_3;z) =I_{\cY^0,\cY_{W_1,W_2}}(w_4',w_1,w_2,w_3;z)
\end{equation}
for all $w_4'\in W_4'$, $w_1\in W_1$, $w_2\in W_2$, $w_3\in W_3$. Since $P_{\cY_1,\cY_2}$ and $I_{\cY^0,\cY_{W_1,W_2}}$ are both analytic on their respective domains, and since \eqref{eqn:P_equals_I} holds on a subset of $\CC$ with an accumulation point, $I_{\cY^0,\cY_{W_1,W_2}}$ is the unique analytic extension of $P_{\cY_1,\cY_2}$ to some simply-connected domain that contains the region $0<\vert z\vert <\varepsilon$, $\arg z\neq \pi$. 

Finally, we need to show that the series defining $I_{\cY^0,\cY_{W_1,W_2}}$ is absolutely convergent for $0<\vert z\vert<1$.  It is enough to show that each Laurent series
\begin{equation*}
 f_{j,k}(z)= \big\langle w_4',\cY^0\big(\cU(1)\cY_k^{(j)}(w_1,z)w_2,1\big)w_3\big\rangle
\end{equation*}
converges for $0<\vert z\vert<1$; at first, we only know that $f_{j,k}(z)$ is convergent for $0 <\vert z\vert<\varepsilon$. We use \cite[Equations 3.113 and 3.115]{HLZ2}, which show that
\begin{align}\label{eqn:HLZ2_lin_comb}
 x^{h_j}\cY^{(j)}_k(w_1,x)w_2=\sum_{t=k}^K (-1)^{k+t}\binom{t}{k}(\log x)^{t-k}\cY^{(j,t)}(w_1,x)w_2,
\end{align}
where $\cY^{(j,t)}$ is a linear combination of intertwining operators of the form $L(0)_{nil}^l\circ\cY^{(j)}\circ(L(0)_{nil}^m\otimes L(0)_{nil}^n)$ for $l+m+n=t$. Since $\cY^{(j,t)}=f^{(t)}\circ\cY_{W_1,W_2}$ for a unique $f^{(t)}\in\mathrm{End}_V(W_1\tens W_2)$, and since $f^{(t)}$ commutes with $\cU(1)$, we get
\begin{align*}
 f_{j,k}(z) &=e^{-h_j\log z}\sum_{t=k}^K (-1)^{k+t}\binom{t}{k}(\log z)^{t-k}\big\langle w_4',\cY^0\big(\cU(1)\cY^{(j,t)}(w_1,e^{\log z})w_2,1\big)w_3\big\rangle\nonumber\\
 & =e^{-h_j\log z}\sum_{t=k}^K (-1)^{k+t}\binom{t}{k}(\log z)^{t-k} I_{\cY^0\circ(f^{(t)}\otimes\Id_{W_3}),\cY_{W_1,W_2}}(w_4',w_1,w_2,w_3;z)
\end{align*}
for $0<\vert z\vert<\varepsilon$.

By associativity of intertwining operators, each $I_{\cY^0\circ(f^{(t)}\otimes\Id_{W_3}),\cY_{W_1,W_2}}$ is a branch of the multivalued analytic function $\overline{P}_{\til{\cY}_t,\cY_{W_2,W_3}}$ for some intertwining operator $\til{\cY}_t$ of type $\binom{W_4}{W_1\,W_2\tens W_3}$. Since each $\overline{P}_{\til{\cY}_t,\cY_{W_2,W_3}}$ is analytic in particular on the region $0<\vert z\vert<1$, $f_{j,k}(z)$ has a multivalued extension $F_{j,k}(z)$ to this region. But in fact $F_{j,k}(z)$ is single valued because it is determined by a single-valued branch defined on a full punctured disk. Thus $f_{j,k}(z)$ is the Laurent series expansion of the analytic function $F_{j,k}(z)$ about its singularity $z=0$; since $F_{j,k}(z)$ is analytic on the punctured disk $0<\vert z\vert<1$, $f_{j,k}(z)$ converges absolutely for $0<\vert z\vert<1$ as required.
\end{proof}

\subsection{Genus-one correlation functions}

We now begin to consider genus-one correlation functions associated to $V$-module intertwining operators. Let 
\begin{equation*}
 \mathbb{H}=\lbrace \tau\in\CC\,\,\vert\,\,\mathrm{Im}\,\tau>0\rbrace
\end{equation*}
denote the upper half plane. For $\cY$ an intertwining operator of type $\binom{X}{W\,X}$, it was shown in \cite{Mi1, Hu-mod_inv, Fi} that the series
\begin{equation*}
 F_\cY(w;\tau)=\tr_{X} \cY(\cU(q_z)w,q_z)q_\tau^{L(0)-c/24},
\end{equation*}
where $c$ is the central charge of $V$, is independent of $z$ and converges absolutely for $0<\vert q_\tau\vert<1$. In particular, $F_\cY$ is an analytic function on $\mathbb{H}$. When $\cY$ is the vertex operator $Y_X$ for a $V$-module $X$, then we get the \textit{character}
\begin{equation}\label{eqn:ch_def}
 \mathrm{ch}_X(v;\tau) = \tr_X Y_X(\cU(q_z)v,q_z)q_\tau^{L(0)-c/24} =\sum_{h\in\QQ} \tr_{X_{[h]}} o(\cU(1)v) q_\tau^{h-c/24}
\end{equation}
for $v\in V$, where $o(\bullet)=\mathrm{Res}_x\,x^{-1} Y_X(x^{L(0)}(\bullet),x)$ denotes the grading-preserving component of the vertex operator. The characters of the distinct irreducible $V$-modules are linearly independent as functions on $V\times\mathbb{H}$ (see \cite[Theorem 5.3.1]{Zh} and also \cite[Lemma 5.10]{CM}).

\begin{rem}
 The function $F_\cY(w;\tau)$ is a one-point correlation function associated to the torus $\CC/(\ZZ+\ZZ\tau)$ with a puncture at $z$, and with standard local coordinate $w\mapsto w-z$ at the puncture. The lack of dependence on $z$ in $F_\cY$ reflects the fact that the conformal automorphism group of the torus acts transitively on punctures with standard local coordinates.
\end{rem}

 Since the trace of an operator on a finite-dimensional vector space is the same as the trace of its semisimple part, the nilpotent part of $L(0)$ (if any) does not appear in the characters $\mathrm{ch}_X(v;\tau)$. Thus the expansion of $\mathrm{ch}_X(v;\tau)$ as a series in $q_\tau$ about the singularity $q_\tau=0$ in \eqref{eqn:ch_def} indeed involves no $\log q_\tau$ terms. When $V$ is $C_2$-cofinite but non-rational, the nilpotent part of $L(0)$ acting on logarithmic $V$-modules can be detected by Miyamoto's pseudo-trace functions \cite{Mi2}; see also \cite{Arike, AN, Fi}. In this paper, pseudo-traces will be necessary only for one main result and its applications to non-rational vertex operator algebras, so here we add only a brief discussion to establish notation, using the definition of pseudo-trace functions from \cite{AN, Fi}. The reader who is only interested in rationality theorems for $C_2$-cofinite vertex operator algebras may safely assume that all pseudo-traces below are actually traces (see especially Remark \ref{rem:pt_not_necessary} below).
 
 Suppose $P$ is a finite-dimensional associative algebra and $\phi: P\rightarrow\CC$ is a symmetric linear function, that is, $\phi(ab)=\phi(ba)$ for all $a,b\in P$. Then for any finitely-generated projective $P$-module $X$, there is a pseudo-trace map defined as follows:
 \begin{align*}
  \tr_X^\phi: \Endo_P(X) & \rightarrow \CC\nonumber\\
  f & \mapsto \sum_{i=1}^I (\phi\circ b_i'\circ f)(b_i),
 \end{align*}
where $\lbrace b_i\rbrace_{i=1}^I$ is a projective basis of $X$ and $\lbrace b_i'\rbrace_{i=1}^I\subseteq\hom_P(X,P)$ is the corresponding dual basis. Pseudo-traces associated to $\phi$ enjoy the same cyclic symmetry of traces, that is, if $f: X_1\rightarrow X_2$ and $g: X_2\rightarrow X_1$ are $P$-module homomorphisms with $X_1$ and $X_2$ finitely-generated projective $P$-modules, then
\begin{equation*}
 \tr_{X_1}^\phi (g\circ f) =\tr_{X_2}^\phi (f\circ g).
\end{equation*}
For $P=\CC$ and $\phi=\Id_\CC$, pseudo-traces are simply traces of linear maps.

 Now suppose $X$ is a grading-restricted generalized $V$-module and the associative algebra $P$ acts on $X$ by $V$-module endomorphisms. Suppose also that $\cY$ is a $V$-module intertwining operator of type $\binom{X}{W\,X}$ such that $\cY(w,x)$ commutes with the $P$-action on $X$ for all $w\in W$, that is, every coefficient of $\cY(w,x)$ is a $P$-module endomorphism of $X$. If $X$ is projective as a $P$-module, then its conformal weight spaces $X_{[h]}$ are finitely-generated projective $P$-modules and we can define the pseudo-trace function
 \begin{align}\label{eqn:pseudo_tr_def}
  F_\cY^\phi(w;\tau)&  = \tr_X^\phi\cY(\cU(1)w,1)q_\tau^{L(0)-c/24} =\sum_{h\in\QQ} \tr_{X_{[h]}}^\phi o(\cU(1)w)q_\tau^{L(0)-c/24}
 \end{align}
for any symmetric linear function $\phi:P\rightarrow\CC$.

We can define two-point genus-one correlation functions using pseudo-traces of $V$-module intertwining operators similarly. Suppose $W_1$, $W_2$, $M$, and $X$ are grading-restricted generalized $V$-modules such that $M$ and $X$ are also modules for some associative algebra $P$ acting by $V$-module endomorphisms, with $X$ a projective $P$-module. Suppose also that $\cY_1$ and $\cY_2$ are intertwining operators of types $\binom{X}{W_1\,M}$ and $\binom{M}{W_2\,X}$, respectively, that commute with the $P$-actions on $M$ and $X$. Then for a symmetric linear function $\phi$ on $P$, \cite[Propositions 2.8 and 2.14]{Fi} show that the series
\begin{equation}\label{eqn:two_point_trace}
 \tr_{X}^\phi \cY_1(\cU(q_{z_1})L(0)_{nil}^{k_1}w_1,q_{z_1})\cY_2(\cU(q_{z_2})L(0)_{nil}^{k_2}w_2,q_{z_2})q_\tau^{L(0)-c/24},
\end{equation}
for $k_1,k_2\in\NN$, where $L(0)_{nil}$ is the nilpotent part of $L(0)$, satisfy a finite system of linear differential equations with regular singular points. As a consequence, these series converge absolutely to analytic functions when $1>\vert q_{z_1}\vert>\vert q_{z_2}\vert>\vert q_\tau\vert>0$, equivalently $0<\mathrm{Im}\,z_1<\mathrm{Im}\,z_2<\mathrm{Im}\,\tau$, and these analytic functions can be extended to multivalued analytic functions on the region
\begin{equation*}
 \lbrace (z_1,z_2,\tau)\in\CC^2\times\mathbb{H}\,\vert\, z_1-z_2\notin\ZZ+\ZZ\tau\rbrace.
\end{equation*}
See also \cite{Hu-mod_inv} for the case of ordinary traces of non-logarithmic intertwining operators.

 It turns out that \eqref{eqn:two_point_trace} actually converges on a larger set than indicated above: 
\begin{lem}\label{lem:prin_branch_domain}
The series \eqref{eqn:two_point_trace} converges absolutely for all $(z_1,z_2,\tau)\in\CC^2\times\mathbb{H}$ such that  $0<\im(z_2-z_1)<\im\tau$.
\end{lem}
\begin{proof}
It is sufficient to consider the $k_1=k_2=0$ case of \eqref{eqn:two_point_trace}. If $z_1$, $z_2$, and $\tau$ satisfy the conditions in the lemma, then for $\varepsilon\in\mathbb{H}$ with sufficiently small imaginary part, we have
 \begin{equation*}
  0<\mathrm{Im}\,\varepsilon<\mathrm{Im}(z_2-z_1+\varepsilon)<\mathrm{Im}\,\tau.
 \end{equation*}
We then use the $L(0)$-conjugation formula for intertwining operators and the cyclic symmetry of pseudo-traces to prove the following equality of double series indexed by the conformal weights of $X$ and $M$:
\begin{align*}
 \tr_{X}^\phi &\, \cY_1(\cU(q_{\varepsilon})w_1,q_\varepsilon)\cY_2(\cU(q_{z_2-z_1+\varepsilon})w_2,q_{z_2-z_1+\varepsilon})q_\tau^{L(0)-c/24}\nonumber\\
 & =\sum_{h,m\in\QQ} \tr_{X_{[h]}}^\phi \cY_1(\cU(q_{\varepsilon})w_1,q_\varepsilon)\pi_m\cY_2(\cU(q_{z_2-z_1+\varepsilon})w_2,q_{z_2-z_1+\varepsilon})q_\tau^{L(0)-c/24}\nonumber\\
 & =\sum_{h,m\in\QQ} \tr_{X_{[h]}}^\phi q_{-z_1+\varepsilon}^{L(0)}\cY_1(\cU(q_{z_1})w_1,q_{z_1})\pi_m\cY_2(\cU(q_{z_2})w_2,q_{z_2})q_{z_1-\varepsilon}^{L(0)}q_{\tau}^{L(0)-c/24}\nonumber\\
 & =\sum_{h,m\in\QQ} \tr_{X_{[h]}}^\phi \cY_1(\cU(q_{z_1})w_1,q_{z_1})\pi_m\cY_2(\cU(q_{z_2})w_2,q_{z_2})q_{\tau}^{L(0)-c/24}\nonumber\\
 & =\tr_{X}^\phi \cY_1(\cU(q_{z_1})w_1,q_{z_1})\cY_2(\cU(q_{z_2})w_2,q_{z_2})q_{\tau}^{L(0)-c/24}.
\end{align*}
Since the two double series have exactly the same terms and the first converges absolutely by \cite[Theorem 4.1]{Hu-mod_inv}, \cite[Proposition 2.17]{Fi}, the second converges absolutely as well.
\end{proof}

Lemma \ref{lem:prin_branch_domain} implies that $z_2$ in \eqref{eqn:two_point_trace} can be arbitrary; as in the previous subsection, we will always choose $z_2=0$ for simplicity and then set $z_1=z$. Thus the series
\begin{equation*}
 F_{\cY_1,\cY_2}^\phi(w_1,w_2;z,\tau) = \tr_{X}^\phi \cY_1(\cU(q_{z})w_1,q_{z})\cY_2(\cU(1)w_2,1)q_\tau^{L(0)-c/24}
\end{equation*}
converges absolutely to a single-valued analytic function on the simply-connected basic genus-one convergence region
\begin{equation*}
 U_1^{prod} =\lbrace (z,\tau)\in\CC\times\mathbb{H}\,\,\vert\,\,0<\im(-z)<\im\tau\rbrace.
\end{equation*}
By the differential equations in \cite{Hu-mod_inv,Fi}, $F_{\cY_1,\cY_2}^\phi$ extends to a multivalued analytic function $\overline{F}^\phi_{\cY_1,\cY_2}$ on the region
\begin{equation*}
 U_1=\lbrace (z,\tau)\in\CC\times\mathbb{H}\,\,\vert\,\,z\notin\ZZ+\ZZ\tau\rbrace.
\end{equation*}
We call the single-valued analytic function $F_{\cY_1,\cY_2}^\phi(w_1,w_2;z,\tau)$ on the simply-connected domain $U_1^{prod}$ the principal branch of the multivalued function $\overline{F}_{\cY_1,\cY_2}^\phi$. We will need to expand $\overline{F}_{\cY_1,\cY_2}^\phi$ as series in $z$ in open domains near certain singularities. Considering $z=0$ first, we expect $\overline{F}_{\cY_1,\cY_2}^\phi$ to have a series expansion that converges on any punctured disk $0<\vert z\vert<R$ that avoids the non-zero singularities. Thus for $\tau\in\mathbb{H}$, we define
\begin{equation*}
 R_\tau=\min\left\lbrace\vert m+n\tau\vert\,\,\vert\,\,(m,n)\in\ZZ^2\setminus(0,0)\right\rbrace
\end{equation*}
and then
\begin{equation*}
 U_1^{it}=\left\lbrace (z,\tau)\in\CC\times\mathbb{H}\,\,\vert\,\,0<\vert z\vert<R_\tau, \arg z\neq\pi\right\rbrace.
\end{equation*} 
As in the previous subsection, $\cY_1=f_1\circ\cY_{W_1,M}$ and $\cY_2=f_2\circ\cY_{W_2,X}$ for unique $V$-module homomorphisms $f_1$ and $f_2$. Now we prove the following using Proposition \ref{prop:geom_mod_assoc}, enhancing the genus-one associativity result of \cite[Theorem 4.3]{Hu-mod_inv} and \cite[Proposition 2.19]{Fi}:
\begin{prop}\label{prop:genus_1_assoc}
The analytic extension of $F_{\cY_1,\cY_2}^\phi(w_1,w_2;z,\tau)$ from $U_1^{prod}$ to the simply-connected domain $U_1^{it}$
is given by
\begin{align*}
 G_{\cY^0,\cY_{W_1,W_2}}^\phi  ( & w_1,w_2;z,\tau) =\tr_{X}^\phi \cY^0\left(\cU(1)\cY_{W_1,W_2}(w_1,e^{\log z})w_2,1\right) q_\tau^{L(0)-c/24}\nonumber\\
 & :=\sum_{h\in\QQ}\sum_{k=0}^K\left.\left(\tr_{X}^\phi \cY^0\left(\cU(1)(w_1\tens_{h,k}w_2),1\right)q_\tau^{L(0)-c/24}\right) x^{-h-1}(\log x)^k\right\vert_{\log x=\log z}
\end{align*}
for $(z,\tau)\in U_1^{it}$, where $\cY^0=f_1\circ(\Id_{W_1}\tens f_2)\circ\cA_{W_1,W_2,X}^{-1}\circ\cY_{W_1\tens W_2,X}$.
\end{prop}

\begin{proof}
 From the definition of $F_{\cY_1,\cY_2}^\phi$ and Proposition \ref{prop:geom_mod_assoc}, we have
 \begin{align}\label{eqn:genus_1_alpha_calc}
  F_{\cY_1,\cY_2}^\phi(w_1,w_2; z,\tau) &=\sum_{h\in\QQ}\tr_{X_{[h]}}^\phi \cY_1(\cU(q_{z})w_1,q_{z})\cY_2(\cU(1)w_2,1)q_\tau^{L(0)-c/24}\nonumber\\
  & =\sum_{h\in\QQ} \tr_{X_{[h]}}^\phi \cY^0\left(\cU(1)\cY_{W_1,W_2}(w_1,e^{\log z})w_2,1\right)q_{\tau}^{L(0)-c/24},
 \end{align}
 for $(z,\tau)\in U_1^{prod}\cap U_1^{it}$. This is an absolutely convergent series in powers of $q_\tau$, and possibly also $\log q_\tau$ if $L(0)$ acts non-semisimply on $X$, whose coefficients are analytic functions in $z$ with series expansion given in Proposition \ref{prop:geom_mod_assoc}. The coefficient of each power of $q_\tau$ and $\log q_\tau$ on the right side of \eqref{eqn:genus_1_alpha_calc} converges as long as $0<\vert z\vert<1$, independently of $\tau$. To prove the proposition, we need to rearrange the right side of \eqref{eqn:genus_1_alpha_calc} as a convergent series in powers of $z$ and $\log z$ whose coefficients are analytic in $\tau$. To justify this, we will show that the corresponding double sum in powers of $z$ and $q_\tau$ converges absolutely for $\vert z\vert$ and $\vert q_\tau\vert$ sufficiently small.

To begin, we view the right side of \eqref{eqn:genus_1_alpha_calc} as a formal series in powers of $q$ and $\log q$ whose coefficients are analytic functions of $z$ on the region $U_0^{it}$. These coefficients can be analytically extended across the branch cut in $U_0^{it}$, so we can also view them as multivalued functions on the punctured unit disk. We use $F^\phi_{\cY_1,\cY_2}(w_1,w_2;z,q)$ and $G_{\cY^0,\cY_{W_1,W_2}}^\phi(w_1,w_2;z,q)$ to denote the formal $q$-series determined by the left and right sides of \eqref{eqn:genus_1_alpha_calc}, respectively.

Next, from the differential equations in \cite[Theorem 3.9]{Hu-mod_inv} and \cite[Propositions 2.8 and 2.14]{Fi}, there exists $m\in\ZZ_+$ such that the finite set
\begin{equation*}\label{eqn:derivative_set}
 \left\lbrace \left( q\frac{\partial}{\partial q}\right)^i \frac{\partial^j}{\partial z^j}F_{\cY_1,\cY_2}^\phi(L(0)_{nil}^{k_1} w_1, L(0)_{nil}^{k_2} w_2; z, q)\right\rbrace_{0\leq i,j < m,\, 0\leq k_1,k_2\leq K},
\end{equation*}
with $K$ defined as in Remark \ref{rem:K_def}, form the components of a vector solution $\Phi(w_1,w_2;z,q)$ to a first-order system
\begin{equation}\label{eqn:diff_for_F}
 q\frac{\partial}{\partial q}\Phi = A(z, q)\Phi.
\end{equation}
Here $A(z,q)$ is a matrix whose entries are built from $q$-series expansions of Eisenstein series, Weierstrass $\wp$-functions $\wp_m(z,q)$, and their derivatives. Thus the coefficients of powers of $q$ in $A(z,q)$ are analytic for $z\in U_0^{it}$, since the coefficients of powers of $q$ in each $\wp_m(z,q)$ converge in this region (see for example \cite[Equation 2.11]{Hu-mod_inv}). By the following standard analyticity result for regular-singular-point differential equations, formal $q$-series solutions to \eqref{eqn:diff_for_F} converge to analytic solutions; for a proof adapted from the proof of \cite[Theorem 24.III]{Wa}, see Appendix A in the expanded version of this paper at \texttt{arXiv:2108.01898v2}:
\begin{thm}\label{thm:diff_eqn}
 Suppose a matrix $A(z,q)$ as in \eqref{eqn:diff_for_F} has entries which are jointly analytic in $z$ and $q$ for $z$ contained in some open subset $U\subseteq\CC$ and for $q$ contained in a disk $B_r(0)=\lbrace q\in\CC\,\,\vert\,\,\vert q\vert<r\rbrace$. Suppose also that $\Phi(z,q)$ is a formal solution to \eqref{eqn:diff_for_F} of the form
 \begin{equation*}
  \Phi(z,q)=\sum_{j=1}^J\sum_{k=0}^K\sum_{n\geq 0}\varphi_{j,k,n}(z)\,q^{h_j+n}(\log q)^k
 \end{equation*}
where the $h_j\in\CC$ are non-congruent modulo $\ZZ$ and each $\varphi_{j,k,n}(z)$ is a vector-valued analytic function on $U$. Then for each $1\leq j\leq J$ and $0\leq k\leq K$, the series $\sum_{n\geq 0} \varphi_{j,k,n}(z)\,q^n$ converges absolutely on $U\times B_r(0)$ to a function jointly analytic in $z$ and $q$.
\end{thm}

Now by \eqref{eqn:genus_1_alpha_calc}, $G_{\cY^0,\cY_{W_1,W_2}}^\phi(w_1,w_2;z,q)$ is a component of a formal solution to \eqref{eqn:diff_for_F} for $z\in U_0^{prod}\cap U_0^{it}$. But in fact this is a formal solution for all $z\in U_0^{it}$ since the coefficients of powers of $q$ in $G^\phi_{\cY^0,\cY_{W_1,W_2}}$, its partial derivatives, and $A(z,q)$ are defined and analytic on this connected region. Thus by Theorem \ref{thm:diff_eqn}, the $q$-series $G_{\cY^0,\cY_{W_1,W_2}}^\phi$ converges absolutely in a punctured disk surrounding $q=0$ for any fixed $z\in U_0^{it}$, with radius of convergence at least the minimum $R\leq 1$ such that the punctured disk $0<\vert q\vert<R$ does not contain any singularities of the Weierstrass $\wp$-functions, which are
\begin{equation*}
 \left\lbrace (z,q=q_\tau)\,\vert\,z\in\ZZ+\ZZ\tau\right\rbrace.
\end{equation*}
In particular, for any fixed $0<\delta<1$, $G_{\cY^0,\cY_{W_1,W_2}}^\phi$ converges absolutely in the polyannulus
 \begin{equation*}\label{eqn:polyannulus}
PA_\delta=  \bigg\lbrace (z,q)\in\CC^2\,\,\bigg\vert\,\,0<\vert z\vert< \min\bigg(1,-\frac{1}{2\pi}\ln\delta\bigg),\,0<\vert q\vert<\delta\bigg\rbrace.
 \end{equation*}
 Moreover, $G_{\cY^0,\cY_{W_1,W_2}}^\phi(w_1,w_2;z,q)$ is analytic in $z$ as well as in $q$ on this polyannulus.

 Now if $G_{\cY^0,\cY_{W_1,W_2}}^\phi$ were a Laurent series in $q$ and $z$, we could immediately conclude that the double sum corresponding to the right side of \eqref{eqn:genus_1_alpha_calc} converges absolutely, and we would be justified in reversing the order of summation for suitable $q$ and $z$. To deal with the problem that $G_{\cY^0,\cY_{W_1,W_2}}^\phi$ is not a Laurent series, we borrow notation from \eqref{eqn:Y_tens_decomp} to write
 \begin{align*}
 G_{\cY^0,\cY_{W_1,W_2}}^\phi & (w_1,w_2;z,q) =\sum_{j=1}^{J}\sum_{k=0}^K e^{h_j\log z}(\log z)^k f_{j,k}(z,q),
\end{align*} 
 where 
 \begin{align*}
  f_{j,k}(z,q)=\tr^\phi_X \cY^0\left(\cU(1)\cY^{(j)}_k(w_1,z)w_2,1\right)q^{L(0)-c/24}
 \end{align*}
involves only integral powers of $z$. Then by the argument in the proof of Proposition \ref{prop:geom_mod_assoc} beginning with \eqref{eqn:HLZ2_lin_comb},
\begin{equation}\label{eqn:fjk_genus_one}
 f_{j,k}(z,q)=e^{-h_j\log z}\sum_{t=k}^K (-1)^{k+t}\binom{t}{k}(\log z)^{t-k} G^\phi_{\cY^0\circ(f^{(t)}\otimes\Id_X),\cY_{W_1,W_2}}(w_1,w_2;z,q)
\end{equation}
for suitable $f^{(t)}\in\Endo_V(W_1\tens W_2)$, and the preceding arguments show $G^\phi_{\cY^0\circ(f^{(t)}\otimes\Id_X),\cY_{W_1,W_2}}$ converges absolutely on $PA_\delta$ to an analytic function in both $q$ and $z$. Moreover, by restricting $\cY^0$ to direct summands of $X$ whose minimal conformal weights $h_{j'}$ for $1\leq j'\leq J'$ are non-congruent mod $\ZZ$, we have an expansion
\begin{align}\label{eqn:Gt}
 G^\phi_{\cY^0\circ(f^{(t)}\otimes\Id_X),\cY_{W_1,W_2}}(w_1,w_2;z,q) =\sum_{j'=1}^{J'}\sum_{k'=0}^K g_{j',k';t}(w_1,w_2;z,q) q^{h_{j'}-c/24}(\log q)^{k'}
\end{align}
where
\begin{equation*}
 g_{j',k';t}(w_1,w_2;z,q)=\sum_{n\geq 0}\bigg(\tr_{X_{[h_{j'}+n]}}^\phi \cY^0_{j'}\left(\cU(1)(f^{(t)}\circ\cY_{W_1,W_2})(w_1,e^{\log z})w_2,1\right)\frac{L(0)^{k'}}{k'!}\bigg) q^n;
\end{equation*}
the $\cY^0_{j'}$ are restrictions of $\cY^0$ to direct summands of $X$. By Theorem \ref{thm:diff_eqn}, each $g_{j',k';t}$ converges absolutely to an analytic function in $q$ and $z$ on $PA_\delta$, so by combining \eqref{eqn:fjk_genus_one} and \eqref{eqn:Gt} and rearranging the resulting absolutely-convergent series, we get
\begin{equation*}
 f_{j,k}(z,q)=\sum_{j'=1}^{J'}\sum_{k'=0}^K g_{j,k;j',k'}(w_1,w_2;z,q)\, q^{h_{j'}-c/24}(\log q)^{k'},
\end{equation*}
where
\begin{equation}\label{eqn:gjkj'k'}
 g_{j,k;j',k'}(w_1,w_2;z,q)=\sum_{n\geq 0}\bigg( \tr_{X_{[h_{j'}+n]}}^\phi\cY_{j'}^0\left(\cU(1)\cY_k^{(j)}(w_1,z)w_2,1\right)\frac{L(0)^{k'}}{k'!}\bigg) q^n
\end{equation}
is a power series in $q$, whose coefficients are Laurent series in $z$, that converges absolutely to a (single-valued) analytic function in $q$ and $z$ on the polyannulus \eqref{eqn:polyannulus}.

Since $g_{j,k;j',k'}$ is an analytic function on a polyannulus, it has a Laurent series expansion, whose coefficients can be computed by contour integrals, which is absolutely convergent as a double sum. This Laurent series must be the double sum rearrangement of the iterated sum in powers of $q$ and $z$ on the right side of \eqref{eqn:gjkj'k'}. Since this double sum converges absolutely, we can reverse the order of summation in \eqref{eqn:gjkj'k'}; the definitions then yield
\begin{align}\label{eqn:genus_1_alpha_series}
&G^\phi_{\cY^0,\cY_{W_1,W_2}}(w_1,w_2;z,q)\nonumber\\
&\qquad=\sum_{h\in\QQ}\sum_{k=0}^K \left. \left(\tr_X^\phi\cY^0(\cU(1)(w_1\tens_{h,k} w_2),1)q^{L(0)-c/24}\right) x^{-h-1}(\log x)^k\right\vert_{\log x=\log z}
\end{align}
for $(z,q)\in PA_\delta$ for any $0<\delta<1$.

We still need to determine the minimum possible radius of convergence for \eqref{eqn:genus_1_alpha_series} when $q=q_\tau$ for any fixed $\tau\in\mathbb{H}$. So far, we have shown convergence when $0<\vert z\vert<\min\left(1,-\frac{1}{2\pi}\ln\delta\right)$ for any $\delta$ such that $\vert q_\tau\vert<\delta<1$. But using the identity \eqref{eqn:fjk_genus_one} for each $f_{j,k}(z,q_\tau)$, viewed as a Laurent series in powers of $z$, the argument that concludes Proposition \ref{prop:geom_mod_assoc} shows that $f_{j,k}(z,q_\tau)$ converges absolutely on any punctured disk $0<\vert z\vert <R$ that does not intersect the potential singularity set $\ZZ+\ZZ\tau$ for two-point genus-one correlation functions. That is, the series \eqref{eqn:genus_1_alpha_series} converges for $0<\vert z\vert< R_\tau$. Finally, since $U_1^{prod}\cap U_1^{it}$ contains $(z,\tau)$ with arbitrarily small $\vert q_\tau\vert$ and $\vert z\vert$, \eqref{eqn:genus_1_alpha_calc} and   \eqref{eqn:genus_1_alpha_series} show that there is a non-empty open subset of $U_1^{prod}\cap U_1^{it}$ on which $G_{\cY^0,\cY_{W_1,W_2}}^\phi(w_1,w_2;z,\tau)$ as defined in the statement of the proposition agrees with $F_{\cY_1,\cY_2}^\phi(w_1,w_2;z,\tau)$. Thus $G_{\cY^0,\cY_{W_1,W_2}}^\phi$ analytically continues $F_{\cY_1,\cY_2}^\phi$ to $U_1^{it}$.
\end{proof}

We can now use the preceding proposition to find series expansions of $\overline{F}_{\cY_1,\cY_2}^\phi$ about its singularities in $\ZZ$. The next result amounts to \cite[Equation 4.10]{Hu-Verlinde} and \cite[Lemma 5.12]{CM}; in its statement and proof, we use the notation $U_1^{it}-(n,0)$, for $n\in\ZZ$, to represent the translation of $U_1^{it}$ by $(-n,0)$, that is, $U_1^{it}-(n,0)=\lbrace (z,\tau)\in\CC\times\mathbb{H}\mid (z+n,\tau)\in U_1^{it}\rbrace$:
\begin{prop}\label{prop:genus_1_alpha}
 For $n\in\ZZ$, the analytic extension of the principal branch $F_{\cY_1,\cY_2}^\phi$ of $\overline{F}_{\cY_1,\cY_2}^\phi$ to the simply connected domain $U_1^{it}-(n,0)$ is given by 
 \begin{align*}
&  G_{\cY^n,\cY_{W_1,W_2}}^\phi  (w_1,w_2;z-n,\tau)=\tr_X^\phi\cY^n\left(\cU(1)\cY_{W_1,W_2}(w_1,e^{\log z})e^{2\pi i nL(0)}w_2,1\right)q_\tau^{L(0)-c/24}\nonumber\\
  &\qquad:=\sum_{h\in\QQ}\sum_{k=0}^K \left.\left( \tr_X^\phi\cY^n\left(\cU(1)(w_1\tens_{h,k} e^{2\pi i n L(0)}w_2),1\right)q_\tau^{L(0)-c/24}\right) x^{-h-1}(\log x)^k\right\vert_{\log x=\log z}.
 \end{align*}
for $(z,\tau)\in U_1^{it}$, where
\begin{equation*}
 \cY^n=f_1\circ(\Id_{W_1}\tens f_2)\circ(\Id_{W_1}\tens(\cR_{W_2,X}^2)^n)\circ\cA_{W_1,W_2,X}^{-1}\circ\cY_{W_1\tens W_2,X}.
\end{equation*}
\end{prop}
\begin{proof}
Suppose $(z,\tau)\in U_1^{prod}\cap U_1^{it}$; then $(z-n,\tau)$ is also in $U_1^{prod}$ and
 \begin{align*}
  F_{\cY_1,\cY_2}^\phi & (w_1,w_2;z-n,\tau)=\tr_X^\phi\cY_1(\cU(q_{z-n})w_1,q_{z-n})\cY_2(\cU(1)w_2,1)q_\tau^{L(0)-c/24}\nonumber\\
  & =\tr_X^\phi e^{-2\pi i nL(0)}\cY_1(\cU(q_z)w_1,q_z)e^{2\pi i nL(0)}\cY_2(\cU(1)w_2,1)q_\tau^{L(0)-c/24}\nonumber\\
  & =\tr_X^\phi \cY_1(\cU(q_z)w_1,q_z)\cY_2(e^{2\pi i n L(0)}\cU(1)w_2,e^{2\pi i n})q_\tau^{L(0)-c/24}\nonumber\\
  & =\tr_X^\phi \cY_1(\cU(q_{z})w_1,q_{z})(f_2\circ(\cR_{W_2,X}^2)^n\circ\cY_{W_2,X})(\cU(1)e^{2\pi inL(0)}w_2,1)q_\tau^{L(0)-c/24},
 \end{align*}
where the last step uses the characterization of the monodromy isomorphisms for $V$-modules and the fact that $\cU(1)$ commutes with the $V$-module homomorphism $e^{2\pi i L(0)}$. Now Proposition \ref{prop:genus_1_assoc} says that 
\begin{align*}
 F_{\cY_1,\cY_2}^\phi & (w_1,w_2;z-n,\tau)  =\tr_X^\phi\cY^n\left(\cU(1)\cY_{W_1,W_2}(w_1,e^{\log z})e^{2\pi i nL(0)}w_2,1\right)q_\tau^{L(0)-c/24}
\end{align*}
for $(z,\tau)\in U_1^{prod}\cap U_1^{it}$, where
\begin{equation*}
 \cY^n=f_1\circ(\Id_{W_1}\tens f_2)\circ(\Id_{W_1}\tens(\cR_{W_2,X}^2)^n)\circ\cA_{W_1,W_2,X}^{-1}\circ\cY_{W_1\tens W_2,X}
\end{equation*}
is an intertwining operator of type $\binom{X}{W_1\tens W_2\,X}$. In other words, $F_{\cY_1,\cY_2}^\phi(w_1,w_2;z-n,\tau)$ agrees with the series $G_{\cY^n,\cY_{W_1,W_2}}^\phi(w_1,w_2;z-n,\tau)$ defined in the statement of the proposition when $(z,\tau)\in U_1^{prod}\cap U_1^{it}$. Then since the pseudo-trace of the iterate of $\cY^n$ and $\cY_{W_1,W_2}$ (expanded as a series in $z$) is analytic on $U_1^{it}$ by Proposition \ref{prop:genus_1_assoc}, $G_{\cY^n,\cY_{W_1,W_2}}^\phi$ is analytic on $U_1^{it}-(n,0)$. Since moreover $G_{\cY^n,\cY_{W_1,W_2}}^\phi$ agrees with $F_{\cY_1,\cY_2}^\phi$ on a non-empty open set, it is the unique analytic extension of $F_{\cY_1,\cY_2}^\phi$ to $U_1^{it}-(n,0)$.
\end{proof}

We next consider the series expansion of a two-point genus-one correlation function at the singularity $z=-\tau$. Since this will be more complicated, we will only consider this problem for the particular $\cY_1$, $\cY_2$ that we shall need, and we will consider only traces rather than more general pseudo-traces. Thus we fix $V$-modules $W'$, $W$, $X$ and a $V$-module homomorphism $e_W: W'\tens W\rightarrow V$. For example, $W'$ could be the contragredient of $W$ and $e_W$ could be an evaluation map as in Section \ref{sec:contragredient}, although this assumption is not necessary at the moment. In any case we will as in Section \ref{sec:contragredient} use the notation $\cE_W=e_W\circ\cY_{W',W}$. We now define the intertwining operator
\begin{align*}
  \cY_{W',W\tens X}^X & =l_X\circ(e_W\tens\Id_X)\circ\cA_{W',W,X}\circ\cY_{W',W\tens X}
\end{align*}
of type $\binom{X}{W'\,W\tens X}$. Then similar to \cite[Definition 5.6]{CM}, we define
\begin{align}\label{eqn:ZWX}
 Z_{W,X}(w',w; z,\tau)&=\tr_X \cY_{W',W\tens X}^X(\cU(q_{z})w',q_{z})\cY_{W,X}(\cU(1)w,1) q_\tau^{L(0)-c/24},
\end{align}
which is the principal branch of a multivalued analytic function $\overline{Z}_{W,X}$. The definition of $\cY_{W',W\tens X}^X$ combined with Proposition \ref{prop:genus_1_assoc} quickly implies:
\begin{prop}\label{prop:Z_Z_til}
 The analytic continuation of $Z_{W,X}(w',w; z,\tau)$ to $U_1^{it}$ is
\begin{align*}
 \tr_X  Y_X & \left(\cU(1)\cE_W(w',e^{\log z})w,1\right)q_\tau^{L(0)-c/24}\nonumber\\
 &:=\sum_{h\in\QQ}\sum_{k=0}^K \left.\mathrm{ch}_X\left(e_W(w'\tens_{h,k} w);\tau\right)x^{-h-1}(\log x)^k\right\vert_{\log x=\log z}.
\end{align*}
 \end{prop}
 
 The next proposition, which is more or less \cite[Equation 4.11]{Hu-Verlinde}, or \cite[Lemma 5.13]{CM}, describes the analytic extension of $Z_{W,X}$ to a region about the singularity $z=-\tau$:
 
 \begin{prop}\label{prop:genus_1_beta}
  The series 
  \begin{equation*}
   Z_{W,X}(w',w;-z-\tau,\tau)=\tr_X\cY_{W',W\tens X}^X(\cU(q_{-z-\tau})w',q_{-z-\tau})\cY_{W,X}(\cU(1)w,1)q_\tau^{L(0)-c/24}
  \end{equation*}
converges absolutely to an analytic function for $(z,\tau)\in U_1^{prod}$. The analytic continuation of this function to $(z,\tau)\in U_1^{it}$ is given by
\begin{align*}
&\tr_{W\tens X}  \cY_{W\tens W',W\tens X}^{W\tens X}\left(\cU(1)\cY_{W,W'}(w,e^{\log z})w',1\right)q_\tau^{L(0)-c/24}\nonumber\\ &\quad:=\sum_{h\in\QQ}
 \sum_{k=0}^K\left.\left(\tr_{W\tens X} \cY_{W\tens W',W\tens X}^{W\tens X}\left(\cU(1)(w\tens_{h,k} w'),1\right)q_\tau^{L(0)-c/24}\right)x^{-h-1}(\log x)^k\right\vert_{\log x=\log z},
\end{align*}
where the intertwining operator $\cY_{W\tens W',W\tens X}^{W\tens X}$ of type $\binom{W\tens X}{W\tens W'\,W\tens X}$ is
\begin{equation*}
 (r_W\tens\Id_X)\circ((\Id_W\tens e_W)\tens\Id_X)\circ(\cA_{W,W',W}^{-1}\tens\Id_X)\circ\cA_{W\tens W',W,X}\circ\cY_{W\tens W',W\tens X}.
\end{equation*}
 \end{prop}
 \begin{proof}
 By Lemma \ref{lem:prin_branch_domain}, $Z_{W,X}(w',w;-z-\tau,\tau)$ converges when $(-z-\tau,\tau)\in U_1^{prod}$, that is,
 \begin{equation*}
  0<\im(z+\tau)<\im\,\tau.
 \end{equation*}
However, the two inequalities here are equivalent to
\begin{equation*}
 \im(-z)<\im\,\tau,\qquad \im(z)< 0,
\end{equation*}
which means that $(-z-\tau,\tau)\in U_1^{prod}$ if and only if $(z,\tau)\in U_1^{prod}$. So $Z_{W,X}(w',w;-z-\tau,\tau)$ converges absolutely to an analytic function on $U_1^{prod}$.
 
 To determine the analytic extension to $U_1^{it}$, we deduce the following equality of double series indexed by the weights of $X$ and $W\tens X$, using the $L(0)$-conjugation formula for intertwining operators and the cyclic symmetry of traces:
  \begin{align}\label{eqn:trace_prop_app}
   \tr_{X}  &\,\cY_{W',W\tens X}^X(\cU(q_{-z-\tau})w',q_{-z-\tau})\cY_{W,X}(\cU(1)w,1)q_\tau^{L(0)-c/24}\nonumber\\
   & =\sum_{h,m\in\QQ} \tr_{X_{[h]}} q_{z+\tau}^{-L(0)}\pi_h\cY_{W',W\tens X}^X(\cU(1)w',1)q_{z+\tau}^{L(0)}\pi_m\cY_{W,X}(\cU(1)w,1)q_\tau^{L(0)-c/24}\nonumber\\
  & =\sum_{h,m\in\QQ} \tr_{X_{[h]}}q_{z+\tau}^{-L(0)}\pi_h\cY_{W',W\tens X}^X(\cU(1)w',1)q_\tau^{L(0)-c/24}\pi_m\cY_{W,X}(\cU(q_{z})w,q_{z})q_{z+\tau}^{L(0)}\nonumber\\
  & =\sum_{h,m\in\QQ} \tr_{(W\tens X)_{[m]}} \pi_m\cY_{W,X}(\cU(q_{z})w,q_{z})\pi_h\cY_{W',W\tens X}^X(\cU(1)w',1)q_\tau^{L(0)-c/24}\nonumber\\
   & =\tr_{W\tens X} \cY_{W,X}(\cU(q_{z})w,q_{z})\cY_{W',W\tens X}^X(\cU(1)w',1)q_\tau^{L(0)-c/24}
  \end{align}
for $(z,\tau)\in U_1^{prod}$. By Proposition \ref{prop:genus_1_assoc}, the analytic extension of the right side trace to $U_1^{it}$ is
\begin{equation*}
 \tr_{W\tens X} \cY^0\left(\cU(1)\cY_{W,W'}(w, e^{\log z})w',1\right)q_\tau^{L(0)-c/24},
\end{equation*}
where $\cY^0=F\circ\cY_{W\tens W', W\tens X}$
with $F: (W\tens W')\tens(W\tens X)\rightarrow W\tens X$ given by the composition
\begin{align*}
 (W & \tens W')\tens(W\tens X) \xrightarrow{\cA_{W,W',W\tens X}^{-1}} W\tens(W'\tens(W\tens X))\nonumber\\
 &\xrightarrow{\Id_W\tens\cA_{W',W,X}} W\tens((W'\tens W)\tens X)\xrightarrow{\Id_W\tens(e_W\tens\Id_X)} W\tens(V\tens X)\xrightarrow{\Id_W\tens l_X} W\tens X.
\end{align*}
By the triangle axiom, naturality of the associativity isomorphisms, and the pentagon axiom, $F$ is the same composition as
\begin{align*}
 (W & \tens W')\tens(W\tens X)\xrightarrow{\cA_{W\tens W',W,X}} ((W\tens W')\tens W)\tens X\nonumber\\
 &\xrightarrow{\cA_{W,W',W}^{-1}\tens\Id_X} (W\tens(W'\tens W))\tens X\xrightarrow{(\Id_W\tens e_W)\tens\Id_X} (W\tens V)\tens X\xrightarrow{r_W\tens\Id_X} W\tens X.
\end{align*}
Thus $F\circ\cY_{W\tens W',W\tens X}$ is the intertwining operator $\cY_{W\tens W',W\tens X}^{W\tens X}$ in the statement of the proposition, so that
\begin{align*}
 \tr_{W\tens X} \cY_{W,X} &\, (\cU(q_{z})w,q_{z})\cY_{W',W\tens X}^X(\cU(1)w',1)q_\tau^{L(0)-c/24}\nonumber\\
 &=\tr_{W\tens X} \cY_{W\tens W',W\tens X}^{W\tens X}\left(\cU(1)\cY_{W,W'}(w,e^{\log z})w',1\right)q_\tau^{L(0)-c/24}
\end{align*}
for $(z,\tau)\in U_1^{prod}\cap U_1^{it}$.
 \end{proof}

 \begin{rem}\label{rem:pt_obstruction}
  We conjecture that a result similar to Proposition \ref{prop:genus_1_beta} also holds when $Z_{W,X}$ is replaced by an analogous pseudo-trace on $X$ associated to a symmetric linear function $\phi$ on a finite-dimensional algebra $P$ acting on $X$ by $V$-module endomorphisms, with $X$ projective as a $P$-module. However, the above proof does not immediately generalize, because the step in the calculation \eqref{eqn:trace_prop_app} that uses the cyclic symmetry of traces is only valid for pseudo-traces if $W\tens X$ is a projective $P$-module. While $P$ does act naturally on $W\tens X$ by $P$-module endomorphisms, it does not seem obvious that $W\tens X$ is a projective $P$-module. This is in fact the only obstruction at the moment to using the methods of this paper to prove that the module category of any strongly finite vertex operator algebra is rigid.
  
From Miyamoto's original construction of pseudo-trace functions in \cite[Section 3]{Mi2}, or alternatively from the results in Proposition 5.2 and Remark 5.6 of \cite{GR}, it appears that the necessary pseudo-traces $\mathrm{Tr}^\phi_{W\tens X}$ would indeed be well defined if the functor $W\tens\bullet$ were exact (and not just right exact). Exactness of $W\tens\bullet$ would hold if $W$ were a rigid $V$-module. But unfortunately, we will need Proposition \ref{prop:genus_1_beta} precisely to prove rigidity of $W$ in Section \ref{sec:main_thms}, so at the moment we can only prove rigidity in situations where no pseudo-traces appear.
 \end{rem}

To conclude this section, we discuss the $S$-transformation of one- and two-point genus-one correlation functions. For a $V$-module $W$ and a function $F: W\times\mathbb{H}\rightarrow\CC$
that is linear in $w\in W$ and holomorphic in $\tau\in\mathbb{H}$, we define
\begin{equation*}
 S(F)(w;\tau)=F\bigg(\left(-\frac{1}{\tau}\right)^{L(0)} w;-\frac{1}{\tau}\bigg),
\end{equation*}
where $\left(-\frac{1}{\tau}\right)^{L(0)}=e^{\log(-\frac{1}{\tau}) L(0)}$ is defined using the principal branch of logarithm. Note that $S(F)$ is also linear in $w$ and holomorphic in $\tau$. It is easy to see that the $S$-transformation $F\mapsto S(F)$ is invertible with inverse defined by
\begin{equation*}
 S^{-1}(F)(w;\tau) =F\bigg(\tau^{-L(0)}w;-\frac{1}{\tau}\bigg),
\end{equation*}
where $\tau^{-L(0)}=e^{-(\log\tau)L(0)}$ is defined using the principal branch of logarithm. We will in particular apply the $S$-transformation to characters of $V$-modules: for any $V$-module $W$, \cite[Theorem 5.5]{Mi2} shows there are pseudo-trace functions $\tr_X^\phi$ such that
\begin{equation}\label{eqn:S_ch_V}
 \mathrm{ch}_W\bigg(\left(-\frac{1}{\tau}\right)^{L(0)}v;-\frac{1}{\tau}\bigg) =\sum_{X,\phi}  \tr_X^\phi Y_X(\cU(1)v,1)q_\tau^{L(0)-c/24}
\end{equation}
for all $v\in V$ and $\tau\in\mathbb{H}$. Moreover, if the Zhu algebra $A(V)$ is semisimple, the results in \cite{Zh} show that $S(\mathrm{ch}_W)$ is in fact a linear combination of characters $\mathrm{ch}_X$, that is, the sum in \eqref{eqn:S_ch_V} runs over the distinct irreducible $V$-modules and $\tr_X^\phi=s_{W,X}\tr_X$ for certain $s_{W,X}\in\CC$.

\begin{rem}
 As mentioned in \cite[Remark 3.3.5]{AN} and discussed in the proof of \cite[Theorem 5.1]{CM}, the proof of \cite[Theorem 5.5]{Mi2} implicitly assumes that $V$ has no finite-dimensional irreducible modules. This assumption always holds if $V$ is simple and infinite dimensional since a finite-dimensional irreducible module is necessarily a simple quotient of $V$ itself. If $V$ is simple and finite dimensional, then $V=\mathbb{C}\mathbf{1}$ and \eqref{eqn:S_ch_V} is obvious.
\end{rem}

For a two-point function associated to intertwining operators $\cY_1$ and $\cY_2$ of types $\binom{X}{W_1\,M}$ and $\binom{M}{W_2\,X}$, we define $S(\overline{F}_{\cY_1,\cY_2}^\phi)$ to be the multivalued analytic function in $z$ and $\tau$ with principal branch
\begin{align*}
 S(F_{\cY_1,\cY_2}^\phi)&(w_1,w_2;z,\tau)\nonumber\\
 &=\tr_X^\phi\cY_1\bigg(\cU(q_{-z/\tau})\left(-\frac{1}{\tau}\right)^{L(0)}w_1,q_{-z/\tau}\bigg)\cY_2\bigg(\cU(1)\left(-\frac{1}{\tau}\right)^{L(0)}w_2,1\bigg)q_{-1/\tau}^{L(0)-c/24}
\end{align*}
defined on the region  
\begin{equation*}
 0<\im\frac{z}{\tau}<\im\left(-\frac{1}{\tau}\right).
\end{equation*}
Note that $(z,\tau)$ is in the domain of the principal branch for both $\overline{F}_{\cY_1,\cY_2}^\phi$ and $S(\overline{F}_{\cY_1,\cY_2}^\phi)$ if and only if $-z$ is in the open parallelogram $P_\tau$ determined by $1$ and $\tau$, that is,
\begin{equation*}
 -z=a+b\tau
\end{equation*}
for $a,b\in\RR$ with $0<a,b<1$.

\section{The main theorems}\label{sec:main_thms}

In this section, we assume that $V$ is an $\NN$-graded simple $C_2$-cofinite vertex operator algebra with grading-restricted generalized module category $\cC$.

\subsection{Factorizability from rigidity}\label{subsec:factorizability}

In this subsection, we assume $V$ is self-contragredient; we prove that if $\cC$ is rigid, then it is factorizable, that is, the braiding on $\cC$ is nondegenerate. The idea is to show that for a rigid $V$-module $W$, the $S$-transformation of $\mathrm{ch}_W$ contains information about the monodromy of $W$ with other objects of $\cC$. To show this, we will use the method developed by Moore-Seiberg \cite{MS} and Huang \cite{Hu-Verlinde} for deriving identities of series expansions of two-point genus-one correlation functions. 

For $V$-modules $W$ and $U$ (not necessarily simple or rigid), we use the contragredient $W'$ and evaluation $e_W:W'\tens W\rightarrow V$ of Section \ref{sec:contragredient} to define the function $Z_{W,U}$ of \eqref{eqn:ZWX}. We start by relating the $S$-transformation of $Z_{W,U}$ to $S(\mathrm{ch}_U)$:
\begin{lem}\label{lem:S-trans}
Let $W$ and $U$ be $V$-modules, and let
\begin{equation*}
 S(\mathrm{ch}_U)(v;\tau)=\sum_{X,\phi} \tr_{X}^\phi Y_{X}(\cU(1)v,1)q_\tau^{L(0)-c/24}
\end{equation*}
for all $v\in V$ and $\tau\in\mathbb{H}$. Then
 \begin{align}\label{eqn:2pt_S-trans_id}
  Z_{W,U} & \bigg(\left(-\frac{1}{\tau}\right)^{L(0)}w',\left(-\frac{1}{\tau}\right)^{L(0)}w;-\frac{z}{\tau},-\frac{1}{\tau}\bigg)\nonumber\\
  &\qquad\qquad=\sum_{X, \phi} \tr_X^\phi \cY^X_{W',W\tens X}(\cU(q_z)w',q_z)\cY_{W,X}(\cU(1)w,1)q_\tau^{L(0)-c/24}
 \end{align}
 for all $w\in W, w'\in W',\tau\in\mathbb{H}$, and $-z\in P_\tau$.
\end{lem}
\begin{proof}
The pseudo-trace functions on the right side of \eqref{eqn:2pt_S-trans_id} are well defined by \cite[Proposition 1.31]{Fi}, that is, the intertwining operators $\cY_{W,X}$ and $\cY_{W',W\tens X}^X$ commute with the actions of the associative algebra $P$ on $X$ and $W\tens X$. Now fix $\tau\in\mathbb{H}$ and consider the non-empty open set of $z\in\CC$ satisfying the conditions 
 \begin{equation}\label{eqn:S-trans_cond}
  -z\in P_\tau,\quad 0<\vert z\vert<\min(R_\tau, \vert\tau\vert R_{-1/\tau}),\quad \arg z,\,\arg\left(-\frac{z}{\tau}\right)\neq\pi.
 \end{equation}
 These conditions imply
 \begin{equation*}
  (z,\tau), \left(-\frac{z}{\tau},-\frac{1}{\tau}\right)\in U_1^{prod}\cap U_1^{it},
 \end{equation*}
so we get the following for such $z$ from Proposition \ref{prop:Z_Z_til} and the $L(0)$-conjugation formula:
\begin{align}\label{eqn:S-trans_calc}
 Z_{W,U} & \bigg(\left(-\frac{1}{\tau}\right)^{L(0)}w',\left(-\frac{1}{\tau}\right)^{L(0)}w;-\frac{z}{\tau},-\frac{1}{\tau}\bigg)\nonumber\\
 &= \tr_U Y_U\bigg(\cU(1)\cE_W\bigg(\bigg(-\frac{1}{\tau}\bigg)^{L(0)}w', e^{\log\left(-\frac{z}{\tau}\right)}\bigg)\bigg(-\frac{1}{\tau}\bigg)^{L(0)}w,1\bigg)q_{-1/\tau}^{L(0)-c/24}\nonumber\\
&= \tr_U Y_U\bigg(\cU(1)\bigg(-\frac{1}{\tau}\bigg)^{L(0)}\cE_W\bigg(w',e^{\log\left(-\frac{z}{\tau}\right)-\log\left(-\frac{1}{\tau}\right)}\bigg)w,1\bigg)q_{-1/\tau}^{L(0)-c/24}.
\end{align}
Now since $-z\in P_\tau$, so that in particular $-z\in\mathbb{H}$, we have $-\pi<\arg z<0$. Thus because $0<\arg(-\frac{1}{\tau})<\pi$,
\begin{align*}
 \log\left(-\frac{z}{\tau}\right) & =\ln\left\vert-\frac{z}{\tau}\right\vert+i\left[\arg z+\arg\left(-\frac{1}{\tau}\right)\right]=\log z+\log\left(-\frac{1}{\tau}\right).
\end{align*}
Consequently, the right side of \eqref{eqn:S-trans_calc} is
\begin{align*}
 \tr_U &\, Y_U\bigg(\cU(1)\bigg(-\frac{1}{\tau}\bigg)^{L(0)}\cE_W\left(w',e^{\log z}\right)w,1\bigg)q_{-1/\tau}^{L(0)-c/24}\nonumber\\
 & =\sum_{h\in\QQ}\sum_{k=0}^K\left. \mathrm{ch}_U\bigg(\left(-\frac{1}{\tau}\right)^{L(0)} e_W(w'\tens_{h,k} w);-\frac{1}{\tau}\bigg)x^{-h-1}(\log x)^k\right\vert_{\log x=\log z}\nonumber\\
 & =\sum_{h\in\QQ}\sum_{k=0}^K \sum_{X, \phi} \left. \left(\tr_X^\phi Y_X\left(\cU(1)e_W(w'\tens_{h,k} w),1\right)q_\tau^{L(0)-c/24}\right) x^{-h-1}(\log x)^k\right\vert_{\log x=\log z}\nonumber\\
 & =\sum_{X, \phi} \tr_X^\phi Y_X\left(\cU(1)\cE_W(w',e^{\log z})w,1\right) q_\tau^{L(0)-c/24}\nonumber\\
 & =\sum_{X, \phi} \tr_X^\phi \cY^X_{W',W\tens X}(\cU(q_z)w',q_z)\cY_{W,X}(\cU(1)w,1)q_\tau^{L(0)-c/24}
\end{align*}
for $z$ as in \eqref{eqn:S-trans_cond}, where the last step comes from Proposition \ref{prop:genus_1_assoc}. Thus for any $\tau\in\mathbb{H}$, the two sides of \eqref{eqn:2pt_S-trans_id} are equal for $z$ in a non-empty open set. As both sides are analytic in $z$, they are equal for all $(z,\tau)$ for which both are defined, that is, for $(z,\tau)$ such that $-z\in P_\tau$.
\end{proof}

Now since $V$ is self-contragredient, the dual of any rigid $V$-module $W$ is necessarily $W'$ with evaluation $e_W$ and some coevaluation $i_W: V\rightarrow W\tens W'$. Then for any $V$-module $X$, the open Hopf link endomorphism $h_{W,X}\in\Endo_V(X)$ associated to $W$ is the composition
\begin{align*}
 X\xrightarrow{l_X^{-1}} & V\tens X\xrightarrow{i_W\tens\Id_X} (W\tens W')\tens X\xrightarrow{(\theta_W\tens\Id_{W'})\tens\Id_X} (W\tens W')\tens X\nonumber\\
 &\xrightarrow{\cR_{W,W'}\tens\Id_X} (W'\tens W)\tens X\xrightarrow{\cA_{W',W,X}^{-1}} W'\tens(W\tens X)\xrightarrow{\Id_W\tens\cR^2_{W,X}} W'\tens(W\tens X)\nonumber\\
 &\xrightarrow{\cA_{W',W,X}} (W'\tens W)\tens X \xrightarrow{e_W\tens\Id_X} V\tens X\xrightarrow{l_X} X,
\end{align*}
where $\theta_W=e^{2\pi i L(0)}$. Graphically, 
\begin{equation}\label{eqn:open_Hopf_link}
 h_{W,X}=\begin{matrix}
    \begin{tikzpicture}[scale = 1, baseline = {(current bounding box.center)}, line width=0.75pt]
     \draw[white, double=black, line width = 3pt ] (3,0) -- (3,3) .. controls (3,3.5) and (2,3.25) .. (2,3.75);
     \draw (2,1.5) -- (2,2.25) .. controls (2,2.75) and (1,2.5) .. (1,3);
     \draw[white, double=black, line width = 3pt ]  (1,3) -- (1,4.5) ..controls (1,5.2) and (2,5.2) .. (2,4.5) ..controls (2,4) and (3,4.25) .. (3,3.75) .. controls (3,3.25) and (2,3.5) .. (2,3) .. controls (2,2.5) and (1,2.75) .. (1,2.25) -- (1,1.5) .. controls (1,.8) and (2,.8) .. (2,1.5);
     \draw[white, double=black, line width = 3pt ] (2,3.75) .. controls (2,4.25) and (3,4) .. (3,4.5) --(3,6);
     \draw[dashed] (3,.4) .. controls (1.5,.5) .. (1.5,1);
     \draw[dashed] (1.5, 5) .. controls (1.5,5.5) .. (3,5.6);
     \node at (1,1.75) [draw,minimum width=10pt,minimum height=10pt,thick, fill=white] {$\theta_W$};
     \node at (3,-.25) {$X$};
     \node at (.9,.9) {$W$};
     \node at (2.1,.9) {$W'$};
     \node at (3,6.25) {$X$};
    \end{tikzpicture}
   \end{matrix}
\end{equation}
Note that $h_{W,X}$ is central in $\Endo_V(X)$ by naturality of the unit, associativity, and braiding isomorphisms in $\cC$; that is, $h_{W,\bullet}$ is a natural automorphism of the identity functor on $\cC$.

We now apply the Moore-Seiberg-Huang method to relate the $S$-transformation of $\mathrm{ch}_W$ to $S(\mathrm{ch}_V)$ when $W$ is rigid:
\begin{thm}\label{thm:S_ch_W}
 Suppose $W$ is a rigid $V$-module and
 \begin{equation*}
  S(\mathrm{ch}_V)(v;\tau)=\sum_{X,\phi} \tr_X^\phi Y_X(\cU(1)v,1)q_\tau^{L(0)-c/24}
 \end{equation*}
for $v\in V$ and $\tau\in\mathbb{H}$. Then
 \begin{equation}\label{eqn:S_ch_W}
  S(\mathrm{ch}_W)(v;\tau)=\sum_{X,\phi} \tr_X^\phi \,(h_{W,X}\circ Y_X)(\cU(1)v,1)q_\tau^{L(0)-c/24}
 \end{equation}
for $v\in V$ and $\tau\in\mathbb{H}$.
\end{thm}
\begin{proof}
The pseudo-trace functions in the statement of the theorem are well defined because $h_{W,X}$ being central in $\Endo_V(X)$ implies that $h_{W,X}\circ Y_X$ commutes with any endomorphism algebra of $X$. Now fix $\tau\in\mathbb{H}$. The idea is to analyze
 \begin{equation}\label{eqn:2nd_MSH_reln_1}
  Z_{W,V}\bigg(\left(-\frac{1}{\tau}\right)^{L(0)}w',\left(-\frac{1}{\tau}\right)^{L(0)}w;\frac{z}{\tau}+\frac{1}{\tau},-\frac{1}{\tau}\bigg)
 \end{equation}
in two different ways. On the one hand, Proposition \ref{prop:genus_1_beta} shows how to write \eqref{eqn:2nd_MSH_reln_1} as a series in powers of $-\frac{z}{\tau}$ and $q_{-1/\tau}$ which, with $\tau$ fixed, is effectively a series in powers of $z$. On the other hand, Lemma \ref{lem:S-trans} shows how to use the $S$-transformation of $\mathrm{ch}_V$ to write \eqref{eqn:2nd_MSH_reln_1} as a series in powers of $-z-1$ and $q_\tau$. Then we can apply Proposition \ref{prop:genus_1_alpha} and intertwining operator skew-symmetry to get first a series in powers of $-z$ and $q_\tau$ and then a series in powers of $z$ and $q_{\tau}$. With $\tau$ fixed, these two different calculations yield two seemingly different series in powers of $z$ that nevertheless must coincide because they will both equal \eqref{eqn:2nd_MSH_reln_1} for a non-empty open set of $z$. Taking coefficients of powers of $z$ in these two equal series for any $\tau$ then yields non-trivial identities which, when combined with the assumption that $W$ is rigid, will yield \eqref{eqn:S_ch_W}. However, the detailed argument is subtle because we need to choose $z$ carefully so that both rewritings of \eqref{eqn:2nd_MSH_reln_1} are valid, and because we need to pay attention to branches of logarithm throughout.

To begin the first calculation, Proposition \ref{prop:genus_1_beta} says that if $(-\frac{z}{\tau},-\frac{1}{\tau})\in U_1^{prod}\cap U_1^{it}$, then \eqref{eqn:2nd_MSH_reln_1} is equal to
\begin{equation*}
\tr_{W\tens V}  \cY_{W\tens W', W\tens V}^{W\tens V}\bigg(\cU(1)\cY_{W,W'}\bigg(\left(-\frac{1}{\tau}\right)^{L(0)}w,e^{\log\left(-\frac{z}{\tau}\right)}\bigg)\left(-\frac{1}{\tau}\right)^{L(0)}w', 1\bigg)q_{-1/\tau}^{L(0)-c/24}.
\end{equation*}
We can use the $L(0)$-conjugation property to move the $\left(-\frac{1}{\tau}\right)^{L(0)}$ operators on $w$ and $w'$ outside the intertwining operator $\cY_{W,W'}$. We also want to write the trace on $W\boxtimes V$ as a trace on the isomorphic module $W$; to do so, we insert $r_W^{-1}\circ r_W=\Id_{W\tens V}$ in front of $\cY_{W\tens W',W\tens V}^{W\tens V}$ and apply the cyclic symmetry of traces to see that \eqref{eqn:2nd_MSH_reln_1} equals
\begin{align*}
\tr_W(r_W\circ\cY_{W\tens W',W\tens V}^{W\tens V})\bigg(\cU(1)\left(-\frac{1}{\tau}\right)^{L(0)}\cY_{W,W'}\left(w,e^{\log\left(-\frac{z}{\tau}\right)-\log\left(-\frac{1}{\tau}\right)}\right)w',1\bigg)r_W^{-1}q_{-1/\tau}^{L(0)-c/24}
\end{align*}
for $(-\frac{z}{\tau},-\frac{1}{\tau})\in U_1^{prod}\cap U_1^{it}$. This is now the trace on $W$ of the intertwining operator
 \begin{equation*}
 \cY_{W\tens W',W}^W :=r_W\circ\cY_{W\tens W',W\tens V}^{W\tens V}\circ(\Id_{W\tens W'}\otimes r_W^{-1}).
 \end{equation*}
  By the definition of $\cY_{W\tens W',W\tens V}^{W\tens V}$ in the statement of Proposition \ref{prop:genus_1_beta}, properties of right unit isomorphisms in a tensor category, and the definition of tensor products of $V$-module homomorphisms, this intertwining operator equals
\begin{align*}
& \cY_{W\tens W',W}^W  =r_W\circ(r_W\tens\Id_V)\circ((\Id_W\tens e_W)\tens\Id_V)\circ(\cA_{W,W',W}^{-1}\tens\Id_V) \circ \nonumber\\
 &\quad\hspace{9em}\circ\cA_{W\tens W',W,V}\circ\cY_{W\tens W',W\tens V}\circ(\Id_{W\tens W'}\otimes r_W^{-1})\nonumber\\
 &\quad =r_W\circ(\Id_W\tens e_W)\circ\cA_{W,W',W}^{-1}\circ r_{(W\tens W')\tens W}\circ\cA_{W\tens W',W,V}\circ(\Id_{W\tens W'}\tens r_W^{-1})\circ\cY_{W\tens W',W}\nonumber\\
 &\quad =r_W\circ(\Id_W\tens e_W)\circ\cA_{W,W',W}^{-1}\circ\cY_{W\tens W',W}.
\end{align*}
We simplify further by noting that in the trace on $W$ of $\cY_{W\boxtimes W',W}^W$ appearing above, the expression $\log\left(-\frac{z}{\tau}\right)-\log\left(-\frac{1}{\tau}\right)$ is actually some branch of logarithm of $z$. Specifically, since we will soon apply Lemma \ref{lem:S-trans} to \eqref{eqn:2nd_MSH_reln_1} in our second calculation, we now assume that $z+1\in P_\tau$, or equivalently, $z\in P_\tau-1$. This implies in particular that $\arg\tau<\arg z<\pi$, so that $\arg(-\frac{z}{\tau})=\arg z-\arg\tau-\pi$. Also since $\arg(-\frac{1}{\tau})=\pi-\arg(\tau)$, we get
\begin{equation*}
 \log\left(-\frac{z}{\tau}\right)-\log\left(-\frac{1}{\tau}\right) =\ln\left\vert\frac{z}{\tau}\right\vert+i(\arg z-\arg\tau-\pi)-\ln\left\vert\frac{1}{\tau}\right\vert -i(\pi-\arg\tau)    = \log z-2\pi i.
\end{equation*}
Combining everything so far, we have shown that \eqref{eqn:2nd_MSH_reln_1} is equal to
\begin{equation}\label{eqn:2nd_MSH_reln_2}
 \tr_W \cY_{W\tens W',W}^W\bigg(\cU(1)\left(-\frac{1}{\tau}\right)^{L(0)}\cY_{W,W'}(w,e^{\log z-2\pi i})w',1\bigg)q_{-1/\tau,}^{L(0)-c/24}
\end{equation}
assuming that $z\in P_\tau-1$ and $(-\frac{z}{\tau},-\frac{1}{\tau})\in U_1^{it}$ (the first of these conditions automatically implies $(-\frac{z}{\tau},-\frac{1}{\tau})\in U_1^{prod}$ as well). This completes the first calculation.

We now begin the second calculation. We retain the above assumptions on $z$ and assume further that $(-z,\tau)\in U_1^{it}$. The assumptions we have imposed on $z$ define a non-empty open subset of $\CC^\times$, namely the intersection of the open parallelogram $P_\tau-1$ (with vertices $0$, $\tau$, $-1$, and $\tau-1$) with the open disk of non-zero radius $\min(R_\tau, \vert\tau\vert R_{-1/\tau})$. For $z$ in this non-empty open set, since in particular $z+1\in P_\tau$, Lemma \ref{lem:S-trans} implies that 
\eqref{eqn:2nd_MSH_reln_1} equals
\begin{align*}
 \sum_{X,\phi}  & \tr_X^\phi \cY_{W',W\tens X}^X(\cU(q_{-z-1})w',q_{-z-1})\cY_{W,X}(\cU(1)w,1)q_\tau^{L(0)-c/24}.
 \end{align*}
Then because also $(-z,\tau)\in U_1^{it}$, Proposition \ref{prop:genus_1_alpha} implies that this equals 
\begin{align*}
\sum_{X,\phi} \tr_X^\phi
 \cY^1\left(\cU(1)\cY_{W',W}(w',e^{\log(-z)})\theta_W(w),1\right) q_{\tau}^{L(0)-c/24},
\end{align*}
where
\begin{align*}
 \cY^1 = l_X\circ(e_W\tens\Id_X)\circ\cA_{W',W,X}\circ(\Id_{W'}\tens\cR^2_{W,X})\circ\cA_{W',W,X}^{-1}\circ\cY_{W'\tens W,X}.
\end{align*}
Now we use the braiding to switch the order of $w'$ and $w$ in this sum of pseudo-traces. First, 
\begin{equation*}
 \log(-z)=\log z-\pi i
\end{equation*}
since $0<\arg z<\pi$, so by the definition of $\cR_{W,W'}$,
\begin{align*}
\cY_{W',W}(w',e^{\log(-z)})\theta_W(w) & = e^{-zL(-1)}e^{z L(-1)}\cY_{W',W}(w', e^{\pi i} e^{\log z-2\pi i})\theta_W(w)\nonumber\\
& = e^{-zL(-1)}(\cR_{W,W'}\circ\cY_{W,W'})(\theta_W(w), e^{\log z-2\pi i})w'\nonumber\\
& =e^{-zL(-1)}(\cR_{W,W'}\circ(\theta_W\tens\Id_{W'})\circ\cY_{W,W'})(w, e^{\log z-2\pi i})w'.
\end{align*}
Inserting this into the above sum of pseudo-traces and observing that $\cU(1)e^{-zL(-1)}$ commutes with $\cR_{W,W'}$ and $\theta_W\tens\Id_{W'}$, \eqref{eqn:2nd_MSH_reln_1} becomes
\begin{equation*}
\sum_{X,\phi}\tr_X^\phi \til{\cY}^1\left(\cU(1)e^{-zL(-1)}\cY_{W,W'}(w,e^{\log z-2\pi i})w',1\right)q_\tau^{L(0)-c/24}
\end{equation*}
where
\begin{align*}
 \til{\cY}^1 & = \cY_1\circ(\cR_{W,W'}\otimes\Id_{X})\circ((\theta_W\tens\Id_{W'})\otimes\Id_X)\nonumber\\ 
& =l_X\circ(e_W\tens\Id_X)\circ\cA_{W',W,X}  \circ(\Id_{W'}\tens\cR^2_{W,X})\circ\cA_{W',W,X}^{-1}\circ\nonumber\\
 &\qquad\qquad\circ(\cR_{W,W'}\tens\Id_X)\circ((\theta_W\tens\Id_{W'})\tens\Id_X)\circ\cY_{W\tens W',X}.
\end{align*}
Now, we can remove the factor of $e^{-zL(-1)}$ since by \cite[Equation 1.15]{Hu-mod_inv}, the $L(0)$-commutator formula, and the fact that the pseudo-trace of a commutator is $0$,
\begin{align*}
 \tr_X^\phi \til{\cY}^1(\cU(1)L(-1)\til{w},1)q_\tau^{L(0)-c/24} & = 2\pi i\tr_X^\phi \til{\cY}^1((L(-1)+L(0))\cU(1)\til{w},1)q_\tau^{L(0)-c/24} \nonumber\\
&=2\pi i\tr_X^\phi \,[L(0),\til{\cY}^1(\cU(1)\til{w},1)q_\tau^{L(0)-c/24}]= 0
\end{align*}
for any $\til{w}\in W\tens W'$. Thus in fact \eqref{eqn:2nd_MSH_reln_1} equals
\begin{align}\label{eqn:2nd_MSH_reln_3}
  \sum_{X,\phi} \tr_X^\phi \til{\cY}^1\left(\cU(1)\cY_{W,W'}(w,e^{\log z-2\pi i})w',1\right)q_\tau^{L(0)-c/24}.
\end{align}
This completes the second calculation.

We have now shown that \eqref{eqn:2nd_MSH_reln_2} equals \eqref{eqn:2nd_MSH_reln_3} for all $z$ in a non-empty open subset of $\CC^\times$, since they are both equal to \eqref{eqn:2nd_MSH_reln_1} for such $z$.  By analytically continuing the corresponding analytic function along an appropriate loop about the singularity $z=0$, we may also replace $\log z-2\pi i$ with $\log z$ in these (pseudo-)traces. Now \cite[Proposition 7.8]{HLZ5} shows that
\begin{align*}
 \tr_W \cY_{W\tens W',W}^W & \bigg(\cU(1)\left(-\frac{1}{\tau}\right)^{L(0)}\cY_{W,W'}(w,x)w',1\bigg)q_{-1/\tau}^{L(0)-c/24}\nonumber\\
 &=\sum_{X,\phi} \tr_X^\phi \til{\cY}^1\left(\cU(1)\cY_{W,W'}(w,x)w',1\right)q_\tau^{L(0)-c/24}
\end{align*}
as formal series in powers of $x$ and $\log x$ for any fixed $\tau\in\mathbb{H}$. That is,
\begin{align}\label{eqn:2nd_MSH_reln_4}
 \tr_W\cY_{W\tens W',W}^W & \bigg(\cU(1)\left(-\frac{1}{\tau}\right)^{L(0)}(w\tens_{h,k} w'),1\bigg)q_{-1/\tau}^{L(0)-c/24}\nonumber\\
 &=\sum_{X,\phi} \tr_X^\phi\til{\cY}^1\left(\cU(1)(w\tens_{h,k} w'),1\right)q_\tau^{L(0)-c/24}
\end{align}
for all $h\in\QQ$, $k\in\NN$, $w\in W$, $w'\in W'$.
 Since the vectors $w\tens_{h,k} w'$ span $W\tens W'$  (see \cite[Proposition 4.23]{HLZ3}), and since we have assumed there is a coevaluation $i_W: V\rightarrow W\tens W'$, we may replace $w\tens_{h,k} w'$ in \eqref{eqn:2nd_MSH_reln_4} with $i_W(v)$ for any $v\in V$. Then by rigidity,
\begin{align*}
 \cY_{W\tens W',W}^W & \circ(i_W\otimes\Id_W)\nonumber\\
 &= r_W\circ(\Id_W\tens e_W)\circ\cA_{W,W',W}^{-1}\circ (i_W\tens\Id_W)\circ\cY_{V,W}=l_W\circ\cY_{V,W}=Y_W,
\end{align*}
while the definitions of $\til{\cY}^1$ and the open Hopf link endomorphism $h_{W,X}$ imply
\begin{equation*}
 \til{\cY}^1\circ(i_W\otimes\Id_X) =h_{W,X}\circ l_X\circ \cY_{V,X}=h_{W,X}\circ Y_X.
\end{equation*} 
Thus \eqref{eqn:2nd_MSH_reln_4} specializes to \eqref{eqn:S_ch_W} for any $v\in V$, completing the proof of the theorem.
\end{proof}

Using \eqref{eqn:S_ch_W}, the invertibility of the $S$-transformation, and the linear independence of characters, we now show that nondegeneracy of the braiding on $\cC$ follows from rigidity:
\begin{thm}\label{thm:factorizability}
Let $V$ be a strongly finite vertex operator algebra and $\cC$ the braided tensor category of grading-restricted generalized $V$-modules. If $\cC$ is rigid, then $\cC$ is a factorizable finite ribbon category.
\end{thm}
\begin{proof}
By results from \cite{Hu-C2} and the rigidity assumption, $\cC$ is a finite braided ribbon category with duals given by contragredient modules and twist $e^{2\pi i L(0)}$. By \cite[Theorem 1.1]{Sh}, $\cC$ is factorizable if it has trivial M\"{u}ger center. That is, we need to show that if $W$ is a $V$-module such that $\cR^2_{W,X}=\Id_{W\tens X}$ for all $V$-modules $X$, then $W\cong V^{\oplus n}$ for some $n\in\NN$.

 If indeed $\cR^2_{W,X}=\Id_{W\tens X}$ for all $V$-modules $X$, then the open Hopf link endomorphism $h_{W,X}$ is simply $(\dim_\cC W)\Id_X$ for all $X$, where $\dim_\cC W$ is the categorical dimension of $W$; since $V$ is simple, this categorical dimension is just a scalar. Then \eqref{eqn:S_ch_W} simplifies to
 \begin{equation*}
  S(\mathrm{ch}_W) =(\dim_\cC W)\,S(\mathrm{ch}_V),
 \end{equation*}
 or equivalently $ \mathrm{ch}_W=(\dim_\cC W)\,\mathrm{ch}_V$ since the $S$-transformation is invertible. Because traces depend only on the semisimplification of a module, $\mathrm{ch}_W$ is a linear combination of the characters of the composition factors of $W$, and since the characters of distinct irreducible $V$-modules are linearly independent, it follows that $V$ is the only composition factor of $W$ (with multiplicity $\dim_\cC W$).

Now because $\cC$ is a finite tensor category over $\CC$ with simple unit object $V$, \cite[Theorem 4.4.1]{EGNO} shows that all self-extensions of $V$ split. Thus if $W$ is a module in the M\"{u}ger center of $\cC$ with length $\ell(W)\leq 2$, we may conclude that $W\cong V^{\oplus\ell(W)}$. The same relation then holds for $\ell(W)>2$ by a straightforward induction on the length, proving the theorem.
\end{proof}

\begin{exam}
 Theorem \ref{thm:factorizability} applies to the non-rational triplet vertex operator algebras $\cW(p)$ for $p>1$ an integer. See \cite{AM} for the definition of $\cW(p)$ and the proof that it is a strongly finite vertex operator algebra. The tensor category of grading-restricted generalized $\cW(p)$-modules is rigid \cite{TW} (see also \cite[Theorem 7.6]{MY}), so it is factorizable by Theorem \ref{thm:factorizability}. However, as shown in \cite[Theorem 4.7]{GN}, it is not too difficult to prove directly that the braiding on the category of $\cW(p)$-modules is nondegenerate. For any $p>1$, the automorphism group of $\cW(p)$ contains $\ZZ/m\ZZ$ for $m\in\ZZ_+$, and the fixed-point subalgebra $\cW(p)^{\ZZ/m\ZZ}$ is also strongly finite \cite{ALM}. In \cite[Theorem 6.7]{CMY-typical-singlet}, it is shown that the tensor category of grading-restricted generalized $\cW(p)^{\ZZ/m\ZZ}$-modules is rigid and therefore factorizable by Theorem \ref{thm:factorizability}.

  The even subalgebra $SF^+(d)$ of the vertex operator superalgebra of $d$ pairs of symplectic fermions for $d\geq 1$ is another example of a strongly finite (but non-rational) vertex operator algebra \cite{Ab}. For $d=1$, $SF^+(1)$ is isomorphic to the triplet algebra $\cW(2)$, and more generally, $SF^+(d)$ contains $\cW(2)^{\otimes d}$ as a vertex operator subalgebra. In \cite[Theorem 5.1]{McR-Deligne}, it is shown that the category of $\cW(2)^{\otimes d}$-modules is tensor equivalent to the Deligne tensor product of $d$ copies of the category of $\cW(2)$-modules, and thus is rigid. Then the vertex operator algebra extension theory of \cite{KO, CKM1} implies that the category of $SF^+(d)$-modules is rigid, and therefore Theorem \ref{thm:factorizability} shows that the category of $SF^+(d)$-modules is factorizable (see \cite[Corollary 5.3]{McR-Deligne}). 
This factorizable symplectic fermion category is conjecturally equivalent to an explicit braided tensor category constructed by Runkel in \cite{Ru}. 
\end{exam}

\subsection{Rigidity}\label{subsec:rig}

In this subsection, we first prove that the tensor category $\cC$ of grading-restricted generalized $V$-modules is rigid, under the assumption that
the $S$-transformation $S(\mathrm{ch}_V)$ is a linear combination of characters. Actually, we need a somewhat more general result for the next section, where we will consider the $\ZZ$-graded subalgebras of $C_2$-cofinite but $\frac{1}{2}\NN$-graded affine $W$-algebras. In such a case, the $\ZZ$-graded subalgebra is the fixed-point subalgebra of an order-$2$ automorphism of the $\frac{1}{2}\NN$-graded vertex operator algebra. Thus we now suppose that $V$ is the fixed-point subalgebra under a finite-order automorphism of a larger vertex operator algebra $A$:
\begin{assum}\label{assum:rigidity}
 Assume the following setting:
 \begin{itemize}
  \item $A$ is an $\NN$-graded simple $C_2$-cofinite vertex operator algebra with conformal vector $\omega$.
  
  \item The $S$-transformation of $\mathrm{ch}_A$ is a linear combination of characters: for certain $A$-modules $X$ and $s_{A,X}\in\CC$,
  \begin{equation*}
   \mathrm{ch}_A\bigg(\left(-\frac{1}{\tau}\right)^{L(0)} a;-\frac{1}{\tau}\bigg) =\sum_X s_{A,X}\,\mathrm{ch}_X(a;\tau)
  \end{equation*}
for all $a\in A$, $\tau\in\mathbb{H}$.

\item $V$ is the fixed-point vertex operator subalgebra of a finite-order vertex operator algebra automorphism $g$ of $A$.

\item $V$ is $\NN$-graded and self-contragredient with respect to a possibly different conformal vector $\til{\omega}=\omega+v_{-2}\vac$, where $v\in V$ satisfies $v_0\omega=0$. For an object $W$ in the category $\cC$ of grading-restricted generalized $V$-modules, $\theta_W$ and $W'$ will denote the twist automorphism and contragredient module, respectively, with respect to $\til{\omega}$.
 \end{itemize}
\end{assum}

Given the setting of Assumption \ref{assum:rigidity}, we make the following observations:
\begin{itemize}
\item By the main theorem of \cite{DLM_cpt_aut_gps}, $V$ is simple, and by the main theorem of \cite{Mi-cyc-orb} (see also the discussion in \cite[Section 3.1]{CM}), $V$ is $C_2$-cofinite (and $\NN$-graded with respect to either conformal vector $\omega$ or $\til{\omega}$).

\item By Proposition \ref{prop:ten_cat_ind_of_omega}, the braided tensor category $\cC$ is independent of whether we use $\omega$ or $\til{\omega}$ as the conformal vector of $V$. In particular, the $A$-modules $X$ appearing in the $S$-transformation of $\mathrm{ch}_A$ are objects of $\cC$ when considered as $V$-modules. 

\item We are using the conformal vector $\omega$ to define characters of $A$- and $V$-modules, since $A$ might not be $\NN$-graded as a vertex operator algebra with respect to $\til{\omega}$. 

\item We must distinguish between the twist isomorphisms $\theta_W$ defined using $\til{\omega}$ and the maps $e^{2\pi i L(0)}$ defined using $\omega$. Both, however, define natural isomorphisms in $\cC$ since $V$ (although not necessarily $A$) is $\NN$-graded with respect to both $\omega$ and $\til{\omega}$.

\item Since $V$ is self-contragredient and $\NN$-graded with respect to $\til{\omega}$, all the constructions and results of Section \ref{sec:contragredient} apply in this setting, provided we use contragredient modules $W'$ and twist isomorphisms $\theta_W$ constructed using $\til{\omega}$.
\end{itemize}
In the setting of Assumption \ref{assum:rigidity}, we also have the following lemma:
\begin{lem}\label{lem:A_simple_monodromy}
 If $W$ is a simple grading-restricted generalized $V$-module, then $(\cR_{W,A}^2)^{-1}=\Id_W\tens g^m$ for some $0\leq m\leq\vert g\vert-1$.
\end{lem}
\begin{proof}
 By the main theorem of \cite{DLM_cpt_aut_gps}, $A$ has the following decomposition as a $V$-module:
\begin{equation}\label{eqn:A_decomp}
 A =\bigoplus_{n\in\ZZ/\vert g\vert\ZZ} A_n,\qquad A_n =\lbrace a\in A\,\,\vert\,\,g\cdot a=e^{2\pi i n/\vert g\vert}a\rbrace,
\end{equation}
and the $A_n$ are distinct simple $V$-modules (with $A_0=V$). Moreover, the $A_n$ are simple current $V$-modules by \cite[Theorem 4.2]{CM}, or more specifically, 
\begin{equation}\label{eqn:simple_current}
A_m\tens A_n\cong A_{m+n} 
\end{equation}
for $m,n\in\ZZ/\vert g\vert\ZZ$. Then if $W$ is a simple $V$-module, so is each $W\tens A_n$, so that $(\cR_{W,A_n}^2)^{-1}=c_n\,\Id_{W\tens A_n}$ for some $c_n\in\CC$ (with $c_0=1$). From \eqref{eqn:simple_current} and the hexagon axiom, $c_{m+n}=c_m c_n$ for $m,n\in\ZZ/\vert g\vert\ZZ$, which means that $c_1 =e^{2\pi i m/\vert g\vert}$ for some $0\leq m\leq \vert g\vert-1$. Thus 
\begin{equation*}
 (\cR_{W,A_n}^2)^{-1}=e^{2\pi i mn/\vert g\vert}\Id_{W\tens A_n}
\end{equation*}
for each $n\in\ZZ/\vert g\vert\ZZ$, that is, $(\cR_{W,A}^2)^{-1} =\Id_W\tens g^m$.
\end{proof}

We now prove the following fundamental relation using the Moore-Seiberg-Huang method again. The result and its proof are similar to much of the work in \cite[Sections 5.3 and 5.4]{CM} (especially Lemmas 5.15 and 5.16, and Propositions 5.17 and 5.23 of \cite{CM}). The statement involves the $V$-module homomorphisms $\Phi_W: W\tens W'\rightarrow(W\tens W')'$ and $\til{\Psi}_W: (W\tens W')'\tens W\rightarrow W$ defined in Section \ref{sec:contragredient}:
\begin{thm}\label{thm:MSH_reln_1}
Under Assumption \ref{assum:rigidity}, let $W$ be a grading-restricted generalized $V$-module such that $(\cR_{W,A}^2)^{-1}=\Id_W\tens g^m$ for $0\leq m\leq\vert g\vert-1$. Then for any $\til{w}\in W\tens W'$ and $\tau\in\mathbb{H}$,
 \begin{align}\label{eqn:MSH_reln_1}
 \tr_A\,(g^m\circ Y_A)&\bigg(\cU(1)\left(-\frac{1}{\tau}\right)^{L(0)}\til{e}_W(\til{w}),1\bigg)q_{-1/\tau}^{L(0)-c/24}\nonumber\\
 &=\sum_X s_{A,X}\tr_{W\tens X} \cY_{\til{\Psi}_W,X}\left(\cU(1)[\Phi_W\circ(f\tens\Id_{W'})](\til{w}),1   \right)q_\tau^{L(0)-c/24},
 \end{align}
where $f= e^{2\pi i L(0)}\theta_W^{-1}\in\Endo_V(W)$ and  $\cY_{\til{\Psi}_W,X}=(\til{\Psi}_W\tens\Id_X)\circ\cA_{(W\tens W')',W,X}\circ\cY_{(W\tens W')', W\tens X}$ is of type $\binom{W\tens X}{(W\tens W')'\,W\tens X}$. 
\end{thm}

\begin{proof}
 Fix $w\in W$, $w'\in W'$, and $\tau\in\mathbb{H}$. For this proof, the idea is to analyze
 \begin{equation}\label{eqn:rigidity_starting_point}
\sum_X s_{A,X} Z_{W,X}(w',f(w); -z-\tau,\tau)  
 \end{equation}
in two different ways, similar to the proof of Theorem \ref{thm:S_ch_W}.

On the one hand, Proposition \ref{prop:genus_1_beta} says that for each $X$,
\begin{align*}
 Z_{W,X}(w',f(w);-z-\tau,\tau) = \tr_{W\tens X} \cY_{W\tens W',W\tens X}^{W\tens X}\left(\cU(1)\cY_{W,W'}(f(w),e^{\log z})w',1   \right)q_\tau^{L(0)-c/24}
\end{align*}
when $(z,\tau)\in U_1^{prod}\cap U_1^{it}$, and then using Lemma \ref{lem:PsiPhi}(2) we get
\begin{align*}
&\cY_{W\tens W',W\tens X}^{W\tens X}\nonumber\\
&\qquad= (r_W\tens\Id_X)\circ((\Id\tens e_W)\tens\Id_X)\circ(\cA_{W,W',W}^{-1}\tens\Id_X)\circ\cA_{W\tens W',W,X}\circ\cY_{W\tens W',W\tens X}\nonumber\\
 & \qquad=(\til{\Psi}_W\tens\Id_X)\circ((\Phi_W\tens\Id_W)\tens\Id_X)\circ\cA_{W\tens W',W,X}\circ\cY_{W\tens W',W\tens X}\nonumber\\
 &\qquad =(\til{\Psi}_W\tens\Id_X)\circ\cA_{(W\tens W')',W,X}\circ\cY_{(W\tens W')',W\tens X}\circ(\Phi_W\otimes\Id_{W\tens X})\nonumber\\
 &\qquad =\cY_{\til{\Psi}_W,X}\circ(\Phi_W\otimes\Id_{W\tens X}).
\end{align*}
Since $\Phi_W$ commutes with $\cU(1)$, this shows that \eqref{eqn:rigidity_starting_point} is equal to 
\begin{equation}\label{eqn:rig_proof_right}
 \sum_X s_{A,X}\tr_{W\tens X} \cY_{\til{\Psi}_W,X}\left(\cU(1)(\Phi_W\circ\cY_{W,W'})(f(w), e^{\log z})w',1   \right)q_\tau^{L(0)-c/24}
\end{equation}
when $(z,\tau)\in U_1^{prod}\cap U_1^{it}$.
We then obtain the right side of \eqref{eqn:MSH_reln_1} by taking coefficients of powers of $z$ and $\log z$ in \eqref{eqn:rig_proof_right} and allowing $w$ and $w'$ to vary.

On the other hand, assume in addition that $z+\tau\in P_\tau$, which is to say that $-z\in P_\tau-1$, and that $\left(\frac{z}{\tau},-\frac{1}{\tau}\right)\in U_1^{it}$. That is,
\begin{equation*}
 0<\vert z\vert <\vert\tau\vert R_{-1/\tau},\quad \arg\frac{z}{\tau}\neq \pi.
\end{equation*}
Note that the $z\in\CC$ satisfying all assumptions so far form a non-empty open set. Then Lemma \ref{lem:S-trans} and Proposition \ref{prop:genus_1_alpha} show that \eqref{eqn:rigidity_starting_point} equals
\begin{align}\label{eqn:MSH_reln_2}
 &Z_{W,A}  \bigg(\left(-\frac{1}{\tau}\right)^{L(0)}w',\left(-\frac{1}{\tau}\right)^{L(0)} f(w);\frac{z+\tau}{\tau},-\frac{1}{\tau}\bigg)\nonumber\\
 &\quad=Z_{W,A} \bigg(\left(-\frac{1}{\tau}\right)^{L(0)}w',\left(-\frac{1}{\tau}\right)^{L(0)} f(w);\frac{z}{\tau}+1,-\frac{1}{\tau}\bigg)\nonumber\\
 &\quad =\tr_A \cY^{-1}\bigg(\cU(1)\cY_{W',W}\bigg(\left(-\frac{1}{\tau}\right)^{L(0)}w',e^{\log\frac{z}{\tau}}\bigg)\left(-\frac{1}{\tau}\right)^{L(0)}e^{-2\pi i L(0)}f(w),1\bigg)q_{-1/\tau}^{L(0)-c/24}\nonumber\\
 &\quad =\tr_A \cY^{-1}\bigg(\cU(1)\left(-\frac{1}{\tau}\right)^{L(0)}\cY_{W',W}\left(w',e^{\log\frac{z}{\tau}-\log\left(-\frac{1}{\tau}\right)}\right)\theta_W^{-1}(w),1\bigg)q_{-1/\tau}^{L(0)-c/24}.
\end{align}
Since we are in particular assuming $-z\in P_\tau-1$, which means that $\arg \tau<\arg(-z)<\pi$, and since $\arg(-\frac{1}{\tau})=\pi-\arg\tau$, we have
\begin{align*}
 \log\frac{z}{\tau} & =\ln\left\vert\frac{z}{\tau}\right\vert+i\left[\arg(-z)+\arg\left(-\frac{1}{\tau}\right)-2\pi\right] =\log(-z)+\log\left(-\frac{1}{\tau}\right)-2\pi i.
\end{align*}
This allows us to simplify the right side of \eqref{eqn:MSH_reln_2}; we get further simplification from
\begin{align*}
 \cY^{-1} & =l_A\circ(e_W\tens\Id_A)\circ\cA_{W',W,A}\circ(\Id_{W'}\tens(\cR^{2}_{W,A})^{-1})\circ\cA_{W',W,A}^{-1}\circ\cY_{W'\tens W,A}\nonumber\\
 & =g^m\circ l_A\circ(e_W\tens\Id_A)\circ\cY_{W'\tens W,A}\nonumber\\
 & = g^m\circ l_A\circ\cY_{V,A}\circ((e_W\otimes\Id_A).
\end{align*}
Since $l_A\circ\cY_{V,A}=Y_A$ and $e_W$ commutes with $\cU(1)$, the right side of \eqref{eqn:MSH_reln_2} equals
\begin{align}\label{eqn:MSH_reln_3}
 \tr_A \, (g^m\circ Y_A)\bigg(\cU(1)\left(-\frac{1}{\tau}\right)^{L(0)}\cE_W\left(w',e^{\log(-z)-2\pi i}\right)\theta_W^{-1}(w),1\bigg)q_{-1/\tau}^{L(0)-c/24}.
\end{align}
Next, the definition of $\til{e}_W$ and properties of $\theta$ (namely, the balancing equation, $\theta_V=\Id_V$, and $\theta_{W'}=\theta_W'$) imply
\begin{equation*}
 e_W =\til{e}_W\circ\cR_{W,W'}^{-1}\circ(\Id_{W'}\tens\theta_W^{-1}) =\til{e}_W\circ\cR_{W',W}\circ(\Id_{W'}\tens\theta_W).
\end{equation*}
Consequently, \eqref{eqn:MSH_reln_3} equals
\begin{align*}
 \tr_A & \, (g^m\circ Y_A)\bigg(\cU(1)\left(-\frac{1}{\tau}\right)^{L(0)}(\til{e}_W\circ\cR_{W',W}\circ\cY_{W',W})\left(w',e^{\log(-z)-2\pi i}\right)w,1\bigg)q_{-1/\tau}^{L(0)-c/24}\nonumber\\
 & =\tr_A\,(g^m\circ Y_A)\bigg(\cU(1)\left(-\frac{1}{\tau}\right)^{L(0)}e^{-zL(-1)}\til{\cE}_W\left(w,e^{\log(-z)-\pi i}\right)w',1\bigg)q_{-1/\tau}^{L(0)-c/24}\nonumber\\
 & =\tr_A\,(g^m\circ Y_A)\bigg(\cU(1)e^{\frac{z}{\tau}L(-1)}\left(-\frac{1}{\tau}\right)^{L(0)}\til{\cE}_W\left(w,e^{\log(-z)-\pi i}\right)w',1\bigg)q_{-1/\tau}^{L(0)-c/24}.
\end{align*}
Now, we can remove the factor of $e^{\frac{z}{\tau}L(-1)}$ because
\begin{equation*}
\tr_A\,(g^m\circ Y_A)(\cU(1)L(-1)v,1)q_{-1/\tau}^{L(0)-c/24} =0
\end{equation*}
for all $v\in V$ by \cite[Equation 1.15]{Hu-mod_inv}, as in the proof of Theorem \ref{thm:S_ch_W} (see also \cite[Proposition 4.3.1]{Zh}). Moreover, since we are assuming that $0<\arg(-z)<\pi$, we have
\begin{equation*}
 \log(-z) =\log z+\pi i.
\end{equation*}
Therefore, we have now shown that for any $\tau\in\mathbb{H}$,
\begin{equation}\label{eqn:MSH_reln_4}
 \tr_A\,(g^m\circ Y_A)\bigg(\cU(1)\left(-\frac{1}{\tau}\right)^{L(0)}\til{\cE}_W(w,e^{\log z})w',1\bigg)q_{-1/\tau}^{L(0)-c/24}
\end{equation}
is equal to \eqref{eqn:MSH_reln_1}, and therefore also to \eqref{eqn:rig_proof_right}, on a non-empty open set of $z\in\CC$.

Since \eqref{eqn:rig_proof_right} and \eqref{eqn:MSH_reln_4} are series in powers of $z$ and $\log z$ which are equal as analytic functions on a non-empty open set of $\CC$, their coefficients are equal (see \cite[Proposition 7.8]{HLZ5}). This proves the theorem because $W\tens W'$ is spanned by vectors $w\tens_{h,k} W'$ for $w\in W$, $w'\in W'$, $h\in\QQ$, and $k\in\NN$ (see \cite[Proposition 4.23]{HLZ3}).
\end{proof}

We can now prove that the category $\cC$ of grading-restricted generalized $V$-modules is rigid:
\begin{thm}\label{thm:rigidity}
Under Assumption \ref{assum:rigidity}, the tensor category $\cC$ of grading-restricted generalized $V$-modules is rigid, with duals given by contragredients with respect to the conformal vector $\til{\omega}$. In particular, $\cC$ is a braided ribbon category with twist $\theta$ defined using $\til{\omega}$.
\end{thm}
\begin{proof}
Since every grading-restricted generalized $V$-module has finite length, it is enough to show that all simple $V$-modules are rigid (see \cite[Theorem 4.4.1]{CMY2}). Lemma \ref{lem:A_simple_monodromy} shows that Theorem \ref{thm:MSH_reln_1} applies to any simple $V$-module $W$, so we will use this theorem to check the rigidity criterion of Theorem \ref{thm:rigidity_criterion}. Since $V$ is simple, $V$ is generated as a module for itself by any non-zero $v\in V$. Thus recalling the $V$-module isomorphism $\varphi: V\rightarrow V'$, if we can show that $(\til{e}_W)'(\varphi(v))\in\im \Phi_W$ for at least one non-zero $v\in V$, then it will follow that  $(\til{e}_W)'(\varphi(v))\in\im \Phi_W$ for all $v\in V$. That is, $\im \,(\til{e}_W)'$ will be a $V$-submodule of $\im\Phi_W$ and thus $W$ will be rigid by Theorem \ref{thm:rigidity_criterion}. In fact, $\im\,(\til{e}_W)'$ will be a non-zero submodule of $\im\Phi_W$: since $W$ is simple and thus non-zero, the map $\til{e}_W$ is non-zero and thus surjective because $V$ is simple; hence its contragredient $(\til{e}_W)'$ is injective. However, it is not necessary to know that $\im\,(\til{e}_W)'\neq 0$ in order to apply Theorem \ref{thm:rigidity_criterion}.

Now to show that $(\til{e}_W)'(\varphi(v))\in\im \Phi_W$ for at least one non-zero $v\in V$, fix an arbitrary conformal weight $n\in\NN$ and choose a basis $\lbrace v_i\rbrace_{i=1}^{\dim V_{(n)}}$ of the conformal weight space $V_{(n)}$ with respect to $\til{\omega}$. Then let $\lbrace v_i'\rbrace_{i=1}^{\dim V_{(n)}}$ be the dual basis of $V_{(n)}$ in the sense that
\begin{equation*}
 \langle\varphi(v_i'),v_j\rangle =\delta_{i,j}.
\end{equation*}
Then for any $\til{w}\in(W\tens W')_{[n]}$ (again using conformal weights with respect to $\til{\omega}$),
\begin{align*}
 \til{e}_W(\til{w}) & =\sum_{i=1}^{\dim V_{(n)}}\left\langle\varphi(v_i'),\til{e}_W(\til{w})\right\rangle v_i=\sum_{i=1}^{\dim V_{(n)}} \left\langle (\til{e}_W)'(\varphi(v_i')),\til{w}\right\rangle v_i.
\end{align*}
Now for any $v\in V$, consider
\begin{equation*}
 \tr_A\,(g^m\circ Y_A)\bigg(\cU(1)\left(-\frac{1}{\tau}\right)^{L(0)} v,1\bigg)q_{-1/\tau}^{L(0)-c/24}.
\end{equation*}
By \eqref{eqn:MSH_reln_1} and the surjectivity of $\til{e}_W$, this function of $\tau$ has a series expansion in powers of $q_\tau$ with no $\log q_\tau$ terms. For an arbitrary fixed weight $h\in\QQ$ (now with respect to $\omega$), let $C_h(v)$ denote the coefficient of $q_\tau^{h-c/24}$ in this expansion. Thus Theorem \ref{thm:MSH_reln_1} says that for any $\til{w}\in(W\tens W')_{[n]}$,
\begin{align}\label{eqn:rig_calc_1}
 \sum_{i=1}^{\dim V_{(n)}} & C_h(v_i)  \left\langle(\til{e}_W)'(\varphi(v_i')), \til{w}\right\rangle\nonumber\\
 & =\sum_X s_{A,X}\,\tr_{(W\tens X)_{[h]}} \cY_{\til{\Psi}_W,X}\left(\cU(1)[\Phi_W\circ(f\tens\Id_{W'})](\til{w}),1\right)\nonumber\\
 & =\sum_X\sum_{j=1}^{\dim(W\tens X)_{[h]}} s_{A,X}\left\langle b_{X;j}', \cY_{\til{\Psi}_W,X}\left(\cU(1)[\Phi_W\circ(f\tens\Id_{W'})](\til{w}),1\right)b_{X;j}\right\rangle,
\end{align}
where $\lbrace b_{X;j}\rbrace$ is a basis of $(W\tens X)_{[h]}$ and $\lbrace b_{X;j}'\rbrace$ is the corresponding dual basis of $(W\tens X)_{[h]}^*$, which we have embedded into $(W\tens X)^*$ such that $\langle b_{X;j}',b\rangle=0$ for $b\in\bigoplus_{h'\neq h} (W\tens X)_{[h']}$.

Now, since $h$ is a conformal weight for $\omega$, the dual basis vectors $b_{X;j}'$ are elements of the graded dual of $W\tens X$ with respect to $\omega$, and then Proposition \ref{prop:graded_duals_same} says they are also vectors in the contragredient module $(W\tens X)'$ with respect to $\til{\omega}$. So we now view the bilinear pairings on the right side of \eqref{eqn:rig_calc_1} as those between $W\tens X$ and its $\til{\omega}$-contragredient. As we continue to analyze \eqref{eqn:rig_calc_1} using \eqref{eqn:Omega_intw_op_def} and \eqref{eqn:Adjoint_intw_op_def}, we will slightly abuse notation and use $L(0)$, $L(1)$ to denote the vertex operator modes for $\til{\omega}$, since $\omega$ will only be relevant to the conformal weight label of $(W\tens X)_{[h]}$. Also, because $W$ is simple, the homomorphism $f\in\Endo_V(W)$ is simply a non-zero scalar multiple $c\,\Id_W$. Thus we now rewrite \eqref{eqn:rig_calc_1} as follows:
\begin{align*}
 c^{-1} & \sum_{i=1}^{\dim V_{(n)}}  C_h(v_i)  \left\langle(\til{e}_W)'(\varphi(v_i')), \til{w}\right\rangle\nonumber\\
 & = \sum_X\sum_{j=1}^{\dim(W\tens X)_{[h]}} s_{A,X} \left\langle e^{L(1)} b_{X;j}',\Omega_0(\cY_{\til{\Psi}_W,X})\left(b_{X;j},e^{-\pi i}\right)\cU(1)\Phi_W(\til{w})\right\rangle\nonumber\\
 & =\sum_X\sum_{j=1}^{\dim(W\tens X)_{[h]}} s_{A,X}\left\langle A_0(\Omega_0(\cY_{\til{\Psi}_W,X}))\left( e^{\pi i L(0)}e^{L(1)} b_{X;j},e^{\pi i}\right)e^{L(1)}b_{X;j}',\cU(1)\Phi_W(\til{w})\right\rangle.
\end{align*}
Note that $A_0(\Omega_0(\cY_{\til{\Psi}_W,X}))$ is an intertwining operator of type $\binom{(W\tens W')''}{W\tens X\,(W\tens X)'}$; let us use $\til{\cY}_{W,X}$ to denote the intertwining operator $\delta_{W\tens W'}^{-1}\circ A_0(\Omega_0(\cY_{\til{\Psi}_W,X}))$ of type $\binom{W\tens W'}{W\tens X\,(W\tens X)'}$. We also use $\cU(1)^o$ to denote the operator on $\overline{W\tens W'} = ((W\tens W')')^*$ that is adjoint to the operator $\cU(1)$ on $(W\tens W')'$. Using this notation and Proposition \ref{prop:Phi_and_Phi_prime}, we continue to calculate
\begin{align}\label{eqn:rig_calc_2}
 & c^{-1}  \sum_{i=1}^{\dim V_{(n)}}  C_h(v_i)  \left\langle(\til{e}_W)'(\varphi(v_i')), \til{w}\right\rangle\nonumber\\
 & =\sum_X\sum_{j=1}^{\dim(W\tens X)_{[h]}} s_{A,X}\left\langle (\Phi_W'\circ\delta_{W\tens W'})(\til{w}), \cU(1)^{o}\til{\cY}_{W,X}\left(e^{\pi i L(0)} e^{L(1)} b_{X;j},e^{\pi i}\right)e^{L(1)} b_{X;j}'\right\rangle\nonumber\\
 & =\sum_X\sum_{j=1}^{\dim(W\tens X)_{[h]}} s_{A,X}\left\langle \Phi_W\left(\pi_n\left(\cU(1)^{o}\til{\cY}_{W,X}\left(e^{\pi i L(0)} e^{L(1)} b_{X;j},e^{\pi i}\right)e^{L(1)} b_{X;j}'\right)\right), \til{w}\right\rangle
\end{align}
where $\pi_n: \overline{W\tens W'}\rightarrow(W\tens W')_{[n]}$ is the the projection.

Both sides of \eqref{eqn:rig_calc_2} vanish for $\til{w}\in \bigoplus_{\til{h}\neq n} (W\tens W')_{[\til{h}]}$, so we have now shown that
\begin{equation*}
 (\til{e}_W)'\bigg(c^{-1}\sum_{i=1}^{\dim V_{(n)}} C_h(v_i)\varphi(v_i')\bigg)\in\im\Phi_W
\end{equation*}
for any $n\in\NN$, $h\in\QQ$. It remains to show that $\sum_{i=1}^{\dim V_{(n)}} C_h(v_i)\varphi(v_i')\neq 0$ for some $n$ and $h$. If not, then the linear independence of the $\varphi(v_i')$, ranging over all conformal weights $n$, would imply that $C_h(v)=0$ for all $h\in\QQ$, $v\in V$. By the definition of $C_h(v)$, this would mean
\begin{equation}\label{eqn:trace_vanish}
  \tr_A\,(g^m\circ Y_A)\bigg(\cU(1)\left(-\frac{1}{\tau}\right)^{L(0)} v,1\bigg)q_{-1/\tau}^{L(0)-c/24} =0
\end{equation}
for all $v\in V$ and $\tau\in\mathbb{H}$. Now the decomposition \eqref{eqn:A_decomp} yields
\begin{equation*}
 g^m\circ Y_A=\bigoplus_{n\in\ZZ/\vert g\vert\ZZ} e^{2\pi i mn/\vert g\vert} Y_{A_n},
\end{equation*}
so \eqref{eqn:trace_vanish} would imply $\sum_{n\in\ZZ/\vert g\vert\ZZ} e^{2\pi imn/\vert g\vert} S(\mathrm{ch}_{A_n})=0$. But since the $S$-transformation is invertible, this would contradict the linear independence of characters of non-isomorphic irreducible $V$-modules; so in fact $\sum_{i=1}^{\dim V_{(n)}} C_h(v_i)\varphi(v_i')\neq 0$ for some $n$ and $h$, as required.
\end{proof}

\begin{rem}\label{rem:coevaluation}
The proof of Theorem \ref{thm:rigidity} simplifies in the case that $A=V$, $g=\Id_V$, $\til{\omega}=\omega$, and $V_{(0)}=\CC\vac$ (that is, $V$ has positive energy). In this case, we can choose the isomorphism $\varphi: V\rightarrow V'$ to satisfy $\langle\varphi(\vac),\vac\rangle=1$, and then \eqref{eqn:rig_calc_2} specializes to
 \begin{align*}
  (\til{e}_W)'& (\varphi(\vac))\nonumber\\ &=\frac{\sum_X\sum_{j=1}^{\dim(W\tens X)_{[h]}} s_{V,X}\,\Phi_W\left(\pi_0\left(\cU(1)^{o}\til{\cY}_{W,X}\left(e^{\pi i L(0)} e^{L(1)} b_{X;j},e^{\pi i}\right)e^{L(1)} b_{X;j}'\right)\right)}{\sum_X s_{V,X}\dim X_{[h]}},
 \end{align*}
 where $h\in\QQ$ is chosen so that the denominator is non-zero. Combining this formula with the proof of Proposition \ref{prop:Phi_iso_rigid}, we obtain a somewhat explicit formula for the coevaluation:
 \begin{equation*}
  i_W(\vac)=\frac{\sum_X\sum_{j=1}^{\dim(W\tens X)_{[h]}} s_{V,X}\,\pi_0\left(\cU(1)^{o}\til{\cY}_{W,X}\left(e^{\pi i L(0)} e^{L(1)} b_{X;j},e^{\pi i}\right)e^{L(1)} b_{X;j}'\right)}{\sum_X s_{V,X}\dim X_{[h]}}.
 \end{equation*}
By working out the intertwining operators $\til{\cY}_{W,X}$ precisely, one can then obtain a formula for the categorical dimension of $W$, defined by $\til{e}_W\circ i_W=(\dim_\cC W)\Id_V$:
\begin{equation*}
 \dim_\cC W=\frac{\sum_X s_{V,X}\,\dim(W\tens X)_{[h]}}{\sum_X s_{V,X} \dim X_{[h]}}.
\end{equation*}
On the other hand, \eqref{eqn:S_ch_W} implies $\dim_\cC W=\frac{s_{W,V}}{s_{V,V}}$, where for irreducible $V$-modules $X$, the $S$-matrix entries $s_{W,X}$ are defined by $S(\mathrm{ch}_W)=\sum_X s_{W,X}\,\mathrm{ch}_X$. These two categorical dimension formulas are related by the Verlinde formula \cite{Ve, MS, Hu-Verlinde}.
\end{rem}

We note the $A=V$, $g=\Id_V$, $\til{\omega}=\omega$ case of Theorem \ref{thm:rigidity} as a corollary:
\begin{cor}\label{cor:rigidity}
 Let $V$ be a strongly finite vertex operator algebra. If $S(\mathrm{ch}_V)$ is a linear combination of characters, then the tensor category $\cC$ of grading-restricted generalized $V$-modules is rigid. In particular, $\cC$ is a finite braided ribbon category.
\end{cor}

\subsection{Rationality}\label{subsec:rat}

For this subsection, we continue in the setting of Assumption \ref{assum:rigidity}. Since the tensor category $\cC$ of grading-restricted generalized $V$-modules is rigid by Theorem \ref{thm:rigidity}, the hypotheses of Theorem \ref{thm:factorizability} hold for $V$ when we use $\til{\omega}$ as conformal vector. Thus $\cC$ is a factorizable finite ribbon category. Our goal is to prove that $\cC$ is semisimple, and therefore $V$ is rational, under the following additional assumption:
\begin{assum}\label{assum:rationality}
 The Zhu algebra of $A$ with respect to $\omega$ is semisimple.
\end{assum}

We begin with some properties of projective covers. By \cite{Hu-C2}, every irreducible $V$-module $W$ has a projective cover; recall this is a pair $(P_W,p_W)$ such that $P_W$ is a projective $V$-module, $p_W: P_W\rightarrow W$ is a surjection, and for any surjection $p: P\rightarrow W$ with $P$ projective, there is a surjection $q: P\rightarrow P_W$ such that 
\begin{equation*}
 \xymatrixcolsep{3pc}
 \xymatrix{
 & P \ar[d]^{p} \ar[ld]_{q} \\
 P_W \ar[r]_{p_W} & W\\}
\end{equation*}
commutes. We need a couple elementary lemmas about the projective covers:
\begin{lem}\label{lem:PW_indecomposable}
 For any irreducible $V$-module $W$, the projective cover $P_W$ is indecomposable.
\end{lem}
\begin{proof}
 Suppose $P_W= P_1\oplus P_2$ for two submodules $P_i\subseteq P_W$; as direct summands of a projective module, $P_1$ and $P_2$ are both projective. Moreover, for at least one $i$, 
 $$p_W\vert_{P_i}: P_i\rightarrow W$$
 is non-zero and thus surjective since $W$ is irreducible. Then by the property of the projective cover, there is a surjection
 \begin{equation*}
q: P_i\rightarrow P_W  
 \end{equation*}
such that $p_W\circ q=p_W\vert_{P_i}$. Since there is also an injection $P_i\rightarrow P_W$, and since both $P_i$ and $P_W$ are grading-restricted generalized $V$-modules with finite-dimensional weight spaces, they have the same graded dimension. Thus $P_W=P_i$ for either $i=1$ or $i=2$.
\end{proof}

\begin{lem}\label{lem:PW_generator}
 For any irreducible $V$-module $W$, the projective cover $P_W$ is generated by any $\til{w}\in P_W\setminus\ker p_W$.
\end{lem}
\begin{proof}
 For $\til{w}\in P_W\setminus\ker p_W$, let $\til{W}$ denote the submodule of $P_W$ generated by $\til{w}$. Since $p_W(\til{w})\neq 0$ and $W$ is irreducible, $p_W\vert_{\til{W}}: \til{W}\rightarrow W$ is surjective. So by projectivity of $P_W$, there exists $q: P_W\rightarrow\til{W}$ such that
 \begin{equation*}
  \xymatrixcolsep{3pc}
 \xymatrix{
 & P_W \ar[d]^{p_W} \ar[ld]_{q} \\
\til{W} \ar[r]_{p_W\vert_{\til{W}}} & W\\}
 \end{equation*}
commutes. As the inclusion $i: \til{W}\rightarrow P_W$ satisfies $p_W\circ i=p_W\vert_{\til{W}}$, we get
\begin{equation*}
 p_W\circ(i\circ q)^N=p_W
\end{equation*}
for all $N\in\NN$. Since $i\circ q$ is an endomorphism of a (finite-length) indecomposable module by Lemma \ref{lem:PW_indecomposable}, it is either an isomorphism or nilpotent by Fitting's Lemma. It cannot be nilpotent since $p_W\neq 0$, so it is an isomorphism. In particular, $i$ is surjective as well as injective, and we get $P_W=\til{W}$.
\end{proof}

We use the factorizability of $\cC$ in the next proposition:
\begin{prop}\label{prop:unimodular}
 For any irreducible $V$-module $W$, there is an injection $W\rightarrow P_W$.
\end{prop}
\begin{proof}
Since $\cC$ is factorizable, \cite[Proposition 4.5]{ENO} (see also \cite[Proposition 8.10.10]{EGNO}) shows that $\cC$ is \textit{unimodular}, which means that for any irreducible $V$-module $W$, $P_W'\cong P_{W'}$ (see \cite[Sections 6.4 and 6.5]{EGNO}). Thus the surjection $p_{W'}: P_{W'}\rightarrow W'$ induces an injection
\begin{equation*}
 W\cong (W')'\xrightarrow{p_{W'}'} P_{W'}'\cong(P_W')'\cong P_W
\end{equation*}
for any irreducible $V$-module $W$.
\end{proof}

Now to show that the rigid tensor category $\cC$ is semisimple, it is enough to show that its unit object $V$ is projective as a $V$-module (see for example \cite[Corollary 4.2.13]{EGNO}, the proof of \cite[Theorem 5.24]{CM}, or \cite[Lemma 3.6]{McR2}). We aim to prove this by showing $P_W\cong W$ for any module $W$ in a suitable subset of simple $V$-modules that contains $V$.

 To define suitable subcategories of $\cC$, consider the open Hopf link endomorphisms $h_{A_n,W}$ \eqref{eqn:open_Hopf_link} for modules $W$ in $\cC$, with $A_n$ as in \eqref{eqn:A_decomp}. Since the open Hopf link gives a representation of the tensor ring of $\cC$ on $\Endo_V(W)$ (see for example \cite[Section  3.1.3]{CG}), \eqref{eqn:simple_current} yields
 \begin{equation*}
  h_{A_{m+n},W} =h_{A_m,W}\circ h_{A_n,W}
 \end{equation*}
for $m,n\in\ZZ/\vert g\vert\ZZ$. In particular, $h_{A_1,W}^{\vert g\vert}=\Id_W$, so $h_{A_1,W}$ is diagonalizable with eigenvalues of the form $e^{2\pi i n/\vert g\vert}$, $0\leq n\leq\vert g\vert-1$, for any module $W$ in $\cC$. Also, because $h_{A_1,W}$ is a $V$-module homomorphism, the $h_{A_1,W}$-eigenspaces are $V$-module direct summands of $W$.

Because categorical dimension also gives a representation of the tensor ring of $\cC$ on $\CC$, $\dim_\cC A_1$ is also a $\vert g\vert$th root of unity. Now for $0\leq n\leq\vert g\vert-1$, define $\cC^n$ to be the full subcategory of $\cC$ whose objects $W$ satisfy 
\begin{equation*}
 \frac{1}{\dim_\cC A_1} h_{A_1,W}=e^{2\pi in/\vert g\vert}\Id_W.
\end{equation*}
By the above discussion, every indecomposable module in $\cC$ is an object of $\cC^n$ for some $n$, and $V$ is an object of $\cC^{0}$ because $h_{A_1,V}=(\dim_\cC A_1)\Id_V$. Moreover, because $h_{A_1,\bullet}$ is a natural automorphism of the identity functor on $\cC$,
\begin{equation*}
 \hom_\cC(W,X)=0
\end{equation*}
whenever $W$ and $X$ are objects of $\cC^m$ and $\cC^n$, respectively, with $m\neq n$ in $\ZZ/\vert g\vert\ZZ$. Thus as a category, $\cC$ has a direct sum decomposition $\cC=\bigoplus_{n\in\ZZ/\vert g\vert\ZZ} \cC^n$. This implies that each $\cC^n$ is closed under subquotients, and if an irreducible $V$-module is an object of $\cC^n$, then so is its projective cover.

Our goal is to show that $P_W\cong W$ when $W$ is an irreducible $V$-module in $\cC^{0}$, which will show in particular that $V$ is a projective $V$-module. We first characterize $\cC^0$ as follows:
\begin{lem}\label{lem:C0_char}
 A grading-restricted generalized $V$-module $W$ is an object of $\cC^0$ if and only if $\cR_{A_1,W}^2=\Id_{A_1\tens W}$.
\end{lem}
\begin{proof}
 If $W$ is an object of $\cC^{0}$, then
\begin{equation*}
 \Id_W=\frac{1}{\dim_{\cC} A_1} h_{A_1,W}= l_W\circ f\circ\cA_{A_1',A_1,W}\circ(\Id_{A_1'}\tens\cR_{A_1,W}^2)\circ\cA_{A_1',A_1,W}^{-1}\circ f^{-1}\circ l_W^{-1},
\end{equation*}
where because $A_1$ is a simple current, $f:=\frac{1}{\dim_\cC A_1} e_{A_1}$ is an isomorphism with inverse
\begin{equation*}
 f^{-1}=\cR_{A_1,A_1'}\circ(\theta_{A_1}\tens\Id_{A_1'})\circ i_{A_1}.
\end{equation*}
It follows that $\Id_{A_1'}\tens\cR_{A_1,W}^2=\Id_{A_1'\tens(A_1\tens W)}$, and this implies 
\begin{equation}\label{eqn:mono_calc}
 F\circ(\Id_{A_1}\tens(\Id_{A_1'}\tens\cR_{A_1,W}^2))\circ F^{-1}=\Id_{A_1\tens W},
\end{equation}
where
\begin{equation*}
 F=l_{A_1\tens W}\circ(\til{e}_{A_1}\tens\Id_{A_1\tens W})\circ\cA_{A_1,A_1',A_1\tens W};
\end{equation*}
again $\til{e}_{A_1}$ is invertible because $A_1$ is a simple current. Then by naturality of the associativity and unit isomorphisms, \eqref{eqn:mono_calc} reduces to $\cR_{A_1,W}^2=\Id_{A_1\tens W}$. Conversely, the definitions show that if $\cR_{A_1,W}^2=\Id_{A_1\tens W}$, then $h_{A_1,W}=(\dim_\cC A_1)\Id_W$, so $W$ is an object of $\cC^{0}$.
\end{proof}

Now by Lemma \ref{lem:C0_char}, the simple current property \eqref{eqn:simple_current}, and the hexagon axiom, $W$ is an object of $\cC^{0}$ if and only if $\cR_{A_n,W}^2=\Id_{A_n\tens W}$ for all $n\in\ZZ/\vert g\vert\ZZ$ (see for example \cite[Theorem 2.11(2)]{CKL}), that is, if and only if $\cR_{A,W}^2=\Id_{A\tens W}$. Then by \cite[Proposition 2.65]{CKM1}, $W$ is an object of $\cC^{0}$ if and only if the induced module $A\tens W$ (see for example \cite[Section 2.7]{CKM1}) is a grading-restricted generalized $A$-module. Note that as a $V$-module, $A\tens W\cong\bigoplus_{n\in\ZZ/\vert g\vert\ZZ} A_n\tens W$, so because $V=A_0$, $A\tens W$ contains $W$ as a $V$-module direct summand. We will use these induced modules and the assumed semisimplicity of the Zhu algebra of $A$ to prove that irreducible $V$-modules $W$ in $\cC^{0}$ are projective. Before doing so, we state two further lemmas:
\begin{lem}\label{lem:A_n_tens_W_in_C0}
 If $W$ is an object of $\cC^0$, then so is $A_n\tens W$ for $n\in\ZZ/\vert g\vert\ZZ$.
\end{lem}
\begin{proof}
 By the hexagon axiom and Lemma \ref{lem:C0_char}, it is enough to show that $\cR_{A_1,A_n}^2=\Id_{A_1\tens A_n}$. But this holds because under the isomorphism $A_1\tens A_n\cong A_{n+1}$ of \eqref{eqn:simple_current}, we may identify the tensor product intertwining operator $\cY_{A_1,A_n}$ with $Y_A\vert_{A_1\otimes A_n}$, which has no monodromy.
\end{proof}

\begin{lem}\label{lem:P_An_tens_W}
 For any simple $V$-module $W$ and $n\in\ZZ/\vert g\vert\ZZ$, $(A_n\tens P_W,\Id_{A_n}\tens p_W)$ is a projective cover of $A_n\tens W$ in $\cC$.
\end{lem}
\begin{proof}
 First, $A_n\tens P_W$ is projective because $P_W$ is projective and $A_n$ is rigid, so there is a surjection $q: A_n\tens P_W\rightarrow P_{A_n\tens W}$ such that the diagram
\begin{equation*}
 \xymatrix{
 & A_n\tens P_W \ar[ld]_q \ar[d]^{\Id_{A_n}\tens p_W} \\
 P_{A_n\tens W} \ar[r]_(.48){p_{A_n\tens W}} & A_n\tens W\\
 }
\end{equation*}
commutes. Then since $P_{A_n\tens W}$ is projective, there is an injection $\til{q}: P_{A_n\tens W}\rightarrow A_n\tens P_W$ such that $q\circ\til{q}=\Id_{P_{A_n\tens W}}$, that is, $P_{A_n\tens W}$ is a direct summand of $A_n\tens P_W$. But $A_n\tens P_W$ is indecomposable since $P_W$ is indecomposable and $A_n$ is a simple current, so $q$ is an isomorphism identifying the projective cover $(P_{A_n\tens W}, p_{A_n\tens W})$ with $(A_n\tens P_W,\Id_{A_n}\tens p_W)$.
\end{proof}

We are now ready to prove that $V$ is rational:
\begin{thm}\label{thm:orbifold_rationality}
 Under Assumptions \ref{assum:rigidity} and \ref{assum:rationality}, the category of grading-restricted generalized $V$-modules is semisimple. In particular, $V$ is rational, and if $V$ has positive energy with respect to $\til{\omega}$, then $V$ (with conformal vector $\til{\omega}$) is strongly rational.
\end{thm}
\begin{proof}
As mentioned above, we will prove that $P_W\cong W$ for all simple $V$-modules $W$ in $\cC^0$. We first consider the case that the lowest conformal weight $h_W$ of $W$ (with respect to the conformal vector $\omega$) is the lowest conformal weight of both induced $A$-modules $A\tens W$ and $A\tens P_W$. Then the $A$-module surjection $\Id_A\tens p_W: A\tens P_W\rightarrow A\tens W$ restricts to a Zhu algebra module surjection between conformal-weight-$h_W$ spaces. As the Zhu algebra of $A$ is semisimple, this surjection splits, and we get
\begin{equation*}
 \til{q}: (A\tens W)_{[h_W]}\rightarrow (A\tens P_W)_{[h_W]}
\end{equation*}
such that $(\Id_A\tens p_W)\circ \til{q}=\Id_{(A\tens W)_{[h_W]}}$. Note that $\til{q}$ is also an $A(V,\omega)$-module homomorphism when we consider $A\tens W$ and $A\tens P_W$ as $V$-modules.

We want to show that the $A(V,\omega)$-module surjection 
\begin{equation*}
 p=p_W\vert_{(P_W)_{[h_W]}}: (P_W)_{[h_W]}\rightarrow W_{[h_W]}
\end{equation*}
also splits. In fact, the decomposition \eqref{eqn:A_decomp} of $A$ as a $V$-module shows there are $V$-module homomorphisms $\iota_A: V\rightarrow A$ and $\varepsilon_A: A\rightarrow V$ such that $\varepsilon_A\circ\iota_A=\Id_V$. Now define
\begin{equation*}
 q= l_{P_W}\circ(\varepsilon_A\tens\Id_{P_W})\circ\til{q}\circ(\iota_A\tens\Id_W)\circ l_W^{-1}\vert_{W_{[h_W]}}: W_{[h_W]}\rightarrow (P_W)_{[h_W]}.
\end{equation*}
Then $q$ is an $A(V,\omega)$-module homomorphism, and it is straightforward from the definitions and naturality of the unit isomorphisms that $p\circ q=\Id_{W_{[h_W]}}$.

Now for any non-zero $w\in W_{[h_W]}$, Lemma \ref{lem:PW_generator} shows that $\til{w}=q(w)$ generates $P_W$, so that
 \begin{equation*}
  P_W =\mathrm{span}\left\lbrace v_n \til{w}\,\,\vert\,\,v\in V, n\in\ZZ\right\rbrace
 \end{equation*}
by \cite[Proposition 4.5.6]{LL}. In particular,
\begin{equation*}
 (P_W)_{[h_W]} =\mathrm{span}\left\lbrace o(v)\til{w}\,\,\vert\,\,v\in V\right\rbrace = A(V,\omega)\cdot\til{w}=q(W_{[h_W]}),
\end{equation*}
so $(P_W)_{[h_W]}\cong W_{[h_W]}$ as an $A(V)$-module. It is now clear that the injection $i: W\rightarrow P_W$ guaranteed by Proposition \ref{prop:unimodular} restricts to an $A(V)$-module isomorphism
\begin{equation*}
 i\vert_{W_{[h_W]}}: W_{[h_W]}\rightarrow (P_W)_{[h_W]}.
\end{equation*}
In particular, the vector $\til{w}$ that generates $P_W$ is contained in the image of $i$, showing that $i$ is surjective as well as injective, and therefore $P_W\cong W$.

 We now order the (finitely many) distinct simple $V$-modules $W$ in $\cC^0$ by their lowest conformal weights $h_W$ and prove inductively that each $P_W\cong W$. Since projective objects in rigid tensor categories are also injective (see \cite[Proposition 6.1.3]{EGNO}), this will also show that $W$ is injective in $\cC$. For $W$ in $\cC^0$ such that $h_W$ is minimal, $h_W$ is also the lowest conformal weight of both $A\tens W\cong\bigoplus_{n\in\ZZ/\vert g\vert\ZZ} A_n\tens W$ and $A\tens P_W\cong\bigoplus_{n\in\ZZ/\vert g\vert\ZZ} A_n\tens P_W$. This is because the composition factors of these two modules are objects of $\cC^0$ by Lemmas \ref{lem:A_n_tens_W_in_C0} and \ref{lem:P_An_tens_W}, by the fact that projective covers of simple modules in $\cC^0$ are objects of $\cC^0$, and by the closure of $\cC^0$ under subquotients. Thus the argument of the preceding paragraphs shows that $P_W\cong W$ when $h_W$ is minimal.

Now let $W$ be any simple $V$-module in $\cC^0$ and assume by induction that all simple $V$-modules in $\cC^0$ with lowest conformal weight strictly less than $h_W$ are both projective and injective. Again we consider the induced $A$-module $A\tens W\cong\bigoplus_{n\in\ZZ/\vert g\vert\ZZ} A_n\tens W$. If $h_{A_n\tens W}< h_W$ for some $n\in\ZZ/\vert g\vert\ZZ$, then $W$ is projective (and injective) because
\begin{equation*}
 W\cong V\tens W\cong(A_n'\tens A_n)\tens W\cong A_n'\tens(A_n\tens W)
\end{equation*}
and because $A_n\tens W$ is projective by the inductive hypothesis. Thus we may assume that $h_W$ is the lowest conformal weight of $A\tens W$. Then $h_W$ is also the lowest conformal weight of $A\tens P_W$ since if $A_n\tens P_W\cong P_{A_n\tens W}$ had a composition factor $X$ with $h_X< h_W$, then $X$ would be a direct summand of the indecomposable module $P_{A_n\tens W}$ not isomorphic to $A_n\tens W$ (since $X$ would be projective and injective by the inductive hypothesis). So again, the first three paragraphs of the proof show that $P_W\cong W$, so $W$ is projective and injective in $\cC$.

We have now shown that all simple $V$-modules in $\cC^0$ are projective, so in particular $V$ itself is a projective $V$-module. Since any rigid tensor category is semisimple if its unit object is projective (see \cite[Corollary 4.2.13]{EGNO}, the proof of \cite[Theorem 5.24]{CM}, or \cite[Lemma 3.6]{McR2}), we conclude that $\cC$ is semisimple. Finally, Lemma 3.6 and Proposition 3.7 of \cite{CM} (see also \cite[Proposition 4.16]{McR}) show that $V$ is rational.
\end{proof}

Recall that Zhu \cite{Zh} showed that if $A$ is rational and $C_2$-cofinite, then the linear span of the characters of $A$ is indeed closed under the $S$-transformation. However, as is observed for instance in \cite[Theorem 5.1]{AvE}, rationality of $A$ is not necessary in Zhu's proof: semisimplicity of the Zhu algebra is sufficient. Thus the assumption on $S(\mathrm{ch}_A)$ in Assumption \ref{assum:rigidity} is automatic under Assumption \ref{assum:rationality}.
Consequently, the $A=V$, $g=\Id_V$, $\til{\omega}=\omega$ case of Theorem \ref{thm:orbifold_rationality} yields our final main result:
\begin{cor}\label{cor:rationality}
 If $V$ is a strongly finite vertex operator algebra with semisimple Zhu algebra, then $V$ is rational. In particular, the category of grading-restricted generalized $V$-modules is a semisimple modular tensor category, and if $V$ has positive energy, then $V$ is strongly rational.
\end{cor}

\begin{rem}
 Theorem \ref{thm:orbifold_rationality} also recovers Carnahan and Miyamoto's cyclic orbifold rationality result \cite{CM}, with somewhat weaker hypotheses and a somewhat different proof. As noted previously, the calculations in Theorem \ref{thm:MSH_reln_1} are essentially equivalent to much of the work in \cite[Sections 5.3 and 5.4]{CM}, but assuming $A$ is rational (not just that its Zhu algebra is semisimple), \cite[Theorem 4.4]{CM} gives a simpler direct argument that $V$ is projective as a $V$-module. (This argument was generalized to non-orbifold-type extensions of $V$ in \cite{McR2}, and we will use this generalization to study coset vertex operator algebras in Subsection \ref{subsec:cosets}.)
 
 If one knows directly that $V$ is projective, there is another argument to prove $\cC$ is rigid easier than that in Theorem \ref{thm:rigidity}. Indeed, projectivity of $V$ implies there is a coevaluation candidate $i_W: V\rightarrow W\tens W'$ such that $\til{e}_W\circ i_W=\Id_V$. Then taking $\til{w}=i_W(v)$, and assuming $\omega=\til{\omega}$ for simplicity, \eqref{eqn:MSH_reln_1} simplifies to
 \begin{align*}
  \tr_A\,(g^m\circ Y_A) & \bigg(\cU(1)\left(-\frac{1}{\tau}\right)^{L(0)}v,1\bigg)q_{-1/\tau}^{L(0)-c/24}\nonumber\\
  &=\sum_X s_{A,X} \tr_{W\tens X} ((\mathfrak{R}_W\tens\Id_X)\circ Y_{W\tens X})(\cU(1)v,1)q_\tau^{L(0)-c/24},
 \end{align*}
where $\mathfrak{R}_W$ is the rigidity composition defined in the proof of Proposition \ref{prop:Phi_iso_rigid}. Since the left side is non-zero by the argument concluding the proof of Theorem \ref{thm:rigidity}, $\mathfrak{R}_W$ must also be non-zero. Assuming as we may that $W$ is simple, $\mathfrak{R}_W$ is then a non-zero scalar multiple of $\Id_W$, and we can rescale $i_W$ so that $\mathfrak{R}_W=\Id_W$. In particular, we do not need the homomorphism $\Phi_W$ in this proof of rationality for finite cyclic orbifold vertex operator algebras.
\end{rem}

\begin{rem}\label{rem:pt_not_necessary}
 In the proof of Theorem \ref{thm:orbifold_rationality}, we used pseudo-traces only to show that the category $\cC$ of $V$-modules is factorizable (via Theorem \ref{thm:factorizability}).
 However, in the setting of Assumption \ref{assum:rigidity}, we do not really need pseudo-traces to prove Theorem \ref{thm:factorizability}. Indeed, repeating the calculations in the proof of Theorem \ref{thm:S_ch_W} with $V$ replaced by $A$ yields
 \begin{align*}
  S(\mathrm{ch}_{W\tens A})(v;\tau) =\sum_X s_{A,X}\,\tr_X\,(h_{W,X}\circ Y_X)(\cU(1)v,1)q_\tau^{L(0)-c/24},
 \end{align*}
with $h_{W,X}$ defined using $e^{2\pi i L(0)}$ (with respect to $\omega$) rather than the twist $\theta_W$ (with respect to $\til{\omega}$). Similar to the proof of Theorem \ref{thm:factorizability}, this shows that if $W$ is an object in the M\"{u}ger center of $\cC$, then $\mathrm{ch}_{W\tens A}$ is a multiple of $\mathrm{ch}_A$, that is,
\begin{equation*}
 \sum_{m=0}^{\vert g\vert-1} \mathrm{ch}_{W\tens A_m} =c\sum_{n=0}^{\vert g\vert-1} \mathrm{ch}_{A_n}
\end{equation*}
for some $c\in\CC$. 

Because $\mathrm{ch}_W$ appears on the left side of the above equation and because characters of distinct irreducible $V$-modules are linearly independent, all composition factors of $W$ must come from the $A_n$ for $n\in\ZZ/\vert g\vert\ZZ$. But $A_0=V$ is the only $A_n$ in the M\"{u}ger center of $\cC$, because if $X$ is an irreducible $V$-submodule of a $g$-twisted $A$-module (which exists by \cite[Theorem 9.1]{DLM}), then the simple current property of $A_n$ and the definition of $g$-twisted module implies $\cR_{A_n,X}^2=e^{-2\pi i n/\vert g\vert}\Id_{A_n\tens X}$. It is easy to see using exactness of the tensor product on $\cC$ and naturality of the monodromy isomorphisms that the M\"{u}ger center is closed under subquotients, so $V$ is the only composition factor of any object in the M\"{u}ger center. This implies $\cC$ is factorizable, as in the proof of Theorem \ref{thm:factorizability}.
\end{rem}

\section{Applications}\label{sec:applications}

In this section, we use Corollary \ref{cor:rigidity} and Theorem \ref{thm:orbifold_rationality} to prove rationality for some examples of $C_2$-cofinite vertex operator algebras.

\subsection{Rationality for \texorpdfstring{$C_2$}{C2}-cofinite \texorpdfstring{$W$}{W}-algebras}\label{subsec:W-algebras}

Affine $W$-algebras are a large class of vertex operator (super)algebras that generalize vertex operator (super)algebras based on affine Lie (super)algebras. For the discussion here, we mainly use \cite{AvE} as a reference. 

Given a simple Lie algebra $\mathfrak{g}$, a nilpotent element $f\in\mathfrak{g}$, and a level $k\in\CC$, the universal affine $W$-algebra $\cW^k(\mathfrak{g},f)$ is the quantum Drinfeld-Sokolov reduction, relative to $f$, of the universal affine vertex operator algebra $V^k(\mathfrak{g})$ associated to $\mathfrak{g}$ at level $k$ \cite{FF, KRW}. Then $\cW_k(\mathfrak{g},f)$ is defined to be the unique simple quotient of $\cW^k(\mathfrak{g},f)$. In \cite{KW}, Kac and Wakimoto conjectured that a class of simple affine $W$-algebras called \textit{exceptional} are rational; later in \cite{Ar-C2}, Arakawa proved that a larger class of simple $W$-algebras are $C_2$-cofinite and conjectured that these are also rational. In \cite{Ar-rat}, Arakawa proved that exceptional $W$-algebras associated to principal nilpotent elements are rational, and several further cases have been proved since then; see Remark \ref{rem:W-alg-rat-progress} below.

To define exceptional $W$-algebras, recall that a level $k$ is \textit{admissible} for $\mathfrak{g}$ if it has the form
\begin{equation*}
 k=-h^\vee+\frac{p}{q}
\end{equation*}
where $h^\vee$ is the dual Coxeter number of $\mathfrak{g}$, $p$ and $q$ are relatively prime positive integers, $p\geq h^\vee$ if $q$ is relatively prime to the lacing number of $\mathfrak{g}$, and $p\geq h$ (the Coxeter number of $\mathfrak{g}$) if the lacing number of $\mathfrak{g}$ divides $q$. In \cite{Ar-C2}, Arakawa showed that the associated variety \cite{Ar-assoc-var} of the simple admissible-level affine vertex operator algebra $L_k(\mathfrak{g})$ is the closure of a certain nilpotent orbit $\mathbb{O}_q\subseteq\mathfrak{g}$ depending only on $q$, and that $\cW_k(\mathfrak{g},f)$ is $C_2$-cofinite if $f\in\mathbb{O}_q$. The pair $(f,q)$ is called \textit{exceptional} if $f\in\mathbb{O}_q$, extending the notion of exceptional pair in \cite{KW}, and following \cite{AvE}, we say that $\cW_k(\mathfrak{g},f)$ is an \textit{exceptional $W$-algebra} if $k$ is admissible and $f\in\mathbb{O}_q$. Then the Kac-Wakimoto-Arakawa conjecture \cite{KW, Ar-C2} is that all exceptional simple affine $W$-algebras are rational.
\begin{exam}
For $\mathfrak{g}=\mathfrak{sl}_2$, there is only one conjugacy class of non-zero nilpotent elements $f$, and $f\in\mathbb{O}_q$ for $q\geq 2$. The exceptional $W$-algebras at levels $k=-2+\frac{p}{q}$ for $\mathrm{gcd}(p,q)=1$ and $p,q\geq 2$ are the simple rational Virasoro vertex operator algebras at central charge $1-\frac{6(p-q)^2}{pq}$. For $q=1$, the nilpotent orbit $\mathbb{O}_1$ is $0$, and the corresponding affine $W$-algebras are the simple affine vertex operator algebras $L_k(\mathfrak{sl}_2)$ at level $k\in\NN$.
\end{exam}

Whether a simple affine $W$-algebra $\cW_k(\mathfrak{g},f)$ is $\NN$-graded and self-contragredient depends on its conformal vector, which exists as long as $k$ is non-critical, that is, $k\neq -h^\vee$, but is not unique. The most often used conformal vectors in $\cW_k(\mathfrak{g},f)$ come from semisimple elements $x_0\in\mathfrak{g}$ that determine \textit{good gradings} of $\mathfrak{g}$:
\begin{itemize}
 \item $\mathfrak{g}=\bigoplus_{j\in\frac{1}{2}\ZZ} \mathfrak{g}_j$ where $\mathfrak{g}_j=\lbrace x\in\mathfrak{g}\,\,\vert\,\,[x_0,x]=jx\rbrace$.
 
 \item $f\in\mathfrak{g}_{-1}$ and $\mathrm{ad}(f):\mathfrak{g}_j\rightarrow\mathfrak{g}_{j-1}$ is injective for $j\geq\frac{1}{2}$ and surjective for $j\leq\frac{1}{2}$.
\end{itemize}
In the notation of \cite{KRW}, $\cW_k(\mathfrak{g},x_0,f)$ represents $\cW_k(\mathfrak{g},f)$ as a vertex operator algebra with conformal vector $\omega_{x_0}$ constructed from $x_0$. By \cite[Theorem 4.1]{KW2}, $\omega_{x_0}$ gives $\cW_k(\mathfrak{g},x_0,f)$ a $\frac{1}{2}\NN$-grading, and $\cW_k(\mathfrak{g},f)$ is $\NN$-graded (and positive-energy) with respect to this conformal vector if $x_0$ induces a good \textit{even} grading of $\mathfrak{g}$, that is, $\mathfrak{g}_j=0$ for $j\notin\ZZ$. One can always obtain a good grading for any non-zero nilpotent element $f\in\mathfrak{g}$ by embedding $f$ into an $\mathfrak{sl}_2$-triple $\lbrace e,h,f\rbrace$ and taking $x_0=\frac{1}{2}h$. This gives the \textit{Dynkin grading} of $\mathfrak{g}$ associated to $f$, and we will call the associated conformal vector $\omega_{h/2}$ the Dynkin conformal vector. If the Dynkin grading of $\mathfrak{g}$ is even, so that $\cW_k(\mathfrak{g},\frac{1}{2}h,f)$ has positive energy, then we say that $f$ is an even nilpotent element of $\mathfrak{g}$.

Now Proposition 6.1 and Remark 6.2 of \cite{AvE} give a criterion for the $\NN$-graded subalgebra $\cW_k^0(\mathfrak{g},x_0,f)\subseteq\cW_k(\mathfrak{g},x_0,f)$ to be self-contragredient. Let $\mathfrak{g}^f$ be the centralizer of $f$ in $\mathfrak{g}$; the subalgebra $\mathfrak{g}^f$ is $\mathrm{ad}(x_0)$-invariant, so the $\frac{1}{2}\ZZ$-grading on $\mathfrak{g}$ restricts to a $\frac{1}{2}\ZZ$-grading on $\mathfrak{g}^f$. Then $\cW_k(\mathfrak{g},f)$, with the conformal vector determined by $x_0$, is self-contragredient if
\begin{equation}\label{eqn:W_alg_self_contra}
 (k+h^\vee)\langle x_0, v\rangle -\frac{1}{2}\tr_{\mathfrak{g}_+} \mathrm{ad}(v) =0 \quad\text{for all}\quad v\in\mathfrak{g}^f_0,
\end{equation}
where $\langle\cdot,\cdot\rangle$ is the nondegenerate invariant bilinear form on $\mathfrak{g}$ such that $\langle\alpha,\alpha\rangle=2$ for long roots $\alpha$ and $\mathfrak{g}_+=\bigoplus_{j>0} \mathfrak{g}_j$. 
Proposition 4.2 and Lemma 5.5 of \cite{AvEM} show that \eqref{eqn:W_alg_self_contra} always holds for the Dynkin conformal vector:
\begin{prop}\label{prop:self-contra-cond}
 If $\mathfrak{g}$ is a simple Lie algebra and $f\in\mathfrak{g}$ is a non-zero nilpotent element contained in an $\mathfrak{sl}_2$-triple $\lbrace e, h, f\rbrace$, then \eqref{eqn:W_alg_self_contra} holds for $x_0=\frac{1}{2}h$. In particular, $\cW_k^0(\mathfrak{g},\frac{1}{2}h, f)$ is a positive-energy self-contragredient vertex operator algebra for any simple Lie algebra $\mathfrak{g}$, nilpotent element $f\in\mathfrak{g}$, and level $k\neq -h^\vee$.
\end{prop}

In \cite[Theorem 4.3]{AvE}, Arakawa and van Ekeren showed that if $k=-h^\vee+\frac{p}{q}$ is an admissible level for $\mathfrak{g}$, $f\in\mathbb{O}_q$ is a nilpotent element, and $x_0\in\mathfrak{g}$ is a semisimple element that induces a good grading of $\mathfrak{g}$ relative to $f$, then the Zhu algebra $A(\cW_k(\mathfrak{g},x_0,f))$ is semisimple. In light of this result and Proposition \ref{prop:self-contra-cond}, Corollary \ref{cor:rationality} now proves the Kac-Wakimoto-Arakawa rationality conjecture for all exceptional $W$-algebras $\cW_k(\mathfrak{g},f)$ such that $f$ is an even nilpotent element. In fact, if $f$ is even, then $\cW_k(\mathfrak{g},f)$ (with the Dynkin conformal vector) is strongly rational. Even nilpotent elements include principal nilpotent elements in all types and all subregular nilpotent elements outside of type $A$. Note that exceptional $W$-algebras associated to principal nilpotent elements are already known to be rational \cite{Ar-rat}, and the same holds for subregular nilpotent elements in simply-laced types \cite{AvE} and some cases in type $B$ \cite{Fa, CL-osym-trialities}.

If $f$ is not even, then we can still use Theorem \ref{thm:orbifold_rationality} to show that $\cW_k^0(\mathfrak{g},\frac{1}{2}h,f)$ is strongly rational, provided we can find a suitable conformal vector making $\cW_k(\mathfrak{g},f)$ $\NN$-graded. This reduces to a Lie algebraic question which is independent of $k$:
\begin{thm}\label{thm:W_alg_rat}
 Let $\mathfrak{g}$ be a finite-dimensional simple Lie algebra, $k\neq-h^\vee$ a non-critical level for $\mathfrak{g}$, and $f\in\mathfrak{g}$ a nilpotent element, and let $\mathfrak{g}=\bigoplus_{j\in\frac{1}{2}\ZZ} \mathfrak{g}_j$ be the Dynkin grading of $\mathfrak{g}$ associated to $f$, given by $\frac{1}{2}h$-eigenvalues. Assume there is a semisimple element $v\in\mathfrak{g}^f_0$ such that for $j\geq 0$, the eigenvalues of $\mathrm{ad}(v)$ on $\mathfrak{g}^f_{-j}$ are contained in $-j-1+\NN$. Then if $\cW_k(\mathfrak{g},f)$ is $C_2$-cofinite and $A(\cW_k(\mathfrak{g},f),\omega_{h/2})$ is semisimple, then $\cW_k^0(\mathfrak{g},\frac{1}{2}h,f)$ is strongly rational, and $\cW_k(\mathfrak{g},\frac{1}{2}h,f)$ is strongly rational as a $\frac{1}{2}\NN$-graded vertex operator algebra.
\end{thm}

\begin{proof}
 To apply Theorem \ref{thm:orbifold_rationality}, we need to check that Assumption \ref{assum:rigidity} holds. In the notation of that assumption, we will take $A=\cW_k(\mathfrak{g},f)$, $V=\cW_k^0(\mathfrak{g},\frac{1}{2}h,f)$, $\til{\omega}=\omega_{h/2}$, and we will use the semisimple element $v\in\mathfrak{g}^f_0$ to construct a conformal vector $\omega\in\cW_k^0(\mathfrak{g},\frac{1}{2}h,f)$ such that $\cW_k(\mathfrak{g},f)$ is $\NN$-graded with respect to $\omega$ and such that $A(\cW_k(\mathfrak{g},f),\omega)$ is semisimple.
 
 To obtain $\omega$, we recall structural results for $\cW_k(\mathfrak{g},f)$ from \cite[Theorem 4.1]{KW2}: $\cW_k(\mathfrak{g},f)$ is strongly generated by vectors $J^{\lbrace x\rbrace}$ labeled by $x\in\mathfrak{g}^f_{-j}$, $j\geq 0$, of conformal weight $j+1$ with respect to $\omega_{h/2}$.
That is, $\cW_k(\mathfrak{g},f)$ is spanned by vectors of the form
 \begin{equation*}
   J^{\lbrace x_1\rbrace}_{-n_1}\cdots J^{\lbrace x_m\rbrace}_{-n_m}\vac
 \end{equation*}
for $m\geq 0$ and $n_i>0$, $1\leq i\leq m$. It is observed in \cite[Proposition 3.1(1)]{CL-sln-trialities} that the proof of \cite[Theorem 4.1]{KW2} shows one may take $J^{\lbrace x\rbrace}$ such that
\begin{equation}\label{eqn:ad_v_cond}
 J^{\lbrace v\rbrace}_0 J^{\lbrace x\rbrace} =J^{\lbrace [v,x]\rbrace}
\end{equation}
for $v\in\mathfrak{g}_0^f$ and $x\in\mathfrak{g}_{-j}^f$, $j\geq 0$.

Now given $v$ as assumed in the theorem, we set $\omega=\omega_{h/2}-J^{\lbrace v\rbrace}_{-2}\vac$. The conditions \eqref{eqn:new_conf_cond} hold for $J^{\lbrace v\rbrace}$ by \cite[Theorem 2.4]{KRW} (see also \cite[Theorem 2.1(b), (c)]{KW2}) and Proposition \ref{prop:self-contra-cond}, so $\omega$ is a conformal vector in $\cW_k^0(\mathfrak{g},\frac{1}{2}h,f)$. Moreover, the condition on the eigenvalues of $\mathrm{ad}(v)$ together with \eqref{eqn:ad_v_cond} imply that $\cW_k(\mathfrak{g},f)$ is strongly generated by vectors of non-negative integral conformal weight with respect to $\omega$. Thus $\cW_k(\mathfrak{g},f)$ is $\NN$-graded with respect to $\omega$, and its $\omega$-conformal weight spaces are finite dimensional by Lemma \ref{lem:Miy_spanning_set} and Remark \ref{rem:span_set_general} since $\cW_k(\mathfrak{g},f)$ is $C_2$-cofinite. Finally, $A(\cW_k(\mathfrak{g}, f),\omega)$ is semisimple by Proposition \ref{prop:Zhu_algebra_iso}.

It now follows from Theorem \ref{thm:orbifold_rationality} that $\cW_k^0(\mathfrak{g},\frac{1}{2}h,f)$ is strongly rational. The assertion that $\cW_k(\mathfrak{g},f)$, with the Dynkin conformal vector, is strongly rational as a $\frac{1}{2}\NN$-graded vertex operator algebra follows from \cite[Corollary 5.5]{McR2}.
\end{proof}

\begin{rem}
 The reason we state Theorem \ref{thm:W_alg_rat} for non-exceptional $W$-algebras is that certain non-admissible-level minimal $W$-algebras are known to be $C_2$-cofinite \cite{ArM1, ArM2}. Rationality also holds in at least some of these cases \cite{Ka}, so it is reasonable to expect rationality for these $C_2$-cofinite minimal $W$-algebras in general.
 \end{rem}

We next show that Theorem \ref{thm:W_alg_rat} applies to all exceptional $W$-algebras for which $\mathfrak{g}$ admits a good even grading for the nilpotent element $f$:
\begin{cor}\label{cor:good_even_grad}
 Let $\mathfrak{g}$ be a simple Lie algebra, $k=-h^\vee+\frac{p}{q}$ an admissible level for $\mathfrak{g}$, and $f\in\mathbb{O}_q$ a nilpotent element that admits a good even grading. Then $\cW_k(\mathfrak{g},f)$, with the Dynkin conformal vector, is strongly rational.
\end{cor}
\begin{proof}
Let $\mathfrak{s}=\lbrace e, h, f\rbrace$ be an $\mathfrak{sl}_2$-triple containing $f$ and let $\mathfrak{g}=\bigoplus_{j\in\frac{1}{2}\ZZ} \mathfrak{g}_j$ be the Dynkin grading of $\mathfrak{g}$. Also let $x_0\in\mathfrak{g}$ be a semisimple element such that the $\mathrm{ad}(x_0)$-eigenvalue grading $\mathfrak{g}=\bigoplus_{j\in\ZZ} \til{\mathfrak{g}}_j$ is a good even grading for $f$. By the definition of good grading, 
\begin{equation}\label{eqn:good_grad_prop}
\mathfrak{g}^f\subseteq\bigoplus_{j\leq 0} \til{\mathfrak{g}}_j.
\end{equation}
 By \cite[Theorem 1.1]{EK}, $v=\frac{1}{2}h-x_0$ is in the center of the $\mathfrak{sl}_2$-centralizer $\mathfrak{g}^{\mathfrak{s}}$, which means that $[h,x_0]=0$ and thus $v$, as the difference of commuting semisimple elements, is semisimple. This also means that $v\in\mathfrak{g}_0^f$. Moreover, if $x\in\mathfrak{g}_{-j}^f$, $j\geq 0$, is an $\mathrm{ad}(v)$-eigenvector, then $x$ is also an $\mathrm{ad}(x_0)$-eigenvector with eigenvalue $-\til{j}\in-\NN$ by \eqref{eqn:good_grad_prop}. Thus
 \begin{equation*}
  [v,x] =\left[\frac{1}{2}h-x_0,x\right]=(-j+\til{j})x,
 \end{equation*}
showing that the $\mathrm{ad}(v)$-eigenvalues on $\mathfrak{g}_{-j}^f$ are contained in $-j-1+\ZZ_+$. Thus we may take $v$ as the semisimple element in Theorem \ref{thm:W_alg_rat}.
\end{proof}

To prove rationality for all exceptional $W$-algebras, it remains to check the assumption of Theorem \ref{thm:W_alg_rat} for all odd nilpotent orbits $\mathbb{O}_q$ that do not admit good even gradings. Such nilpotent orbits can be determined from the tables and results in \cite{Ar-C2} and \cite{EK}: besides a number of cases for low values of $q$ in classical types $B$, $C$, and $D$, there are fifteen cases in exceptional types. We check all these odd nilpotent orbits in Appendix \ref{app:nilpotent_check}, using explicit nilpotent orbit representatives from \cite{dG} for exceptional types; thus we conclude the following consequence of \cite[Theorem 5.9.1(ii)]{Ar-C2}, \cite[Theorem 4.3]{AvE}, and Theorem \ref{thm:W_alg_rat}:
\begin{thm}\label{thm:ex_W_alg_rat}
 Let $\mathfrak{g}$ be a simple Lie algebra, $k=-h^\vee+\frac{p}{q}$ an admissible level for $\mathfrak{g}$, and $f\in\mathbb{O}_q$ a nilpotent element. Then with the Dynkin conformal vector, both $\cW_k(\mathfrak{g},f)$ and its $\ZZ$-graded subalgebra are strongly rational.
\end{thm}

 \begin{rem}\label{rem:W-alg-rat-progress}
 Theorem \ref{thm:ex_W_alg_rat} has already been proved in many, but far from all, cases:
 \begin{itemize}
  \item For the trivial nilpotent element $f=0$, corresponding to denominator $q=1$, the exceptional $W$-algebras are the simple affine vertex operator algebras $L_k(\mathfrak{g})$, proved rational by Frenkel and Zhu \cite{FZ}.
  
  \item For $\mathfrak{g}=\mathfrak{sl}_2$ and $f\neq 0$, the exceptional $W$-algebras are the simple Virasoro vertex operator algebras at central charges $1-\frac{6(p-q)^2}{pq}$, proved rational by Wang \cite{Wan}.
  
  \item Arakawa \cite{Ar-BP} and Creutzig-Linshaw \cite{CL-coset} showed that exceptional $W$-algebras for $\mathfrak{g}=\mathfrak{sl}_3,\mathfrak{sl}_4$ and $f$ a subregular nilpotent element are rational.
  
  \item Arakawa \cite{Ar-rat} showed that for any $\mathfrak{g}$, exceptional $W$-algebras associated to a principal nilpotent element are rational.
  
  \item Arakawa-van Ekeren \cite{AvE} proved rationality for exceptional $\NN$-graded $W$-algebras subject to a technical condition, which they checked for all relevant nilpotent elements in type $A$ and for subregular nilpotent elements in types $D$ and $E$.
  
  \item Fasquel \cite{Fa} proved rationality of the exceptional subregular $W$-algebras for $\mathfrak{g}=\mathfrak{sp}_4$.
  
  \item In \cite{CL-sln-trialities}, Creutzig and Linshaw gave another rationality proof for subregular $W$-algebras in type $A$.
  
  \item In \cite{CL-osym-trialities}, Creutzig and Linshaw proved rationality for many subregular $W$-algebras in type $B$ and for all exceptional minimal $W$-algebras in type $C$.
 \end{itemize}
\end{rem}

\begin{rem}
 Since exceptional $W$-algebras $\cW_k(\mathfrak{g},f)$ are strongly rational by Theorem \ref{thm:ex_W_alg_rat}, tensor products of irreducible $\cW_k(\mathfrak{g},f)$-modules can be computed in terms of the $S$-transformation of characters of $\cW_k(\mathfrak{g},f)$-modules using the Verlinde formula. When $\cW_k(\mathfrak{g},f)$ is $\NN$-graded with the Dynkin conformal vector, this is the ordinary Verlinde formula of \cite{Ve, MS, Hu-Verlinde}. When $\cW_k(\mathfrak{g},f)$ is $\frac{1}{2}\NN$-graded with the Dynkin conformal vector, fusion rules for both untwisted and Ramond twisted $\cW_k(\mathfrak{g},f)$-modules are of interest, and Verlinde formulae for computing these are given in \cite[Section 1.5.2]{CKM1}.
\end{rem}

\subsection{Rationality for \texorpdfstring{$C_2$}{C2}-cofinite cosets}\label{subsec:cosets}

In this subsection, we will reduce the ``coset rationality problem'' to the problem of $C_2$-cofiniteness for the coset. That is, given a vertex operator algebra inclusion $U\otimes V\hookrightarrow A$ where $A$, $U$ are strongly rational and $U$, $V$ are mutual commutant subalgebras in $A$, we will show that $V$ is also strongly rational provided it is $C_2$-cofinite. In more detail, we work in the following setting:
\begin{itemize}
 \item $(A,Y_A,\vac,\omega)$ is a strongly rational vertex operator algebra.
 \item $(U, Y_U=Y_A\vert_{U\otimes U},\vac,\omega_U)$ is a vertex subalgebra of $A$ which is a strongly rational vertex operator algebra for some conformal vector $\omega_U\in U\cap A_{(2)}$ such that $L(1)\omega_U=0$.
\end{itemize}
In this setting, \cite[Theorem 5.1]{FZ} shows that the \textit{coset} or \textit{commutant} of $U$ in $A$,
\begin{equation*}
 V = C_A(U) :=\lbrace v\in A\,\,\vert\,\,[Y_A(u,x_1), Y_A(v,x_2)]=0\,\,\text{for all}\,\,u\in U\rbrace,
\end{equation*}
is a vertex operator algebra with vertex operator $Y_V=Y_A\vert_{V\otimes V}$, vacuum $\vac$, and conformal vector $\omega_V=\omega-\omega_U$. We will make two further assumptions on $U$ and $V$:
\begin{itemize}
 \item The vertex subalgebras $U$ and $V$ form a dual pair in $A$, that is, both $V=C_A(U)$ and $U=C_V(A)$. Equivalently, $U$ is its own double commutant:
 $U=C_A(C_A(U))$.
 
 \item The vertex operator algebra $V$ is $C_2$-cofinite.
\end{itemize}
As discussed in \cite[Section 2.3]{McR-coset}, the $L_V(0)$-eigenvalue grading of $V$ is compatible with the $L(0)$-eigenvalue grading of $A$, that is, $V_{(n)}=V\cap A_{(n)}$ for any $n\in\ZZ$. Thus $V$ is a positive-energy vertex operator algebra. Moreover, $V$ is simple (see for example \cite[Theorem 2.5(4)]{McR-coset}), so to show that $V$ is self-contragredient, \cite[Corollary 3.2]{Li} implies it is enough to show $L_V(1)V_{(1)}=0$. In fact, because $A$ is self-contragredient and positive-energy, we have $L(1)A_{(1)}=0$, and then for any $v\in V_{(1)}\subseteq A_{(1)}$, we have
\begin{equation*}
 L_V(1)v =L(1)v-L_U(1)v_{-1}\vac=-v_{-1}L_U(1)\vac =0,
\end{equation*}
as required.

We need the following basic result that is proved, for example, in \cite[Appendix A]{McR-coset}:
\begin{thm}\label{thm:gen_coset}
 In the setting of this subsection,
 \begin{enumerate}
  \item There is an injective vertex operator algebra homomorphism $\psi: U\otimes V\rightarrow A$, given by $\psi(u\otimes v)=u_{-1} v$, which gives $A$ the structure of a $U\otimes V$-module.
  
  \item There is a $U\otimes V$-module isomorphism 
  \begin{equation*}
   F: \bigoplus_{i\in I} U_i\otimes V_i\longrightarrow A
  \end{equation*}
where the $U_i$ are the distinct irreducible $U$-modules occurring in $A$, with $U=U_0$ for some $0\in I$, and the $V_i$ are grading-restricted $V$-modules given by $V_i=\hom_U(U_i,A)$. 

\item The $V$-module $V_0=\hom_U(U,A)$ is isomorphic to $V$, and $\dim\hom_V(V,V_i)=\delta_{i,0}$ for all $i\in I$.

\item For any grading-restricted generalized $U\otimes V$-module $X$ (for example, $X$ could be an $A$-module), there is a $U\otimes V$-module isomorphism
\begin{equation*}
 F_X: \bigoplus_{j\in J} M_j\otimes W_j\longrightarrow X
\end{equation*}
where the $M_j$ are the distinct irreducible $U$-modules occurring in $X$ and the $W_j$ are grading-restricted $V$-modules given by $W_j=\hom_U(M_j,X)$.
 \end{enumerate}
\end{thm}

We will use $\cC_U$, $\cC_V$,  and $\cC$ to the denote the categories of (grading-restricted generalized) $U$-, $V$-, and $U\otimes V$-modules, respectively. By assumption, $\cC_U$ is a semisimple modular tensor category, while $\cC_V$ is a braided tensor category. Our next goal is to use Corollary \ref{cor:rigidity} to show that $\cC_V$ is rigid. Thus we need to show that $S(\mathrm{ch}_V)$ is a linear combination of characters. We will in fact prove a little more in the next theorem using the multivariable trace functions of \cite{KM}; this result may be compared to the no-pseudo-trace argument in Lemma 5.16 and Proposition 5.17 of \cite{CM} for cyclic orbifold vertex operator algebras. Recall from Theorem \ref{thm:gen_coset} that the $V_i$ are $V$-submodules of $A$ and $V\cong V_0$:
\begin{thm}\label{thm:KM_app}
 In the setting of this subsection, each $S(\mathrm{ch}_{V_i})$ for $i\in I$ is a linear combination of characters of $V$-modules.
\end{thm}
\begin{proof}
 For $u\in U$, $v\in V$, $\tau_U,\tau_V\in\mathbb{H}$, and an $A$-module $X$, consider the two-variable trace function 
 \begin{equation*}
\mathrm{ch}_X(u\otimes v; \tau_U,\tau_V) =  \tr_X Y_X(\cU(1)\psi(u\otimes v),1)q_{\tau_U}^{L_U(0)-c_U/24} q_{\tau_V}^{L_V(0)-c_V/24}.
 \end{equation*}
From the definition \eqref{eqn:U1_def} of $\cU(1)$,
\begin{equation*}
 \cU(1)\psi(u\otimes v) =\psi(\cU(1)u\otimes \cU(1)v).
\end{equation*}
Then using Theorem \ref{thm:gen_coset}(4) and properties of traces, we get
\begin{equation*}
 \mathrm{ch}_X(u\otimes v; \tau_U,\tau_V) =\sum_{j\in J} \mathrm{ch}_{M_j}(u;\tau_U)\mathrm{ch}_{W_j}(v;\tau_V).
\end{equation*}
Since $U$ is strongly rational and $V$ is positive-energy $C_2$-cofinite, the two-variable trace function converges absolutely to a holomorphic function on $\mathbb{H}\times\mathbb{H}$ for all $u\in U$, $v\in V$, and simple $A$-modules $X$. Then the hypotheses of \cite[Theorem 1]{KM} are satisfied, and we conclude
\begin{equation}\label{eqn:KM_thm_result}
 \mathrm{ch}_A\bigg(\left(-\frac{1}{\tau_U}\right)^{L_U(0)}u\otimes \left(-\frac{1}{\tau_V}\right)^{L_V(0)}v;-\frac{1}{\tau_U},-\frac{1}{\tau_V}\bigg) =\sum_X s_{A,X}\,\mathrm{ch}_{X}(u\otimes v; \tau_U,\tau_V)
\end{equation}
for all $u\in U$, $v\in V$, where $s_{A,X}$ here are $S$-matrix entries for $A$.

Now, the right side of \eqref{eqn:KM_thm_result} is a linear combination of products of characters of (irreducible) $U$-modules and characters of $V$-modules. The left side is
\begin{align*}
 \mathrm{ch}_A & \bigg(\left(-\frac{1}{\tau_U}\right)^{L_U(0)}u\otimes \left(-\frac{1}{\tau_V}\right)^{L_V(0)}v;-\frac{1}{\tau_U},-\frac{1}{\tau_V}\bigg)\nonumber\\
 &= \sum_{i\in I} \mathrm{ch}_{U_i}\bigg(\left(-\frac{1}{\tau_U}\right)^{L_U(0)} u;-\frac{1}{\tau_U}\bigg)\,\mathrm{ch}_{V_i}\bigg(\left(-\frac{1}{\tau_V}\right)^{L_V(0)} v; -\frac{1}{\tau_V}\bigg)\nonumber\\
 & =\sum_{i\in I}\sum_M s_{U_i,M}\,\mathrm{ch}_M(u;\tau_U)\,\mathrm{ch}_{V_i}\bigg(\left(-\frac{1}{\tau_V}\right)^{L_V(0)} v;-\frac{1}{\tau_V}\bigg)\nonumber\\
 & =\sum_M \mathrm{ch}_{M}(u,\tau_U)\sum_{i\in I} s_{U_i,M}\,\mathrm{ch}_{V_i}\bigg(\left(-\frac{1}{\tau_V}\right)^{L_V(0)} v;-\frac{1}{\tau_V}\bigg).
\end{align*}
Since the characters of distinct irreducible $U$-modules are linearly independent, we can identify the coefficients of $\mathrm{ch}_M(u,\tau_U)$ from each side of \eqref{eqn:KM_thm_result} for each irreducible $U$-module $M$. The result is that for each $M$,
\begin{equation*}
 \sum_{i\in I} s_{U_i,M}\, \mathrm{ch}_{V_i}\bigg(\left(-\frac{1}{\tau_V}\right)^{L_V(0)} v;-\frac{1}{\tau_V}\bigg)
\end{equation*}
is a linear combination of characters of $V$-modules. That is, each component of the vector $[s_{N, M}]^T\mathbf{x}$ is a linear combination of characters of $V$-modules, where $\mathbf{x}$ is the vector with a component for each isomorphism class of irreducible $U$-modules, with $N$-component $S(\mathrm{ch}_{V_i})(v,\tau_V)$ if $N\cong U_i$ for some $i\in I$ and $0$ otherwise. So because the $S$-matrix for $U$ is invertible, each $S(\mathrm{ch}_{V_i})$ is a linear combination of characters of $V$-modules.
\end{proof}

Now to prove that $V$ is strongly rational, it is enough to show that the category $\cC_V$ of grading-restricted generalized $V$-modules is semisimple, since we have already shown or assumed that $V$ is positive-energy, self-contragredient, and $C_2$-cofinite. In the next theorem, we will deduce semisimplicity of $\cC_V$ from rationality for $U\otimes V$, which we will prove by checking the criteria of \cite[Theorem 1.2(2)]{McR2}:

\begin{thm}\label{thm:coset_rat}
 Let $A$ be a strongly rational vertex operator algebra and let $U\subseteq A$ be a vertex subalgebra which is a strongly rational vertex operator algebra for some conformal vector $\omega_U\in U\cap A_{(2)}$ such that $L(1)\omega_U=0$. Assume also that:
 \begin{itemize}
  \item The subalgebra $U$ and its coset $V=C_A(U)$ form a dual pair, that is, $U=C_A(V)$.
  \item The coset $V$ is $C_2$-cofinite.
 \end{itemize}
Then the coset $V$ is also strongly rational.
\end{thm}
\begin{proof}
By \cite[Theorem 1.2(2)]{McR2}, $U\otimes V$ will be strongly rational if the following hold:
\begin{itemize}
 \item There is a $U\otimes V$-module homomorphism $\varepsilon: A\rightarrow U\otimes V$ such that $\varepsilon\circ\psi=\Id_{U\otimes V}$.
 \item The vertex operator algebra $U\otimes V$ is $C_2$-cofinite.
 \item The tensor category $\cC$ of grading-restricted generalized $U\otimes V$-modules is rigid.
 \item The categorical dimension $\dim_\cC A$ is non-zero.
\end{itemize}
We will check these conditions and then use rationality of $U\otimes V$ to prove that $V$ is also rational.

First, Theorem \ref{thm:gen_coset}(2) and (3) imply that $A\cong\bigoplus_{i\in I} U_i\otimes V_i$ as a $U\otimes V$-module, with $U=U_0$ for some $0\in I$ and $V\cong V_0$. So we can obtain $\varepsilon$ by projecting onto $U_0\otimes V_0$ with respect to this $U\otimes V$-module direct sum decomposition. Second, since $U\otimes V$ is the tensor product of $C_2$-cofinite vertex operator algebras, it is $C_2$-cofinite (see the proof of \cite[Theorem 5.3]{ABD}).
 
 Third, to prove that the category $\cC$ of grading-restricted generalized $U\otimes V$-modules is rigid, we first note that because $U\otimes V$ is positive-energy and $C_2$-cofinite, $\cC$ is a braided tensor category by \cite{Hu-C2}, with tensor product bifunctor $\tens_{U\otimes V}$. Also, since $U$ is strongly rational, Theorem \ref{thm:gen_coset}(4) implies that every module in $\cC$ is isomorphic to a direct sum of tensor products of $U$- and $V$-modules. Then Theorem 5.2 and Remark 5.3 of \cite{CKM2} show that there are natural isomorphisms
 \begin{equation*}
  (M_1\otimes W_1)\tens_{U\otimes V} (M_2\otimes W_2) \xrightarrow{\cong} (M_1\tens_U M_2)\otimes(W_1\tens_V W_2)
 \end{equation*}
in $\cC$, where $M_1$ and $M_2$ are $U$-modules, $W_1$ and $W_2$ are $V$-modules, and $\tens_U$ and $\tens_V$ are the tensor products in the categories $\mathcal{C}_U$ and $\mathcal{C}_V$, respectively. Moreover, under this identification, all structure isomorphisms in $\cC$ are (vector space) tensor products of structure isomorphisms in $\mathcal{C}_U$ and $\mathcal{C}_V$. Thus because $\mathcal{C}_U$ and $\mathcal{C}_V$ are both rigid ($\mathcal{C}_U$ because $U$ is strongly rational, $\mathcal{C}_V$ by Corollary \ref{cor:rigidity} and Theorem \ref{thm:KM_app}), $\cC$ is also rigid, with duals and (co)evaluations given by vector space tensor products of duals and (co)evaluations in $\mathcal{C}_U$ and $\mathcal{C}_V$. (In fact, $\cC$ is braided tensor equivalent to the Deligne product $\mathcal{C}_U\tens\mathcal{C}_V$ by \cite[Theorem 5.5]{CKM2}.)

Finally, we verify that $\dim_\cC A\neq 0$ using \cite[Theorems 5.10(2) and 5.12]{CKM2}. The assumptions of \cite[Theorem 5.10(2)]{CKM2} are satisfied in our setting due to the strong rationality of $U$ and $A$, the $C_2$-cofiniteness of $V$, Corollary \ref{cor:rigidity} and Theorem \ref{thm:KM_app} which establish rigidity for $\cC_V$, and Theorem \ref{thm:gen_coset} giving the decomposition of $A$ as a $U\otimes V$-module. We conclude that the full subcategories $\mathcal{U}_A\subseteq\mathcal{C}_U$ and $\mathcal{V}_A\subseteq\mathcal{C}_V$, whose objects are (isomorphic to) direct sums of the $U_i$, respectively the $V_i$, are ribbon subcategories and are braid-reversed tensor equivalent to each other. In particular, $\mathcal{U}_A$ and $\mathcal{V}_A$ are fusion categories and $A$ is an extension of $U\otimes V$ in the Deligne product category $\mathcal{U}_A\tens\mathcal{V}_A$, so that \cite[Theorem 5.12]{CKM2} implies that the categorical dimension of $A$ is a (non-zero) positive real number.

Now $U\otimes V$ is strongly rational by \cite[Theorem 1.2(2)]{McR2}. To show the same for $V$, we already know or have assumed that $V$ is positive-energy, self-contragredient, and $C_2$-cofinite, so by Lemma 3.6 and Proposition 3.7 of \cite{CM} (or \cite[Proposition 4.16]{McR}), it only remains to show that $\cC_V$ is semisimple. Thus suppose $f: W_1\rightarrow W_2$ is a surjection in $\mathcal{C}_V$; we need to show that there exists $s: W_2\rightarrow W_1$ such that $f\circ s=\Id_{W_2}$. Since tensoring in the category of $\CC$-vector spaces is exact,
\begin{equation*}
 \Id_U\otimes f: U\otimes W_1\rightarrow U\otimes W_2
\end{equation*}
is a surjection in $\cC$. Since $U\otimes V$ is rational,  there is a $U\otimes V$-module homomorphism
\begin{equation*}
 \sigma: U\otimes W_2\rightarrow U\otimes W_1
\end{equation*}
such that $(\Id_U\otimes f)\circ\sigma =\Id_{U\otimes W_2}$. We choose $\vac'\in U'$ such that $\langle\vac',\vac\rangle=1$ and then define the $V$-module homomorphism $s: W_2\rightarrow W_1$ to be the composition
\begin{equation*}
 W_2\xrightarrow{\cong} \vac\otimes W_2\hookrightarrow U\otimes W_2\xrightarrow{\sigma} U\otimes W_1\xrightarrow{\vac'\otimes\Id_{W_1}} W_1.
\end{equation*}
So for any $w_2\in W_2$, we get
\begin{align*}
 (f\circ s)(w_2) & = \left[f\circ(\vac'\otimes\Id_{W_1})\circ\sigma\right](\vac\otimes w_2)\nonumber\\ &=\left[(\vac'\otimes\Id_{W_2})\circ(\Id_U\otimes f)\circ\sigma\right](\vac\otimes w_2)\nonumber\\ &=(\vac'\otimes\Id_{W_2})(\vac\otimes w_2) = w_2,
\end{align*}
as required.
\end{proof}

Theorem \ref{thm:coset_rat} yields a third rationality proof for certain $ADE$-type principal $W$-algebras (besides \cite{Ar-rat} and Theorem \ref{thm:ex_W_alg_rat}) via the Arakawa-Creutzig-Linshaw coset realization \cite{ACL}:
\begin{exam}
 Take $A=L_k(\mathfrak{g})\otimes L_1(\mathfrak{g})$ where $\mathfrak{g}$ is a simply-laced simple Lie algebra and $k\in\NN$, and take $U$ to be the diagonal subalgebra $L_{k+1}(\mathfrak{g})$. By \cite[Main Theorem 1]{ACL}, the coset $V$ is isomorphic to the simple principal $W$-algebra $\cW_{\ell}(\mathfrak{g})$ at level $\ell=-h^\vee+\frac{k+h^\vee}{k+h^\vee+1}$ ($h^\vee$ is the dual Coxeter number of $\mathfrak{g}$); moreover, $L_{k+1}(\mathfrak{g})$ and $\cW_\ell(\mathfrak{g})$ form a dual pair in $L_k(\mathfrak{g})\otimes L_1(\mathfrak{g})$. These $W$-algebras are $C_2$-cofinite \cite{Ar-C2}, so Theorem \ref{thm:coset_rat} implies that $\cW_\ell(\mathfrak{g})$ for such $\ell$ are strongly rational. In the case $\mathfrak{g}=\mathfrak{sl}_2$, the coset $\cW_\ell(\mathfrak{g})$ is the simple rational (and unitary) Virasoro vertex operator algebra at central charge $1-\frac{6}{(k+2)(k+3)}$ \cite{GKO1, GKO2}.
\end{exam}

\appendix

\section{The assumption in Theorem \ref{thm:W_alg_rat} for odd nilpotent orbits}\label{app:nilpotent_check}

In this appendix, we check the assumption of Theorem \ref{thm:W_alg_rat} for all nilpotent orbits $\mathbb{O}_q$ in \cite[Tables 2--10]{Ar-C2}. By Corollary \ref{cor:good_even_grad}, it is enough to consider nilpotent elements which do not admit good even gradings, and it is easy to determine which nilpotent elements these are from the tables and results in, for example, \cite{BC2, EK}.

\subsection{Classical types}

In classical types, we need to check odd nilpotent orbits in types $B$, $C$, and $D$ only, since every nilpotent element in type $A$ admits a good even grading (see \cite[Proposition 4.2]{EK}). For our discussion of nilpotent elements in classical types here, we will use \cite[Sections 1 and 3]{Ja} as a reference.

Let $\mathfrak{g}=\mathfrak{so}_n$ or $\mathfrak{sp}_n$ (with $n$ even in the symplectic case). We can identify any nilpotent element $f\in\mathfrak{g}$ with an $n\times n$ matrix acting on the defining representation $\CC^n$ of $\mathfrak{g}$, and we can identify conjugacy classes of nilpotent elements by their Jordan canonical forms. Thus nilpotent orbits in $\mathfrak{g}$ are identified with certain partitions of $n$: for $\mathfrak{g}=\mathfrak{so}_n$, even parts of these partitions must have even multiplicity, while for $\mathfrak{sp}_n$, odd parts must have even multiplicity (see for example \cite[Theorem 1.6]{Ja}). We will express the partition of a nilpotent element $f\in\mathfrak{g}$ as follows:
\begin{equation*}
 n=p_1+p_2+\ldots+p_{2r}+q_1+q_2+\ldots+q_s,
\end{equation*}
where $p_{2i-1}=p_{2i}$ for $1\leq i\leq r$; for $\mathfrak{g}=\mathfrak{so}_n$, we assume each $p_i$ is even and each $q_j$ is odd, while for $\mathfrak{sp}_n$ each $p_i$ is odd and each $q_j$ is even.

Now $\CC^n$ has a basis with respect to which the nilpotent element $f\in\mathfrak{g}$ has the matrix
\begin{equation*}
 f=\left[\begin{array}{cccccc}
        J_{p_1} & \cdots & 0 & 0 & \cdots & 0\\
        \vdots & \ddots & \vdots & \vdots & \ddots & \vdots\\
        0 & \cdots & J_{p_{2r}} & 0 & \cdots & 0\\
        0 & \cdots & 0 & J_{q_1} & \cdots & 0\\
        \vdots & \ddots & \vdots & \vdots & \ddots & \vdots\\
        0 & \cdots & 0 & 0 & \cdots & J_{q_s}\\
         \end{array}
\right],
\end{equation*}
where $J_m$ for $m\in\ZZ_+$ denotes the \textit{lower}-triangular $m\times m$ indecomposable Jordan block. By Theorems 1 and 2 in \cite[Subsection 1.11]{Ja}, we may take this basis of $\CC^n$ such that the Gram matrix of the bilinear form preserved by $\mathfrak{g}$ is the following:
\begin{equation*}
 S=\left[\begin{array}{cccccc}
        \til{S}_{1} & \cdots & 0 & 0 & \cdots & 0\\
        \vdots & \ddots & \vdots & \vdots & \ddots & \vdots\\
        0 & \cdots & \til{S}_{r} & 0 & \cdots & 0\\
        0 & \cdots & 0 & S_{q_1} & \cdots & 0\\
        \vdots & \ddots & \vdots & \vdots & \ddots & \vdots\\
        0 & \cdots & 0 & 0 & \cdots & S_{q_s}\\
         \end{array}
\right]
\end{equation*}
where for $m\in\ZZ_+$, $S_m$ is the $m\times m$ matrix
\begin{equation*}
 S_m=\left[\begin{array}{rrrrrr}
        0 & \cdots & 0 & 0 & 0 & 1\\
        0 & \cdots & 0 & 0 & -1 & 0\\
        0 & \cdots & 0 & 1 & 0 & 0\\
        0 & \cdots & -1 & 0 & 0 & 0\\
        \vdots & \reflectbox{$\ddots$} & \vdots & \vdots & \vdots & \vdots\\
        \pm 1 & \cdots & 0 & 0 & 0 & 0\\
         \end{array}
\right],
\end{equation*}
and for $1\leq i\leq r$,
\begin{equation*}
 \til{S}_i=\left[\begin{array}{cc}
                  0 & S_{p_i}\\
                  -S_{p_i} & 0\\
                 \end{array}
\right].
\end{equation*}
We write matrices $X\in\mathfrak{gl}_n$ in block form,
\begin{equation*}
 X=\left[\begin{array}{cc}
          A & B\\
          C & D\\
         \end{array}
\right],
\end{equation*}
where $A$ is $(p_1+\ldots+p_{2r})\times(p_1+\ldots+p_{2r})$ and $D$ is $(q_1+\ldots+q_s)\times(q_1+\ldots+q_s)$. Then the block diagonal form of the Gram matrix $S$ implies that $X\in\mathfrak{g}$, that is, $SX=-X^T S$, if and only if all three of
\begin{equation*}
 \left[\begin{array}{cc}
          A & 0\\
          0 & 0\\
         \end{array}
\right],\left[\begin{array}{cc}
          0 & 0\\
          0 & D\\
         \end{array}
\right],\left[\begin{array}{cc}
          0 & B\\
          C & 0\\
         \end{array}
\right]
\end{equation*}
are matrices in $\mathfrak{g}$.

Now as discussed in \cite[Sections 3.3 and 3.4]{Ja}, the Dynkin grading of $\mathfrak{g}$ can be obtained by restriction from the grading of $\mathfrak{gl}_n$ induced by a suitable grading of $\CC^n$ associated to $f$. From the definition of this grading of $\CC^n$ (given in \cite[Section 3.3]{Ja}), it is clear that diagonal matrices in $\mathfrak{g}$ are contained in $\mathfrak{g}_0$, and that 
\begin{equation}\label{eqn:X_decomp}
X= \left[\begin{array}{cc}
          A & 0\\
          0 & D\\
         \end{array}
\right]+\left[\begin{array}{cc}
          0 & B\\
          C & 0\\
         \end{array}
         \right]
\end{equation}
gives the decomposition of $X\in\mathfrak{g}$ as the sum of a matrix in $\bigoplus_{j\in\ZZ} \mathfrak{g}_j$ and a matrix in $\bigoplus_{j\in\frac{1}{2}+\ZZ}\mathfrak{g}_j$, respectively. In particular, the nilpotent element $f$ is even if and only if either $r=0$ or $s=0$, that is, all parts in the partition of $f$ have the same parity.

Now we want to show that if $f$ is an odd nilpotent element, then $f$ satisfies the conditions of Theorem \ref{thm:W_alg_rat}. Thus we may assume $r,s\geq 1$, so that in particular $r\geq 1$. We then define $v\in\mathfrak{gl}_n$ in terms of its blocks: $B=C=D=0$, while
\begin{equation*}
 A=\frac{1}{2}\left[\begin{array}{ccccc}
                     I_{p_1} & 0 & \cdots & 0 & 0\\
                     0 & -I_{p_2} & \cdots & 0 & 0\\
                     \vdots & \vdots & \ddots & \vdots & \vdots\\
                     0 & 0 & \cdots & I_{p_{2r-1}} & 0\\
                     0 & 0 & \cdots & 0 & -I_{p_{2r}}\\
                    \end{array}
\right],
\end{equation*}
where $I_m$ for $m\in\ZZ_+$ is the $m\times m$ identity matrix. (Recall that $p_{2i-1}=p_{2i}$ for $1\leq i\leq r$.) It is clear that $[f,v]=0$, and $v\in\mathfrak{g}$ as well since for any $m\in\ZZ_+$,
\begin{equation*}
 \left[\begin{array}{cc}
        0 & S_m\\
        -S_m & 0\\
       \end{array}
\right]\left[\begin{array}{cc}
        I_m & 0\\
        0 & -I_m\\
       \end{array}
\right] =\left[\begin{array}{cc}
                0 & -S_m\\
                -S_m & 0\\ 
               \end{array}
\right]=-\left[\begin{array}{cc}
                I_m & 0\\
                0 & -I_m\\
               \end{array}
\right]^T\left[\begin{array}{cc}
                0 & S_m\\
                -S_m & 0\\
               \end{array}
\right].
\end{equation*}
Since $v$ is diagonal, $v\in\mathfrak{g}_0^f$. Also, $\mathrm{ad}(v)$ is semisimple on $\mathfrak{gl}_n$: block matrices of the forms
\begin{equation*}
 \left[\begin{array}{cc}
          A & 0\\
          0 & 0\\
         \end{array}
\right], \left[\begin{array}{cc}
          0 & B\\
          0 & 0\\
         \end{array}
\right], \left[\begin{array}{cc}
          0 & 0\\
          C &0 \\
         \end{array}
\right], \left[\begin{array}{cc}
          0 & 0\\
          0 & D\\
         \end{array}
\right]
\end{equation*}
are sums of $\mathrm{ad}(v)$-eigenvectors with eigenvalues contained in $\lbrace 1,0,-1\rbrace$, $\left\lbrace\frac{1}{2},-\frac{1}{2}\right\rbrace$, $\left\lbrace\frac{1}{2},-\frac{1}{2}\right\rbrace$, and $\lbrace 0\rbrace$, respectively. Thus $\mathrm{ad}(v)$ is diagonalizable on $\mathfrak{g}$ and on $\mathfrak{g}^f$, and the decomposition \eqref{eqn:X_decomp} shows that the $\mathrm{ad}(v)$-eigenvalues on $\mathfrak{g}_{-j}^f$, for $j\geq 0$, are contained in $-j-1+\NN$. Thus $v\in\mathfrak{g}_0^f$ satisfies the conditions in Theorem \ref{thm:W_alg_rat}.

\subsection{Exceptional types}

In exceptional types, there are fifteen nilpotent orbits to check: two in type $G_2$, three for $F_4$, two for $E_6$, two for $E_7$, and six for $E_8$. For each case, we use $f$ to denote the nilpotent element in $\mathfrak{g}$ under consideration, and we embed $f$ into an $\mathfrak{sl}_2$-triple $\lbrace e,h,f\rbrace$. Let $\mathfrak{h}$ be a Cartan subalgebra of $\mathfrak{g}$ containing $h$, and let $\Delta_+$ be a choice of positive roots for $\mathfrak{g}$ such that $\alpha(h)\geq 0$ for $\alpha\in\Delta_+$. For any $\alpha\in\Delta_+$, let $e_\alpha$ and $f_\alpha$ span the root spaces of $\mathfrak{g}$ associated to $\alpha$ and $-\alpha$, respectively. Then the Dynkin grading $\mathfrak{g}=\bigoplus_{j\in\frac{1}{2}\ZZ} \mathfrak{g}_j$ is given by
\begin{equation*}
 \mathfrak{g}_0=\mathfrak{h}\oplus\bigoplus_{\alpha\in\Delta_+,\,\alpha(h)=0} (\CC e_\alpha\oplus\CC f_\alpha)
\end{equation*}
and
\begin{equation*}
 \mathfrak{g}_j=\bigoplus_{\alpha\in\Delta_+,\,\alpha(h)=2j}\CC e_\alpha,\qquad\mathfrak{g}_{-j}=\bigoplus_{\alpha\in\Delta_+,\,\alpha(h)=2j}\CC f_\alpha
\end{equation*}
for $j>0$. Let $\lbrace\alpha_i\rbrace_{i=1}^{\mathrm{rank}(\mathfrak{g})}$ be the simple roots of $\mathfrak{g}$.

We will identify elements $v\in\mathfrak{h}$ with weighted Dynkin diagrams that indicate how the simple roots of $\mathfrak{g}$ act on $v$. For example, in type $E_6$, the notation
\begin{equation*}
 v\longleftrightarrow\begin{array}{rrrrr}
                      & & c_2 & & \\
                      c_1 & c_3 & c_4 & c_5 & c_6
                     \end{array}
,
\end{equation*}
where $c_i\in\CC$, means that $v$ is defined by $\alpha_i(v)=c_i$ for $1\leq i\leq 6$. We will also use weighted Dynkin diagrams to denote the root $\alpha$ in the subscripts of $e_\alpha$ and $f_\alpha$ for $\alpha\in\Delta_+$. For example, for the highest root $\theta=\alpha_1+2\alpha_2+2\alpha_3+3\alpha_4+2\alpha_5+\alpha_6$ in type $E_6$, we denote
\begin{equation*}
 f_\theta=f_{\substack{2\\ 12321}}.
\end{equation*}
In types $G_2$ and $F_4$, our convention will be to orient weighted Dynkin diagrams so that the arrows (which are suppressed in the notation) point towards the right.

We now discuss how to find $v\in\mathfrak{h}_0^f$ satisfying the conditions of Theorem \ref{thm:W_alg_rat}. First, we may take $f$ to be a sum of negative root vectors, $f=\sum_{k=1}^K f_{\beta_k}$ where $\beta_k\in\Delta_+$ satisfy $\beta_k(h)=2$; the tables in \cite[Appendix A]{dG} provide explicit such nilpotent orbit representatives. Given the expansion of each $\beta_k$ as a linear combination of simple roots, we can then determine $\mathfrak{h}_0^f$ by solving the system of linear equations $\lbrace\beta_k(v)=0\rbrace_{k=1}^K$ for $\alpha_i(v)$, $1\leq i\leq\mathrm{rank}(\mathfrak{g})$. It is usually easy to identify a solution $v$ whose eigenvalues on $\mathfrak{g}^f$ are likely to satisfy the conditions of Theorem \ref{thm:W_alg_rat}. We then determine the eigenvalues of $\mathrm{ad}(v)$ for this solution $v$ on $\mathfrak{g}_0\oplus\mathfrak{g}_{-1/2}$. These are easy to find from the weighted Dynkin diagram of $v$, provided we know the weighted Dynkin diagrams of all $\alpha\in\Delta_+$ such that $\alpha(h)=0,1$. Such $\alpha$ in turn are easy to find using the weighted Dynkin diagram of $h$ and the lists in \cite[Appendix B]{dG} of weighted Dynkin diagrams for all positive roots in exceptional types.

If the eigenvalues of $\mathrm{ad}(v)$ on $\mathfrak{g}_0$ and $\mathfrak{g}_{-1/2}$ are contained in $\lbrace 1,0,-1\rbrace$ and $\left\lbrace \frac{3}{2},\frac{1}{2},-\frac{1}{2},-\frac{3}{2}\right\rbrace$, respectively, then we are done because  $\mathrm{ad}(f)$ induces $\mathfrak{g}_0^f$-module isomorphisms $\mathfrak{g}_0\cong\bigoplus_{j\in\NN} \mathfrak{g}_{-j}^f$ and $\mathfrak{g}_{-1/2}\cong\bigoplus_{j\in\frac{1}{2}+\NN}\mathfrak{g}_{-j}^f$ (see \cite[Theorem 1.4]{EK}). In some cases, we encounter $\mathrm{ad}(v)$-eigenvalues $\pm 2$ on $\mathfrak{g}_0$ or $\pm\frac{5}{2}$ on $\mathfrak{g}_{-1/2}$. Then we check that $\mathrm{ad}(f)$ acts injectively on these $\mathrm{ad}(v)$-eigenspaces in $\mathfrak{g}_0\oplus\mathfrak{g}_{-1/2}$, which shows that the corresponding eigenspaces in $\mathfrak{g}^f$ are contained in $\bigoplus_{j\geq 1} \mathfrak{g}_{-j}^f$. 

Now, for the fifteen nilpotent orbits in exceptional types which appear in the tables of \cite{Ar-C2} and do not admit good even gradings, we list a choice of nilpotent orbit representative $f$ coming from the tables of \cite[Appendix A]{dG}, the weighted Dynkin diagram for $h$ in the $\mathfrak{sl}_2$-triple associated to $f$, and a choice of $v\in\mathfrak{h}_0^f$ that satisfies the conditions of Theorem \ref{thm:W_alg_rat}. For a more detailed analysis of each case, see Appendix B.2 in the expanded version of this paper at \texttt{arXiv:2108.01898v2}.

\textbf{Type $G_2$.} For the nilpotent orbit $A_1$ (the minimal nilpotent orbit, containing long root vectors), which is $\mathbb{O}_3$:
\begin{equation*}
f=f_{23},\qquad h\longleftrightarrow\begin{array}{rr}
                      1 & 0\\
                     \end{array}
,\qquad  v\longleftrightarrow\begin{array}{rr}
                      -\frac{3}{2} & 1\\
                     \end{array}
.
\end{equation*}

\noindent For $\til{A}_1$ (containing short root vectors), which is $\mathbb{O}_2$:
\begin{equation*}
  f=f_{12},\qquad h \longleftrightarrow\begin{array}{rr}
                       0 & 1\\
                      \end{array},
\qquad v\longleftrightarrow \begin{array}{rr}
                       1 & -\frac{1}{2} \\
                      \end{array}.
\end{equation*}

\textbf{Type $F_4$.} For the nilpotent orbit $A_1$ (the minimal nilpotent orbit), which is $\mathbb{O}_2$:
\begin{equation*}
 f=f_{2342},\qquad h\longleftrightarrow\begin{array}{rrrr}
                      1 & 0 & 0 & 0\\
                     \end{array}
,\qquad v\longleftrightarrow\begin{array}{rrrr}
                                     -\frac{3}{2} & 1 & 0 & 0\\
                                    \end{array}.
\end{equation*}

\noindent For $\til{A}_2+A_1$, which is $\mathbb{O}_3$:
\begin{equation*}
f =f_{0111}+f_{1121}+f_{1220},\qquad h\longleftrightarrow \begin{array}{rrrr}
                       0 & 1 & 0 & 1\\
                      \end{array}
,\qquad v\longleftrightarrow \begin{array}{rrrr}
                       -1 & -\frac{1}{2} & 1 & -\frac{1}{2}
                      \end{array}.
\end{equation*}

\noindent For $A_2+\til{A}_1$, which is $\mathbb{O}_4$:
\begin{equation*}
 f=f_{1121}+f_{0122}+f_{1220}, \qquad h\longleftrightarrow \begin{array}{rrrr}
                       0 & 0 & 1 & 0\\
                      \end{array}
,\qquad v\longleftrightarrow\begin{array}{rrrr}
                                         1 & 1 & -\frac{3}{2} & 1\\
                                        \end{array}.
\end{equation*}

\textbf{Type $E_6$.} For the nilpotent orbit $3A_1$, which is $\mathbb{O}_2$:
\begin{equation*}
f=f_{\substack{1\\ 01210\\}}+f_{\substack{1\\12211\\}}+f_{\substack{1\\11221\\}},\qquad h\longleftrightarrow\begin{array}{rrrrr}
                  & & 0 & & \\
                  0 & 0 & 1 & 0 & 0\\
                 \end{array}
,\qquad v\longleftrightarrow \begin{array}{ccccc}
                                          && 1 &&\\
                                          0 & 0 & -\frac{1}{2} & 0 & 0\\
                                         \end{array}.
\end{equation*}

\noindent For $2A_2+A_1$, which is $\mathbb{O}_3$:
\begin{align*}
& f=f_{\substack{1\\01210\\}}+f_{\substack{0\\11100\\}}+f_{\substack{1\\00111\\}}+f_{\substack{0\\01111\\}}+f_{\substack{1\\11110\\}},\nonumber\\ & h\longleftrightarrow\begin{array}{ccccc}
                      && 0 &&\\
                      1 & 0 & 1 & 0 & 1\\
                     \end{array}
,\qquad
v\longleftrightarrow\begin{array}{ccccc}
                      && 1 &&\\
                      -\frac{1}{2} & 1 & -\frac{1}{2} & -1 & \frac{1}{2}
                     \end{array}.
\end{align*}
\noindent Note that for $E_6$, the nilpotent orbit $A_4+A_1$, which is $\mathbb{O}_5$, is also odd. But $A_4+A_1$ admits a good even grading (see Table $E_6$ in \cite[Section 7]{EK}).

\textbf{Type $E_7$.} For the nilpotent orbit $4A_1$, which is $\mathbb{O}_2$:
\begin{align*}
&f=f_{\substack{1\hphantom{0}\\012211\\}}+f_{\substack{1\hphantom{0}\\122111\\}}+f_{\substack{1\hphantom{0}\\112221\\}}+f_{\substack{2\hphantom{0}\\123210\\}}\nonumber\\
&h\longleftrightarrow\begin{array}{cccccc}
                      && 1 &&&\\
                      0 & 0 & 0 & 0 & 0 & 1\\
                     \end{array}
,\qquad  v\longleftrightarrow\begin{array}{cccccc}
                      && -\frac{1}{2} &&&\\
                      0 & 1 & -1 & 1 & 0 & -\frac{1}{2}\\
                     \end{array}
\end{align*}

\noindent For $2A_2+A_1$, which is $\mathbb{O}_3$:
\begin{align*}
&f=f_{\substack{0\hphantom{0}\\111111\\}}+f_{\substack{1\hphantom{0}\\012110\\}}+f_{\substack{1\hphantom{0}\\011111\\}}+f_{\substack{1\hphantom{0}\\122100\\}}+f_{\substack{1\hphantom{0}\\112210\\}},\nonumber\\
& h\longleftrightarrow\begin{array}{cccccc}
                      && 0 &&&\\
                      0 & 1 & 0 & 0 & 1 & 0\\
                     \end{array}
,\qquad v\longleftrightarrow\begin{array}{cccccc}
                      && 1 &&&\\
                      1 & -\frac{1}{2} & 0 & -1 & \frac{1}{2} & 0\\
                     \end{array}.
\end{align*} 

\textbf{Type $E_8$.} For the nilpotent orbit $4A_1$, which is $\mathbb{O}_2$:
\begin{align*}
&f =f_{\substack{2\hphantom{0}\hphantom{0}\\1233221\\}}+f_{\substack{2\hphantom{0}\hphantom{0}\\2343210\\}}+f_{\substack{2\hphantom{0}\hphantom{0}\\1343211\\}}+f_{\substack{2\hphantom{0}\hphantom{0}\\1243321\\}},\nonumber\\
& h\longleftrightarrow\begin{array}{ccccccc}
                      && 1 &&&&\\
                      0 & 0 & 0 & 0 & 0 & 0 & 0\\
                     \end{array}
,\qquad v\longleftrightarrow\begin{array}{ccccccc}
                      && -\frac{1}{2} &&&&\\
                      0 & 1 & -1 & 0 & 1 & 0 & 0\\
                     \end{array}.
\end{align*}

\noindent For $2A_2+2A_1$, which is $\mathbb{O}_3$:
\begin{align*}
&f=f_{\substack{2\hphantom{0}\hphantom{0}\\1232100\\}}+f_{\substack{1\hphantom{0}\hphantom{0}\\1232110\\}}+f_{\substack{1\hphantom{0}\hphantom{0}\\1222210\\}}+f_{\substack{1\hphantom{0}\hphantom{0}\\1222111\\}}+f_{\substack{1\hphantom{0}\hphantom{0}\\1122211\\}}+f_{\substack{1\hphantom{0}\hphantom{0}\\0122221\\}},\nonumber\\
& h\longleftrightarrow\begin{array}{ccccccc}
                      && 0 &&&&\\
                      0 & 0 & 0 & 1 & 0 & 0 & 0\\
                     \end{array}
,\qquad v\longleftrightarrow\begin{array}{ccccccc}
                      && 1 &&&&\\
                      1 & 0 & 0 & -\frac{3}{2} & 0 & 1 & 0\\
                     \end{array}.
\end{align*}

\noindent For $2A_3$, which is $\mathbb{O}_4$:
\begin{align*}
&f=f_{\substack{1\hphantom{0}\hphantom{0}\\1221000\\}}+f_{\substack{1\hphantom{0}\hphantom{0}\\1121100\\}}+f_{\substack{1\hphantom{0}\hphantom{0}\\1111110\\}}+f_{\substack{0\hphantom{0}\hphantom{0}\\1111111\\}}+f_{\substack{1\hphantom{0}\hphantom{0}\\0122110\\}}+f_{\substack{1\hphantom{0}\hphantom{0}\\0122211\\}},\nonumber\\
&h\longleftrightarrow\begin{array}{ccccccc}
                      && 0 &&&&\\
                      1 & 0 & 0 & 1 & 0 & 0 & 0\\
                     \end{array}
,\qquad v\longleftrightarrow\begin{array}{ccccccc}
                      && 0 &&&&\\
                      -\frac{1}{2} & 0 & 1 & -\frac{3}{2} & 0 & 1 & 0\\
                     \end{array}.
\end{align*} 

\noindent For $A_4+A_3$, which is $\mathbb{O}_5$:
\begin{align*}
&f=f_{\substack{1\hphantom{0}\hphantom{0}\\0011111\\}}+f_{\substack{0\hphantom{0}\hphantom{0}\\1111110\\}}+f_{\substack{1\hphantom{0}\hphantom{0}\\0111110\\}}+f_{\substack{0\hphantom{0}\hphantom{0}\\0111111\\}}+f_{\substack{1\hphantom{0}\hphantom{0}\\1221000\\}}+f_{\substack{1\hphantom{0}\hphantom{0}\\1121100}}+f_{\substack{1\hphantom{0}\hphantom{0}\\0122100\\}},\nonumber\\
& h\longleftrightarrow\begin{array}{ccccccc}
                      && 0 &&&&\\
                      0 & 0 & 1 & 0 & 0 & 1 & 0\\
                     \end{array},\qquad v\longleftrightarrow\begin{array}{ccccccc}
                      && 1 &&&&\\
                      1 & 1 & -\frac{5}{2} & 1 & 1 & -\frac{3}{2} & 1\\
                     \end{array}.
\end{align*}

\noindent For $A_6+A_1$, which is $\mathbb{O}_7$:
\begin{align*}
& f=f_{\substack{1\hphantom{0}\hphantom{0}\\1110000\\}}+f_{\substack{0\hphantom{0}\hphantom{0}\\1111000\\}}+f_{\substack{1\hphantom{0}\hphantom{0}\\0121000\\}}+f_{\substack{1\hphantom{0}\hphantom{0}\\0111100\\}}+f_{\substack{1\hphantom{0}\hphantom{0}\\0011110\\}}+f_{\substack{0\hphantom{0}\hphantom{0}\\0111110}}+f_{\substack{0\hphantom{0}\hphantom{0}\\0011111\\}},\nonumber\\
& h\longleftrightarrow\begin{array}{ccccccc}
                      && 0 &&&&\\
                      1 & 0 & 1 & 0 & 1 & 0 & 0\\
                     \end{array},\qquad v\longleftrightarrow\begin{array}{ccccccc}
                      && 1 &&&&\\
                      -\frac{1}{2} & 1 & -\frac{3}{2} & 1 & -\frac{3}{2} & 1 & 1\\
                     \end{array}.
\end{align*}

\noindent For $A_7$, which is $\mathbb{O}_8$:
\begin{align*}
&f=f_{\substack{0\hphantom{0}\hphantom{0}\\0000110\\}}+f_{\substack{1\hphantom{0}\hphantom{0}\\1110000\\}}+f_{\substack{0\hphantom{0}\hphantom{0}\\1111000\\}}+f_{\substack{1\hphantom{0}\hphantom{0}\\0011100\\}}+f_{\substack{0\hphantom{0}\hphantom{0}\\0111100\\}}+f_{\substack{0\hphantom{0}\hphantom{0}\\0001111}}+f_{\substack{1\hphantom{0}\hphantom{0}\\0121000\\}},\nonumber\\
& h\longleftrightarrow\begin{array}{ccccccc}
                      && 0 &&&&\\
                      1 & 0 & 1 & 0 & 1 & 1 & 0\\
                     \end{array},\qquad v \longleftrightarrow\begin{array}{ccccccc}
                     && 1 &&&&\\
                     -\frac{1}{2} & 1 & -\frac{3}{2} & 1 & -\frac{1}{2} & \frac{1}{2} & -1\\
                    \end{array}.
\end{align*}

\end{document}